\documentclass[11pt,a4paper,reqno,svgnames]{amsart}

  \usepackage{etoolbox}

\newtoggle{arxiv}   
 \toggletrue{arxiv}    

  \usepackage{mathtools}

 \usepackage{enumitem}

 \usepackage{amsmath,amsthm,amsfonts,amssymb} 
 \usepackage{amsxtra}
 
\usepackage{mathrsfs}
 \usepackage{bm, stmaryrd}
  \usepackage{scalefnt}
\usepackage{latexsym}

 \usepackage{tikz}
 \usepackage{genyoungtabtikz}
 
\usepackage{fullpage}

\makeatletter
\def\l@subsection{\@tocline{2}{0pt}{2.0pc}{5pc}{}}
\makeatother

\usepackage{xcolor}
\usepackage[plainpages=false, pdfpagelabels,  colorlinks=true, pdfstartview=FitV, linkcolor=DarkBlue, citecolor=DarkRed, urlcolor=blue]{hyperref}

\newcommand{\myref}[2]{\hyperref[#2]{#1~\ref*{#2}}}
\newcommand{\myrefnospace}[2]{\hyperref[#2]{#1\ref*{#2}}}

\usepackage{comment}   

\newcommand{\smashr}{\smashoperator[r]}   

\def\perm{{\rm perm}}
\def\el{l}

\def\mixed{^{\mathrm m}}

\numberwithin{equation}{section}

 \DeclareMathOperator{\A}{A}

 \DeclareMathOperator{\Z}{\mathbb{Z}}
 
\DeclareMathOperator{\End}{End}

\DeclareMathOperator{\Res}{Res}
\DeclareMathOperator{\Span}{Span}

\DeclareMathOperator{\ch}{char}

\newcommand{\la}{\lambda}

\newcommand{\ZZ}{{\mathbb{Z}}}

\newcommand{\F}{\mathbb F}
\newcommand{\FF}{\mathbb F}

\renewcommand{\ge}{\geqslant}

\renewcommand{\ge}{\geqslant}
\renewcommand{\geq}{\geqslant}
\renewcommand{\le}{\leqslant}
\renewcommand{\leq}{\leqslant}
\renewcommand{\unrhd}{\trianglerighteqslant}
\renewcommand{\unlhd}{\trianglelefteqslant}
\renewcommand{\succeq}{\succcurlyeq}

 \newcommand{\Std}{{\rm Std}}
 
 \newcommand{\tensor}[1]{^{\otimes #1}}

\theoremstyle{plain}
\newtheorem{thm}{Theorem}[section]
\newtheorem{cor}[thm]{Corollary}

\newtheorem{lem}[thm]{Lemma}
\newtheorem{prop}[thm]{Proposition}

 \theoremstyle{remark}

\newtheorem{rmk}[thm]{Remark}
\newtheorem{notation}[thm]{Notation}
\newtheorem{rem}[thm]{Remark}

\newtheorem*{Acknowledgements*}{Acknowledgements}

\theoremstyle{definition}
\newtheorem{defn}[thm]{Definition}
\newtheorem{eg}[thm]{Example}

\newtheorem*{ack}{Acknowledgements}

 \newcommand{\Q}{\mathbb Q}
\newcommand\C{\mathbb C}
\newcommand\RR{\mathbb R}
  
\renewcommand{\v}{v}  
\renewcommand{\c}{b}  
\renewcommand{\b}{b}  
\newcommand{\M}{m}  
\newcommand{\nn}{{\widetilde m}}  

\newcommand{\wbq}{H}  

\newcommand{\wbs}{\mathcal B}  

\newcommand{\dd}[3]{d_{{#1} \to {#2}}\power {#3}}
\newcommand{\uu}[3]{u_{{#1} \to {#2}}\power {#3}}
\newcommand{\bb}[3]{b_{{#1} \to {#2}}\power {#3}}
\newcommand{\vv}[3]{v_{{#1} \to {#2}}\power {#3}}

\newcommand{\stu}{\mathsf u}
\newcommand{\stv}{\mathsf v}

\newcommand{\mft}{\mathsf t}
\newcommand{\mfs}{\mathsf s}
\newcommand{\mfu}{\mathsf u}
\newcommand{\mfv}{\mathsf v}
\newcommand{\stt}{{\mathsf t}}
\newcommand{\sts}{{\mathsf s}}
\renewcommand{\k}{{r}}
\newcommand{\mfz}{\mathsf z}

\newcommand{ \Ind}{{\rm Ind}}
\renewcommand{\Res}{{\rm Res}}
\newcommand{\suchthat}{\;\ifnum\currentgrouptype=16 \middle\fi|\;} 
\newcommand\z[2]{z_{#1}^{#2}}   

\newcommand{\overbar}[1]{\mkern 1.5mu\overline{\mkern-1.5mu#1\mkern-1.5mu}\mkern 1.5mu}
\def\Kbar{\overbar {\mathbb{K}}}

\newcommand{\coldomeq}{\unrhd_{\rm col}}
\newcommand{\coldom}{\rhd_{\rm col}}

\def\power #1{^{(#1)}}
\def\powerr #1{(#1)}
\def\leftbrace{\left\{}
\def\rightbrace{\right\}}
\def\ignore#1{\relax}

\def\deltabold{{\bm \delta}}
\def\bolddelta{\deltabold}
\def\p #1{ \bm {#1}}
\def\pbar #1{\overline{\p {#1}}}
\def\qbold{{\bm q}}

\newcommand\boldq{\qbold}

\DeclareMathOperator{\im}{im}

\DeclareMathOperator{\id}{id}

\DeclareMathOperator{\node}{node}

\DeclareMathOperator{\spn}{span}

\DeclareMathOperator{\rad}{rad}

\DeclareMathOperator{\sgn}{sign}
\newcommand{\Stab}{{\rm Stab}}
\DeclareMathOperator{\GL}{\sf GL}
\DeclareMathOperator{\SL}{\sf SL}
\DeclareMathOperator{\Orth}{\sf O}
\DeclareMathOperator{\Symp}{\sf Sp}

\begin{document}

\title[The second fundamental theorem  of invariant theory for classical groups]{The cellular second fundamental theorem  of \\ invariant theory for classical groups
 }

\author[C. Bowman]{Christopher Bowman}
\address{School of Mathematics, Statistics and Actuarial Science, University of Kent, Canterbury
UK}
\email{C.D.Bowman@kent.ac.uk}

\author[J. Enyang]{John Enyang}
\address{Department of Mathematics, City University London,   London, United Kingdom}
\email{john.enyang.1@city.ac.uk}

\author[F.M.Goodman]{Frederick M. Goodman}
\address{Department of Mathematics, University of Iowa, Iowa City, IA, USA}
\email{frederick-goodman@uiowa.edu}

\begin{abstract}
 
We construct explicit integral bases 
 for the kernels and the images of diagram algebras (including the symmetric groups, orthogonal and symplectic Brauer algebras)
  acting on tensor space.  
We do this by providing an axiomatic framework for studying  
quotients of diagram algebras.
 \end{abstract}

\maketitle
\tableofcontents

  \section*{Introduction}

 Schur--Weyl duality relates  
the
classical matrix groups 
$\GL(V)$,  $\SL(V)$, $\Orth(V)$, or $\Symp(V)$, where $V$ 
is a finite dimensional  
  vector space,
with certain quotients of diagram algebras -- symmetric group algebras, Brauer algebras or walled Brauer algebras --
via mutually centralizing actions on tensor space.
The surjectivity of the map from the diagram algebra to the centralizer algebra of the matrix group is equivalent to the first fundamental theorem of invariant theory.  Any effective description of the kernel of the map is a form of the second fundamental theorem (SFT) of invariant theory.  

This paper studies the centralizer algebras and the second fundamental theorem from the point of view of cellularity  \cite{MR1376244}.  We construct integral cellular bases for the centralizer algebras, and simultaneously bases of the kernel of the map  from the diagram algebras to the centralizer algebras.  

  There are two remarkable cellular bases of the Iwahori Hecke algebras of  finite type $A$ -- the Kazhdan--Lusztig bases \cite{MR560412,MR1376244}  and the Murphy bases ~\cite{MR1327362}.  Each has its own merits.  The Kazhdan--Lusztig bases encode a great deal of representation theory and have a deep relation to geometry. 
  The Murphy  bases are simpler and   more explicit;  they  encode the  restriction of cell-modules
along the tower of Hecke algebras; they are related to the seminormal bases by a dominance triangular transformation and consequently the Jucys--Murphy elements act on the Murphy bases by dominance triangular matrices. 
Relationships between the two types of bases are investigated in~\cite{MR2266962}.
 As evidence of the enduring utility of the Murphy bases, we mention that they were used in \cite{MR2671176}  to construct graded cellular bases of the Hecke algebras.

   The Kazhdan--Lusztig bases have been generalized in ~\cite{MR2510062} to the Brauer centralizer algebras 
   and to many other examples
     using the theory of dual canonical bases  of quantum groups from \cite{MR1227098}. 
   In this paper we concentrate on generalizing the Murphy bases.  In previous work ~\cite{EG:2012}, we have already generalized the Murphy bases to the Brauer diagram algebras (and to other diagram 
    algebras related to the Jones basic construction).  In this paper we extend this analysis to encompass centralizer algebras for the classical groups.   The bases we obtain, for the diagram algebras and for the centralizer algebras, share all the properties of original Murphy bases mentioned above.

 We are using the phrase ``centralizer algebra" as a shorthand for the image of the diagram algebra (for example, the Brauer diagram algebra) acting on tensor space, over an arbitrary field or over the integers. In fact, these algebras are generically centralizer algebras in the classical sense, see \myref{Theorems}{symplectic SW duality} and  \myrefnospace{}{orthogonal SW duality}.
 
 In order to produce Murphy bases of centralizer algebras,  we first develop a quotient construction for cellularity of towers of diagram algebras.  We then  apply this construction to the integral versions of the Brauer algebras acting on orthogonal or symplectic tensor space.  
  The construction involves modifying the Murphy type basis of the tower of diagram algebras constructed following  ~\cite{EG:2012} in such a way that the modified basis splits into a basis of the kernel of the map $\Phi$ from the diagram algebra to endomorphisms of tensor space, and a subset which maps onto a cellular basis of the image of $\Phi$.\footnote{Integral bases of the Brauer diagram algebras with a similar splitting property were constructed in \cite{MR2500869, MR2430306}.}
  This construction thus provides simultaneously an integral cellular basis of the centralizer algebra, and a version of the SFT, namely an explicit description of the kernel of $\Phi$.   Moreover, it is evident from the construction that $\ker(\Phi)$ is generated as an ideal by certain ``diagrammatic minors" or ``diagrammatic Pfaffians",  so we also recover the version of the SFT from ~\cite{MR2434475}.   The combinatorics underlying our construction is the same as that in ~ \cite{MR951511}, namely  the cellular basis of the centralizer algebra is indexed by pairs of ``permissible paths" on the generic branching diagram for the tower of diagram algebras.      The cell modules of the integral centralizer algebras  are in general proper quotients of certain cell modules of the integral diagram algebra.

All of these results are compatible with reduction from $\ZZ$ to a field of arbitrary characteristic
(except that characteristic 2 is excluded in the orthogonal case).   For a symplectic or orthogonal bilinear form on a finite dimensional vector space $V$ over a field $\Bbbk$,  and for $\Phi$ the corresponding map from the Brauer diagram algebra to $\End(V\tensor r)$,  our bases of $\im(\Phi)$ and of $\ker(\Phi)$ are independent of the field, of the characteristic, and  of the choice of the bilinear form.  They depend only on the dimension of $V$.   It follows from our results that for a fixed field $\Bbbk$ and fixed $\dim(V)$,  and for fixed symmetry type of the form (symplectic or orthogonal)  the Brauer centralizer algebra $\im(\Phi)$  is independent, up to isomorphism, of the choice of the form.  For example, if the field is the real numbers,  and the form is symmetric, the Brauer centralizer algebra is independent, up to isomorphism, of the signature of the form.

We also explain in our context the well-known phenomenon that the seminormal representations of centralizer algebras of the classical groups are truncations of the seminormal representations of the corresponding diagram algebras.

We wish to remark upon our emphasis on working over the integers.  As noted in \cite{MR2336039},  cellularity ``provides a systematic framework for studying the representation theory of non-semisimple algebras which are deformations of semisimple ones."
 Typically, a ``cellular algebra" $A$ is actually a family of algebras $A_S$ defined over various ground rings $S$, and typically there is a generic ground ring $R$ such that:  each instance $A_S$ of $A$ is a specialization of $A_R$, i.e.  $A_S = A_R \otimes_R S$;  with $\FF$ the field of fractions of $R$,  $A_\FF$ is semisimple;  and  if  $k$ is any field, 
the cell modules and cellular basis of $A_k$ are obtained by specializing those of $A_R$, and the simple $A_k$ modules appear as heads of (some of) the cell modules. 
  This point of view was not stressed in the original papers  ~\cite{MR1376244, MR1327362}, but is a sort of folk wisdom. 
   In our applications, the integers are the generic ground ring for the centralizer algebras; it is not altogether obvious, but it follows from our results that the centralizer algebras over fields of prime characteristic are specializations of the integral versions; see \myref{Sections}{section Brauer symplectic} and  \myrefnospace{}{section Brauer orthogonal} for precise statements.

In the orthogonal and symplectic cases, our bases are new.  In the general linear case, our result is equivalent to \cite{MR1680384} for tensor space  and ~\cite{MR3473643, Werth:2015} for  mixed-tensor space, respectively.
   A completely different and very general approach to proving the existence of abstract cellular bases of centralizer algebras of quantum groups over a field has been developed in  ~\cite{MR3708257}.

Our method 
should apply to other examples as well. 
The case of the BMW algebra acting on symplectic tensor space should be straightforward, using the $q$--analogue of the diagrammatic Pfaffians obtained in \cite{MR2725207}.  The case of the BMW algebra acting on orthogonal tensor space could be more challenging as the appropriate $q$--analogues of the diagrammatic minors are not yet available.

\noindent{\bf Outline.}
  The paper is structured as follows.     In \myref{Section}{section diagram algebras} we recall  the necessary background material on diagram algebras and their branching graphs; this is  taken from 
\cite{BEG,EG:2012,MR2794027,MR2774622,MR1376244}.  
In \myref{Section}{section quotient framework}, we introduce an axiomatic framework 
for cellularity of a sequence of quotients of a sequence of diagram algebras.  This culminates in \myref{Theorem}{good paths basis theorem}, which contains the main result on cellular bases of quotient algebras as well as an abstract ``second fundamental theorem" --- that is, a description of the kernel of the quotient map.

In \myref{Section}{section Murphy symm} 
 we treat the Murphy basis  of the symmetric group algebras, and a dual  version, twisted by the automorphism $s_i \mapsto -s_i$ of the symmetric group.
 \myref{Section}{section: Murphy bases Brauer} 
treats   the Murphy and dual Murphy bases of  Brauer algebras.
Finally, in \myref{Sections}{section Brauer symplectic} 
and  \myrefnospace{}{section Brauer orthogonal} 
we apply our abstract theory to the main examples of interest in this paper, namely to the Brauer algebra acting on symplectic or orthogonal tensor space.

There are   four appendices  in the arXiv version of this paper. 
In  \iftoggle{arxiv} 
{\myref{Appendix}{symgroup} }   
{Appendix A}
 we review results of 
H\"arterich~\cite{MR1680384} regarding the action of the symmetric group   and the Hecke algebra on ordinary tensor space.  
In 
 \iftoggle{arxiv}
 {\myref{Appendices}  {section: walled Brauer}      }
 {Appendices B }
 \iftoggle{arxiv}
 {\myref{and} {appendix wba tensor space} }
 {and C}
 we construct Murpy bases of the walled Brauer algebras and of their quotients acting on mixed tensor space, following the techniques used in the main text.
In  
\iftoggle{arxiv} 
{\myref{Appendix}{appendix: diagrammatic minors}}   
{Appendix B} 
 we review results on diagrammatic minors and Pfaffians which are needed for our treatment of the SFT. 

\begin{ack}
We would like to thank 
the Royal Commission for the Exhibition of 1851 
and 
EPSRC grant EP/L01078X/1 for   financial support during this project.  We are grateful to the referees for many helpful suggestions and questions.
\end{ack}

\section{Diagram algebras}  \label{section diagram algebras}

For the remainder of the paper, we shall   let $R$ be an integral domain with field of fractions $\mathbb{F}$. 
In this section, we shall define  diagram algebras and recall the construction of their Murphy bases in terms of ``up'' and ``down'' branching factors, following \cite{EG:2012}.   
 As in \cite{BEG}, we emphasize 
  crucial factorization and compatibility relations between the  ``up'' and ``down'' branching factors.

\subsection{Cellular algebras}
We first recall the definition of a cellular algebra, as in \cite{MR1376244}. 

\begin{defn}\label{c-d}
Let $R$ be an integral domain and let $A$ be a unital algebra over $R$.    A {\sf cell datum }  for $A$ is a tuple $(A,*,\widehat{A},\unrhd, \Std(\cdot), \mathscr{A})$ where:
\begin{enumerate}[leftmargin=*,label=(\arabic{*}), ref=\arabic{*},  font=\normalfont, align=left, leftmargin=*, series= cellular defn]
\item  $*:A\to A$ is an algebra  involution, that is, an $R$--linear anti--automorphism of $A$ such that $(x^*)^* = x$ for $x \in A$.
\item $(\widehat{A},\unrhd)$ is a finite partially ordered set, and  for each $\lambda \in \widehat A$, $\Std(\lambda)$ is a finite indexing set.
\item The set
\begin{align*}
\mathscr{A}=\big\{c_\mathsf{st}^\lambda  \ \big | \  \text{$\lambda\in\widehat{A}$ and $\mathsf{s},\mathsf{t}\in\Std(\lambda)$}\big\},
\end{align*}
is an $R$--basis for $A$.

\medskip \noindent
  Let  $A^{\rhd\lambda}$ denote the $R$--module with basis 
\begin{align*}
\big\{
c^\mu_\mathsf{st}\ \mid \mu\rhd\lambda    \text{ and } \mathsf{s},\mathsf{t}\in\Std(\mu) 
\big\}.
\end{align*}
\item
The following  two conditions hold for the basis $\mathscr A$.

\begin{enumerate}[leftmargin=*,label=(\alph{*}), ref=\alph{*}, font=\normalfont, align=left, leftmargin=*]
\item\label{c-d-1} Given $\lambda\in\widehat{A}$, $\mathsf{t}\in\Std(\lambda)$, and $a\in A$, there exist 
coefficients $r(  a; \mathsf{t}, \mathsf{v}) \in R$, for $\mathsf{v}\in\Std(\lambda)$, such that, for all $\mathsf{s}\in\Std(\lambda)$, 
\begin{align}\label{r-act}
c_\mathsf{st}^\lambda a\equiv 
\sum_{\mathsf{v}\in\Std(\lambda)}
r( a; \mathsf{t}, \mathsf{v}) c_{\mathsf{sv}}^\lambda \mod{A^{\rhd\lambda}},
\end{align}

\item\label{c-d-2} If $\lambda\in\widehat{A}$ and $\mathsf{s},\mathsf{t}\in\Std(\lambda)$, then $(c_\mathsf{st}^\lambda)^*  \equiv (c_{\stt\sts}^\lambda)  \mod A^{\rhd \lambda}$.
\end{enumerate}
\end{enumerate}

\medskip \noindent
$A$ is called a {\sf cellular algebra} if it has a cell datum.
   The basis $\mathscr{A}$ is called a {\sf  cellular basis} of $A$.  
\end{defn}

 If $A$ is a cellular algebra over $R$, and $R \to S$ is a homomorphism of integral domains, then  the specialization $A^S = A\otimes_R S$ is a cellular algebra over $S$, with
 cellular basis $$\mathscr A^S = \{c^\lambda_{\sts\stt}	\otimes 1_S	\mid  \lambda\in\widehat{A},  \text{ and }\mathsf{s},\mathsf{t}\in\Std(\lambda) \}.$$
 In particular, $A^\FF$ is a cellular algebra.     Since the map $a \mapsto a \otimes 1_F$ is injective, we regard $A$ as contained in $A^\FF$ and  we identify $a \in A$ with $a \otimes 1_F \in A^\FF$.  
 
An order ideal $\Gamma \subset \widehat A$ is a subset with the property that if $\la \in \Gamma$ and 
$\mu \unrhd \la$, then $\mu \in \Gamma$.  It follows from the axioms of a cellular algebra that for any order ideal $\Gamma$ in $\widehat A$, 
$$
A^\Gamma = \Span \big\{c^\la_{\mfs \mft}  \ \big \vert \   \la \in \Gamma \text{ and }  \mfs, \mft \in \Std(\la)\big\}
$$
is an involution--invariant two sided ideal of $A$.  In particular $A^{\rhd \la}$, defined above,  and
$$
A^{\unrhd \la} = \Span \big\{
c^\mu_{\sf st}\ \big \vert \  \text{$\mu\unrhd\lambda$ and $\sts ,{\stt}\in \Std(\mu)$ }
\big\}
$$
are involution--invariant two sided ideals.

\newcommand\cell[3]{\Delta_{#1}^{#2}(#3)}

\begin{defn} \label{definition: cell module}
Let $A$ be a cellular algebra over $R$ and $\lambda\in\hat{A}$. The  {\sf cell module} $\cell  {} {}  \lambda$  is the right $A$--module defined as follows.  As an $R$--module, $\cell  {} {} \lambda$ is free with basis indexed by $\Std(\lambda)$,  say $\{c^\la_\mft  \mid \mft \in \Std(\la) \}$.  
The right $A$--action is given by 
$$
c_\mft^\la a =  \sum_{{{\sf v}\in\hat{A}^\lambda}} r( a; \mathsf{t}, \mathsf{v})  c^\la_\mfv,
$$
where the coefficients $r( a; \mathsf{t}, \mathsf{v})$  are those of Equation~\eqref{r-act}.
\end{defn}

Thus, for any $\mfs \in \Std(\la)$,    a model for the cell module $\cell {}{}\la$ is given by 
$$
\Span\{ c^\la_{\mfs \mft} + A^{\rhd \la}  \mid  \mft \in \Std(\lambda)\} \subseteq A^{\unrhd \la}/A^{\rhd \la}.
$$ 
 When we need to emphasize the algebra or the ground ring, we may write $\cell A {} \lambda$ or 
 $\cell {} R \lambda$.   Note that whenever $R \to S$ is a homomorphism of integral domains,  
  $\cell {} S \lambda = \cell {} {} \lambda \otimes_R  S$ is the cell module for $A^S$ corresponding to $\lambda$.  
 
If $A$ is an $R$--algebra with involution $*$,  then $*$ induces functors $M \to M^*$   interchanging left and right $A$--modules, and taking $A$--$A$ bimodules to $A$--$A$ bimodules.   We identify $M^{**}$ with 
$M$ via $x^{**} \mapsto x$  and for modules ${}_A M$ and $N_A$ we have 
$
(M \otimes_R N)^* \cong  N^* \otimes_R M^*,  
$
as $A$--$A$ bimodules, with the isomorphism determined by $(m \otimes n)^* \mapsto n^* \otimes m^*$. 
For a right $A$--module $M_A$, using both of these isomorphisms, we identify 
$(M^* \otimes M)^*$ with $M^{*} \otimes M^{**}  = M^* \otimes M$, via 
$(x^* \otimes y)^*  \mapsto y^* \otimes x$.
Now we apply these observations with $A$ a cellular algebra and $\cell {} {} \la$ a cell module.  The assignment $$\alpha_\la : c^\la_{\mfs \mft}  + A^{\rhd \la}  \mapsto  (c^\la_\mfs)^* \otimes (c^\la_\mft)$$ determines an
$A$--$A$ bimodule isomorphism from $A^{\unrhd\la}/\A^{\rhd\la}$ to $(\cell {} {} \la)^* \otimes_R \cell {}  {} \la$. Moreover,
we have   $*\circ \alpha_\la =  \alpha_\la \circ *$, which reflects the cellular algebra axiom
$(c^\la_{\mfs \mft})^* \equiv c^\la_{\mft \mfs} \mod  A^{\rhd \la}$.  

\def\inv{^{-1}}

A certain bilinear form on the cell modules plays an essential role in the theory of cellular algebras.    Let $A$ be a cellular algebra over $R$ and let $\lambda \in \widehat A$.   The cell module $\cell {}{}\la$ can be regarded as an $A/A^{\rhd \la}$ module.
 For
$x, y, z \in \cell {}{}\lambda$, it follows from the definition of  the cell module and the map $\alpha_\lambda$ that $x \alpha_\la\inv(y^* \otimes z) \in R z$.  Define $\langle x, y \rangle$ by
\begin{equation} \label{defn of bilinear form}
x \alpha_\la\inv(y^* \otimes z)  =  \langle x, y \rangle z.
\end{equation}
Then $ \langle x, y \rangle$  is $R$-linear in each variable and we have $\langle x a, y\rangle =
\langle x,  y a^* \rangle$  for $x, y \in \cell {}{}\lambda$ and $a \in A$.    Note that
$$
c^\lambda_{\mfs \mft}  c^\lambda_{\mfu \mfv} = \langle c^\la_\mft, c^\la_\mfu \rangle c^\la_{\mfs \mfv},
$$
which is the customary definition of the bilinear form.

\begin{defn}[\mbox{\cite{MR3065998}}]
A cellular algebra, $A$,  is said to be {\sf cyclic cellular} if every cell module  is cyclic as an $A$-module.   \end{defn}

If $A$ is cyclic cellular, $\la \in \widehat A$, and $\delta(\la)$ is a generator of the cell module $\cell {}{}\la$,
let $\M_\la$ be a lift in $A^{\unrhd \la}$ of $\alpha_\la\inv(\delta(\la)^* \otimes \delta(\la))$, i.e.
$\alpha_\la\inv(\delta(\la)^* \otimes \delta(\la)) = m_\lambda + A^{\rhd \la}$. 

\begin{lem} \label{properties of c lambda}
The element
$\M_\la$ has the following properties:
\begin{enumerate}[leftmargin=*,label=(\arabic{*}), font=\normalfont, align=left, leftmargin=*]
\item \label{cyclic gen a} $\M_\la \equiv \M_\la^* \mod{A^{\rhd \la}}$.
\item  \label{cyclic gen b} $A^{\unrhd \la} =  A \M_\la A +  A^{\rhd \la}$.
\item   \label{cyclic gen c}  $(\M_\la A + A^{\rhd \la})/A^{\rhd \la} \cong \cell {} {} \la$, as right $A$--modules. 
\end{enumerate}
\end{lem}

\begin{proof}  Lemma 2.5 in ~\cite{MR3065998}.
\end{proof}

  \def\cellgenerator{cell generator}
 \def\cellgenerators{cell generators}
We will call the elements $\M_\la$  {\sf \cellgenerators}; 
in examples of interest to us,  they are  given explicitly and satisfy $\M_\la^* = \M_\la$.

We will need the  following elementary lemma regarding specializations of algebras.

\begin{lem} \label{lemma specialization cd}
Let $R$ be a commutative ring with identity, $A$ an $R$--algebra, and $M$ an $A$--module 
Let $\tau : R \to S$ be a unital ring homomorphism.   Note that $M \otimes_R S$ is an $A \otimes_R S$ module.
Let $\varphi : A \to \End_R(M)$ be the homomorphism corresponding to the $A$--module structure of $M$,   and 
$\varphi_S : A \otimes_R S \to \End_S(M \otimes_R S)$  the homomorphism corresponding to the $A \otimes_R S$--module structure of $M \otimes_R S$.  Then there exists an $R$--algebra homomorphism
$\theta : \varphi(A) \to \varphi_S(A \otimes_R S)$, making the following diagram commute:
\begin{equation} \label{specialization CD}
 \begin{minipage}{5cm}\begin{tikzpicture}
 \draw (0,0) node {$A$};\draw (3,0) node {$\varphi(A)$};
\draw[->](0.4,0) to node [above] {\scalefont{0.8} $\varphi$} (2.4,0);
\draw[->](0,-0.3) to node [right] {\scalefont{0.8}$\otimes 1_S$} (0,-1.2);
\draw[->](3,-0.3) to node [right] {\scalefont{0.8}$\theta$} (3,-1.2);
\draw[->](0.9,-1.5) to node [below] {\scalefont{0.8} $\varphi_S$} (1.9,-1.5);
\draw (0,-1.5) node {$A\otimes_R S$};\draw (3,-1.5) node {$\varphi_S(A\otimes_R S)$};
\end{tikzpicture}\end{minipage}.
\end{equation}
\end{lem}

\begin{proof}   
Note that $\varphi_S$ is defined by $\varphi_S(a \otimes 1_S)(m \otimes 1_S) = \varphi(a)(m) \otimes 1_S$.    Define $\theta(\varphi(a)) = \varphi_S(a \otimes 1_S)$.  This is well defined because if $a \in \ker(\varphi)$,  then $a \otimes 1_S \in \ker \varphi_S$.
\end{proof}

\begin{rem} \label{remark injectivity of specialization}
 In case $R \subset S$ are fields, the map $\theta$ in \eqref{specialization CD} is injective, because
$\theta(\varphi(a))(m \otimes 1_S) = \varphi(a)(m) \otimes 1_S$.   If $\theta(\varphi(a)) = 0$,  then $\varphi(a)(m) = 0$ for all $m \in M$, so $\varphi(a) = 0$.  
\end{rem}
 
\subsection{Sequences of diagram algebras}  \label{subsection diagram algebras}
 Here and in the remainder of the paper, we will consider    an   increasing sequence   $(A_r)_{r \ge 0}$
of  cellular  algebras over an integral domain $R$ with field of fractions $\FF$.   
We assume that all the inclusions  $A_r  \hookrightarrow A_{r+1}$  are unital and that the involutions are consistent; that is the involution on $A_{r+1}$, restricted to $A_r$, agrees with the involution on $A_r$.      We will establish a list of assumptions \eqref{diagram 1}--\eqref{diagram 6}.  For convenience, we call an increasing sequence of  cellular algebras satisfying these assumptions a {\sf sequence of diagram algebras}. 

Let $(\widehat A_r, \unrhd)$ denote the partially ordered set in the cell datum for $A_r$.   For $\lambda \in \widehat A_r$,  let $\Delta_r(\lambda)$  denote the corresponding  cell module.    If $S$ is an integral domain with a unital homomorphism $R \to S$,  write $A_r^S = A_r \otimes_R S$  and $\Delta_r^S(\lambda)$  for $\Delta_r(\lambda) \otimes_R S$.    In particular, write 
 $A_r^\FF = A_r \otimes_R \FF$  and $\Delta_r^\FF(\lambda)$  for $\Delta_r(\lambda) \otimes_R \FF$.

 \begin{defn} \label{definition: cell-filtration}  Let $A$ be a cellular  algebra over $R$.  
 If $M$ is a right $A$--module, a  {\sf  cell-filtration}  of  $M$  is a filtration by right $A$--modules
\begin{align*}
\{ 0\}= M_0  \subseteq   M_1  \subseteq \cdots  \subseteq    M_r=M,
\end{align*}
such that $  M_{i} / M_{i-1} \cong  \cell {} {} {\lambda\power i}$ for some $\lambda^{(i)} \in \widehat{A}$.   We  say that the filtration is {\sf  order preserving} if  $\lambda^{(i)}\rhd \lambda^{(i+1)}$ in
  $\widehat{A}$  for all $i\geq 1$.  
 \end{defn}

\begin{defn}  Let $A \subseteq B$ be a unital inclusion of  cellular algebras over an integral domain $R$ (with  consistent involutions).   
\begin{enumerate}[leftmargin=*,label=(\arabic{*}), font=\normalfont, align=left, leftmargin=*]
\item Say the inclusion is  {\sf  restriction--coherent}  if for every $\mu \in \widehat{B}$, the restricted module $\Res^B_A(\cell B {} \mu)$ has an order preserving cell-filtration (as an $A$--module). 
\item Say the inclusion is {\sf  induction--coherent} if  for every $\lambda \in \widehat{A}$, the induced module $\Ind^B_A(\cell A {} \lambda)$ has an order preserving cell-filtration (as a $B$--module).
\end{enumerate}
\end{defn}
 
 \begin{defn}[\cite{MR2794027,MR2774622}]  Let  $(A_r)_{r \ge 0}$ be an increasing sequence of cellular algebras over an integral domain $R$.  
 We say the tower is {\sf  restriction--coherent} if each inclusion $A_r \subseteq A_{r+1}$ is restriction coherent, and {\sf  induction--coherent} if each inclusion is induction coherent.  We say the tower is coherent if it is both restriction-- and induction--coherent.
 \end{defn}

 \begin{rem}  We have changed the terminology from ~\cite{MR2794027,MR2774622, EG:2012}, as the weaker notion of coherence, in which the order preserving requirement is omitted, plays no role here. 
 \end{rem}
 
 We now list the first of our assumptions for a sequence of diagram algebras:
 \begin{enumerate}[leftmargin=*,label=(D\arabic{*}), ref=D\arabic{*},  series = DiagramAlgebras]
\item  \label{diagram 1}  $A_0 = R$.
 \item   \label{diagram 2}  \label{diagram cyclic cellular}  The algebras $A_r$ are cyclic cellular  for all $\k\geq 0$.  
 \end{enumerate}
 
 For all $r$ and for all $\lambda \in \widehat A_r$, fix once and for all a bimodule isomorphism
$\alpha_\lambda: A_r^{\unrhd \la}/A_r^{\rhd \la} \to (\cell r {} \lambda)^* \otimes_R \cell r {} \lambda$, 
a generator $\delta_r(\la)$ of the cyclic $A_r$--module $\cell r {} \lambda$,  
and a \cellgenerator \ $\M_\lambda \in A_r^{\unrhd \la}$  satisfying  $\alpha_\la(\M_\la + A_r^{\rhd \la}) = (\delta_r(\la))^* \otimes \delta_r(\la)$, as in the discussion preceding
\myref{Lemma}{properties of c lambda}.  
We require the following mild assumption on the \cellgenerators.

 \begin{enumerate}[leftmargin=*,label=(D\arabic{*}), ref=D\arabic{*},  resume = DiagramAlgebras]
 \item  \label{diagram symmetry}  \label{diagram 5}  The cell generators satisfy
 $\M_\la = \M_\la^*$.
 \end{enumerate}
 
 Our list of assumptions continues as follows:

  \begin{enumerate}[leftmargin=*,label=(D\arabic{*}), ref=D\arabic{*},  resume = DiagramAlgebras]
 \item  \label{diagram 3}  \label{diagram semisimplicity}
 $A_r^\FF$ is split semisimple for all $\k\geq 0$.  
\item  \label{diagram 4} \label{diagram restriction coherent}
 The sequence of algebras  $(A_r)_{r \ge 0}$  is  restriction--coherent.  
\end{enumerate}

As discussed in ~\cite[Section 3]{EG:2012},  under the assumptions \eqref{diagram 1}--\eqref{diagram restriction coherent}
above, there exists a well-defined  
multiplicity--free branching diagram $\widehat A $ associated with the sequence $(A_r)_{r \ge 0}$.  The branching diagram is an infinite, graded, directed graph with vertices $\widehat A_r$ at level $r$ and edges determined as follows.  
If  $\lambda \in \widehat A_{r-1}$ and $\mu \in \widehat A_r$, there is an edge $\lambda \to \mu$ in $\widehat A$ if and only if $\Delta_{r-1}(\lambda)$  appears as a subquotient of an order preserving cell filtration of  $\Res^{A_{r}}_{A_{r-1}} (\Delta_r(\mu))$.   
 In fact, $\la \to \mu$ if and only if 
the simple $A^\FF_{r-1}$--module $\cell {\k-1} \FF \lambda$ is a direct summand of the restriction of  $\cell {\k} \FF \mu$ to $A_{r-1}^\FF$.   Note that $\widehat A_0$ is a singleton;  we denote   its unique element by $\varnothing$.   We can choose $\cell 0 {} \varnothing = R$,  $\delta_0(\varnothing) = 1$, and $\M_\varnothing = 1$.

  \begin{defn}
Given  $\nu \in \widehat{A}_{r}$, we  
define a  {\sf  standard   tableau} of shape $ \nu$  to be a 
directed path  $\stt$  on the branching diagram $\widehat A$ from $\varnothing \in \widehat A_0$ to $\nu$, 
\begin{equation} \label{path notation}
\stt = (\varnothing= \stt(0) \to \stt(1) \to    \stt(2)\to   \dots \to  \stt(r-1)\to \stt(r) = \nu). 
\end{equation}
We   let $\Std_r(\nu ) $     denote the set of all such paths and we set 
$\Std_{r}= \cup_{ \nu \in \widehat{A}_r}\Std_{r}(\nu )$.    \end{defn}

Given an algebra satisfying axioms \eqref{diagram 1} to   \eqref{diagram restriction coherent} it is shown in ~\cite[Section 3]{EG:2012} that 
there exist certain ``down--branching factors" 
$d_{\lambda \to \mu} \in A_r$,  for $\lambda \in \widehat A_{\k-1}$ and $\mu \in \widehat A_r$  with 
  $\lambda \to \mu$ in $\widehat A$, related to the cell filtration of $\Res^{A_{\k}}_{A_{\k-1}} (\Delta_r(\mu))$.  Given a  path
$\mft \in \Std_{r}(\nu )$ as in \eqref{path notation}
 define the ordered product $d_\stt$  of branching factors by
\begin{equation} \label{d of a path}
d_\mft =  d_{\stt({r-1}) \to \stt( r)}  d_{ \stt( {r-2}) \to \stt( {r-1})} \cdots  d_{\stt(0) \to \stt( 1)}.
\end{equation}
We say two cellular bases of an algebra $A$ with involution are {\sf equivalent} if they determine the same
two sided ideals $A^{\unrhd \lambda}$  and isomorphic cell modules.

\begin{thm}[\cite{EG:2012},  Section 3] \label{theorem abstract Murphy basis}  
Let $(A_\k)_{\k \ge 0}$ be a sequence of algebras satisfying assumptions \eqref{diagram 1}--\eqref{diagram restriction coherent}. 
\begin{enumerate}[leftmargin=*,label=(\arabic{*}), font=\normalfont, align=left, leftmargin=*]
\item Let $\lambda \in \widehat A_\k$.  The set  $\{ \M_\lambda d_\mft  +  A_\k^{\rhd \lambda} \suchthat    \mft \in \Std_\k(\lambda)  \}$   is a basis of the cell module $\Delta_\k(\lambda)$. 
\item The set  $\{ d_\mfs^* \M_\lambda d_\mft \suchthat  \lambda \in \widehat A_\k \text{ and } \mfs, \mft \in \Std_\k(\lambda) \} $ is  a cellular basis of $A_\k$, equivalent to the original cellular basis. 
\item  For a fixed $\lambda \in \widehat A_\k$, we let $\mu(1) \rhd \mu(2) \rhd \cdots  \rhd \mu(s)$ be a listing of the $\mu \in \widehat A_{\k-1}$ such that $\mu \to \lambda$.    Let
$$
M_j = \Span_R \leftbrace  \M_\lambda d_\mft +  A_\k^{\rhd \lambda} \suchthat  \mft \in \Std_\k(\lambda), \mft(k-1) \unrhd \mu(j)                         \rightbrace.
$$
Then
$$
(0) \subset  M_1 \subset  \cdots \subset  M_s = \Delta_\k(\lambda)
$$
is a filtration of $\Delta_\k(\lambda)$ by $A_{\k-1}$-submodules, and $M_j/M_{j-1} \cong \Delta_{\k-1}(\mu(j))$.  
\end{enumerate}
\end{thm}

We will now continue with our list of assumed properties of the sequence of algebra 
$(A_\k)_{\k \ge 0}$ with one final key axiom.  
 
 \begin{enumerate}[leftmargin=*,label=(D\arabic{*}), ref=D\arabic{*},  resume=DiagramAlgebras]
  \item   \label{diagram 6} \label{diagram compatibility}
 There exist   ``up--branching factors"  $u_{\lambda \to \mu}\in A_{\k}^R$  for $\lambda \in \widehat A_{\k-1}$ and $\mu \in \widehat A_{\k}$ satisfying the compatibility relations
 \begin{equation} \label{abstract branching compatibility}
 \M_\mu d_{\lambda \to \mu} =   (u_{\lambda \to \mu})^*  \M_\lambda. 
\end{equation}
 \end{enumerate}

\begin{eg}\label{theoneexample}  It is shown in ~\cite{EG:2012} that the Hecke algebras of type $A$, the symmetric group algebras,  the   Brauer algebras, the Birman--Wenzl--Murakami algebras,  the partition algebras, and the Jones--Temperley--Lieb algebras  all are examples of sequences of algebras satisfying properties \eqref{diagram 1}--\eqref{diagram 6}.  In 
\iftoggle{arxiv}
 {\myref{Appendix}  {section: walled Brauer}      }
 {Appendix B in the arXiv version of this paper},
we show that one can extract single sequences from the double sequence of walled Brauer algebras, so that 
 properties \eqref{diagram 1}--\eqref{diagram 6} are satisfied.   
In each case the ground ring $R$ can be taken to be the generic ground ring for the class of algebras.  For example, for the Hecke algebras, this is $\Z[\boldq, \boldq\inv]$, and for the Brauer algebras it is $\Z[\deltabold]$, where  $\boldq$ and   $\deltabold$ are indeterminants.
\end{eg}

\begin{rmk} In all the examples listed above,  the branching factors $d_{\lambda \to \mu}$  and $u_{\lambda \to \mu}$ and the \cellgenerators \  $m_\lambda$ are determined explicitly.  For the symmetric group algebras and the Hecke algebras of finite type $A$, the branching factors can be chosen so that  the basis $\{d_\mfs^* \M_\lambda d_\mft \}$ coincides with  Murphy's cellular basis or its dual version, see  \myref{Section}{dual Murphy basis of S n}.   
 In all of these examples, 
$u$-branching factors are related to cell filtrations of induced cell modules; see \cite{EG:2012} for details.  However, for the purposes of this paper it is enough to know that the $u$-branching coefficients exist and are explicitly determined.

\end{rmk}

\begin{defn}
We write $m_{\mfs  \mft}^\lambda =   d_\mfs^*  \M_\lambda d_\mft$.    Also write
$m_\mft =  \M_\lambda d_\mft + A_\k^{\rhd \lambda} \in \Delta_\k(\lambda)$.   We refer to the cellular basis
$\{m_{\mfs  \mft}^\lambda  \suchthat  \la \in \widehat A_\k \text{ and }  \mfs, \mft \in \Std_\k(\la)\}$ as the {\sf Murphy  cellular basis} of $A_\k$  and $\{m^\la_\mft  \suchthat \mft \in \Std_\k(\la) \}$  as the 
{\sf Murphy  basis} of the cell module $\cell r {} \la$. 
\end{defn}


 \begin{rmk} (Remark on notation for branching factors)   Let $\lambda \in \widehat A_{\k-1}$  and $\mu \in \widehat A_r$  with $\lambda \to \mu \in \widehat A $.   In situations where it seems  helpful to emphasize the level on the branching diagram, we will write, for example,  $\dd \lambda \mu r$ instead of $d_{\lambda \to  \mu}$.   
See for instance, 
\myref{Theorem}{theorem:  closed form determination of the branching factors}.  
\end{rmk}

  \begin{defn}
Given $0\leq s \leq r$ and    $\lambda  \in   \widehat{A}_{s}$, $\nu \in \widehat{A}_{r}$, we  
define a  {\sf  skew  standard   tableau} of shape $ \nu \setminus  \lambda  $ and degree $r-s$  to be a directed path $\stt$  on the branching diagram $\widehat A$ from $\lambda$ to $\mu$, 
\begin{equation} \label{skew path notation}
\lambda = \stt(s) \to \stt(s+1) \to    \stt(s+2)\to   \dots \to  \stt(r-1)\to \stt(r) = \nu. 
\end{equation}
We let $\Std_{s,r}(\nu \setminus  \lambda)$ denote the set of all such paths with given $\lambda$ and $\nu$. 
 Given $0\leq s \leq r$, we set $\Std_{s,r}= \cup_{\la \in \widehat{A}_s,\nu \in \widehat{A}_r}\Std_{s,r}(\nu\setminus\lambda)$.  
 \end{defn}
 
\ignore{
\begin{rmk}Given two paths $\sts \in \Std_s(\la)$ and $\stt \in \Std_{s,r}(\nu\setminus\la)$, we let $\sts \circ \stt\in \Std_{r}(\nu )$ denote the obvious   path obtained by concatenation.  
\end{rmk}
}
  
   Given two paths $\sts \in \Std_{q, s}(\mu \setminus \lambda)$ and $\stt \in \Std_{s, r}(\nu \setminus \mu)$ such that the final point of $\sts$ is the initial point of $\stt$,  define $\sts \circ \stt \in \Std_{q, r}(\nu \setminus \lambda)$ to be the obvious path obtained by concatenation.
 
\begin{rmk} \label{factorization of Murphy basis}
  
 Given a  path
$\mft \in \Std_{s,r}(\nu \setminus  \lambda)$  as in \eqref{skew path notation}
define
\begin{align*}
d_\mft &=  d_{\stt({r-1}) \to \stt( r)}  d_{ \stt( {r-2}) \to \stt( {r-1})} \cdots  d_{\stt(s) \to \stt(s+ 1)},  \\
\intertext{and}
u_\mft  &=     u_{\stt(s) \to \stt( s+1)}  \cdots
u_{ \stt( {r-2}) \to \stt( {r-1})} u_{\stt({r-1}) \to \stt( r)}.
\end{align*}
  Then it follows from the compatibility relation \eqref{abstract branching compatibility} and induction on $r -s$ that
\begin{equation} \label{abstract branching compatibility for path}
u_\mft^*  \M_\la =  \M_\nu  d_\mft.
\end{equation}
Because $\M_\varnothing$ can be chosen to be $1$,  this gives in particular for $\mft \in \Std_r(\nu)$, 
\begin{equation} \label{abstract branching compatibility for path 2}
u_\mft^*  =  \M_\nu  d_\mft.
\end{equation}
Therefore the cellular basis $\{ m_{\mfs \mft}^\nu\}$  can also be written in the apparently asymmetric form
$$
 m_{\mfs  \mft}^\nu = d_\mfs^* \M_\nu d_\mft =  d_\mfs^*  u_\mft^*.  
$$
Using the symmetry of the cellular basis  $(m_{\mfs  \mft}^\nu)^* = m_{\mft \mfs}^\nu$ (which follows from the assumption \eqref{diagram symmetry}), we also get
$$
 m_{\mfs  \mft}^\nu = u_\mfs  d_\mft.
$$
Using \eqref{abstract branching compatibility for path 2}, we have the following form for the basis $\{m^\la_\mft \suchthat  \mft \in \Std_r(\nu) \}$ of the cell module $\cell r {} \nu$:
\begin{equation}\label{abstract branching compatibility for path 3}
m^\la_\mft =  u^*_\mft + A_r^{\rhd \nu}.
\end{equation}
Now, for any $0\leq q \le s \le r$,   let $\mft_{[q, s]}$ denote the truncated path,
$$
 \stt(q) \to \stt(q+1) \to    \stt(q+2)\to   \dots \to  \stt(s-1)\to \stt(s). 
 $$
The representative $u^*_\mft$ of $m_\mft$  has the remarkable property that  for any $0\leq s \le r$,
\begin{equation} \label{eqn factorization of u t}
u^*_\mft =  u^*_{\mft_{[s, r]}}   u^*_{\mft_{[0, s]}},
\end{equation}
and 
\begin{equation}  \label{eqn property of u t trunacation}
 u^*_{\mft_{[0, s]}} = \M_{\mft(s)} d_{\mft_{[0, s]}} \in  \M_{\mft(s)} A_s \subseteq A_s^{\unrhd \mft(s)}.
 \end{equation}
 Here, \eqref{eqn factorization of u t} follows directly from the definition of $u_{\stt}$, while \eqref{eqn property of u t trunacation}   comes from applying \eqref{abstract branching compatibility for path 2} to $ u^*_{\mft_{[0, s]}}$ in place of $u^*_{\mft}$.
   The compatibility relations \eqref{abstract branching compatibility for path} together with the factorizability 
\eqref{eqn factorization of u t}  
 of representatives $u^*_\stt$  of the Murphy basis play a crucial role in this paper. 
   In our view, these are the distinguishing properties of the Murphy bases of diagram algebras, and even in the original context of the Hecke algebras ~\cite{MR1327362} these properties provide  new insight. 
  \end{rmk}

\subsection{Seminormal bases, dominance triangularity, and restriction of cell modules}
\label{section some consequences from BEG}
We have explored certain consequences of our standing 
assumptions \eqref{diagram 1}--\eqref{diagram 6} in an companion paper ~\cite{BEG}.   We recall some of the results of that paper that will be applied here.

One can define analogues of seminormal bases in the algebras and the cell modules defined over the field of fractions $\FF$, as follows.  Let 
$\z r \lambda$ denote the minimal central idempotent in $A_r^\FF$  corresponding to the minimal two sided ideal labeled by $\lambda \in \widehat A_r$.     For $r \ge s$ and for a path
$\stt \in \Std_{s,r}(\nu \setminus \lambda)$ as in \eqref{skew path notation}, define
$$
F_\mft = \prod_{s \le j \le r}  \z j {{\mft\powerr j}}.
$$
The factors are mutually commuting so the order of the factors does not have to be specified.  In particular the set of $F_\mft$ for $\mft \in \Std_r(\nu)$  and $\nu \in \widehat A_r$,  is a family of mutually   orthogonal minimal idempotents, with $\sum_{\mft \in \Std_r(\nu) }  F_\mft = \z r \nu$.    The collection of idempotents
$F_\mft$   (for $r \ge 1$,  $\nu \in \hat A_r$, and $\mft \in \Std_r(\nu)$)  is called the family of {\em Gelfand-Zeitlin} idempotents for the tower $(A_r)_{r \ge 0}$.   The family is characterized in ~\cite[Lemma 3.10]{MR2774622}.

Define   $f_\mft = m_\mft F_\mft $ in $\Delta_r^\FF(\nu)$  and 
$F_{\mfs  \mft} =  F_\mfs  m_{\mfs  \mft}^\lambda F_\mft$, for $\nu \in \widehat A_r$ and
$\mfs, \mft \in \Std_r(\nu)$.    These are analogues for diagram algebras of the seminormal bases of the Hecke algebras of the symmetric groups.  This construction, and its relation to other constructions of seminormal bases, is discussed in detail in ~\cite{BEG}.

 The following two partial orders on standard tableaux play an important role in the theory of diagram algebras.

\begin{defn}[Dominance order for paths]  For $\mfs, \mft \in \Std_{s, r}$,   define $\mfs \unrhd \mft$ if $\mfs(j) \unrhd \mft(j)$ for all $s\leq j\leq \k$.  
\end{defn}

This is evidently a partial order, which we call the  dominance order.  In particular  the dominance order is defined on $\Std_r$ and on $\Std_r(\nu)$ for $\nu \in \widehat A_r$.    The corresponding strict partial order is denoted 
$\mfs \rhd \mft$ if $\mfs \ne \mft$ and $\mfs \unrhd \mft$.

\begin{defn}[Reverse lexicographic order for paths]   For $\mfs, \mft \in \Std_{s, r}$,   define $\mfs \succeq \mft$ if  $\mfs = \mft$ or if for the last index $j$ such that $\mfs(j) \ne \mft(j)$, we have
$\mfs(j) \rhd \mft(j)$.  
\end{defn}

This is also a partial order on paths.  The corresponding strict partial order is denoted $\mfs \succ \mft$ if $\mfs \ne \mft$ and $\mfs \succeq \mft$.   Evidently $\mfs \rhd \mft$ implies $\mfs \succ \mft$.

We now review several results from ~\cite{BEG}.  The most useful technical result is that the Murphy bases and the seminormal bases of the cell modules are related by a dominance--unitriangular transformation.

\begin{thm}[\cite{BEG}, Theorem 3.3]  \label{dominance triangularity}
 Fix $\la \in \widehat A_r$.     For all $\mft \in \Std_r(\la)$,  there exist coefficients 
 $r_\mfs, r'_\mfs  \in \FF$ such that
$$
m^\la_\mft =  f^\la_\mft  +  
\smashoperator[r]{\sum_{\begin{subarray}c \mfs \in \Std_r(\la) \\  \mfs \rhd \mft \end{subarray}}  } \ 
r_\mfs f^\la_\mfs 
\qquad \quad 
f^\la_\mft =  m^\la_\mft  +  
\smashoperator[r]{\sum_{\begin{subarray}c \mfs \in \Std_r(\la) \\  \mfs \rhd \mft \end{subarray}}}\ 
 r'_\mfs m^\la_\mfs. 
$$
 \end{thm}
 
\begin{cor}[\cite{BEG}, Corollary 3.4]   \label{corollary seminormal basis}  For $r \ge 0$, we have that 
\begin{enumerate}[leftmargin=*,label=(\arabic{*}), font=\normalfont, align=left, leftmargin=*] 
\item
   $\{f^\la_\mft \suchthat \mft \in \Std_r(\la)\}$ is a basis of $\cell r \FF \la$ for all  $\la \in \widehat A_r$.
\item    $\{ F_{\mfs \mft}^\la \suchthat \la \in \widehat A_r \text{ and }  \mfs, \mft \in \widehat A_r\}$  is a cellular basis of $A_r^\FF$.  
\end{enumerate}
\end{cor}

 \begin{prop}[\cite{BEG}, Proposition 3.9] \label{path basis dominance theorem} Let $1 \le s < r$,  $\nu \in \widehat A_r$,  $\la  \in \widehat A_{s}$ and  $\mft  \in \Std_{s,r}(\nu \setminus \la )$.   Let $x \in \M_\la  A_{s}$ and write
 $$
 x = 
\smashoperator[r]{ \sum_{\mathsf s \in \Std_{s}(\la )}  }\ 
 \alpha_{\mathsf s}  u^*_{\mathsf s}  + y,
 $$
 with $y \in A_{s}^{\rhd \la }$.    Then there exist coefficients $r_{\mathsf z} \in R$, such that 
 $$
 u^*_\mft x \equiv 
 \smashoperator[r]{\sum_{\mathsf s \in \Std_{s}(\la )} }\  
 \alpha_{\mathsf s}  u^*_\mft u^*_{\mathsf s}  + 
 \sum_{\mathsf z}  r_{\mathsf z}  u^*_{\mathsf z}   \quad \mod A_r^{\rhd \nu},
 $$
 where the sum is over $\mathsf z \in \Std_r(\nu)$ such that $\mathsf z_{[s, r]} \rhd \mft$ and
 $\mathsf z(s) \rhd \la $.  
 \end{prop}

Finally we mention, without going into details, the relation of the assumptions \eqref{diagram 1}--\eqref{diagram 6} to {\sf Jucys--Murphy elements}.    Assume that  $(A^S_\k)_{\k \ge 0}$ is a tower of algebras satisfying assumptions \eqref{diagram 1}--\eqref{diagram 6}  and assume in addition that the tower has Jucys--Murphy elements, in the sense of  \cite{MR2774622}.   This assumption holds for  Hecke algebras of type $A$, the symmetric group algebras,  the   Brauer algebras, the Birman--Wenzl--Murakami algebras,  the partition algebras, and the Jones--Temperley--Lieb algebras.  We will see in 
\myref{Sections}{section Brauer symplectic}  and \myrefnospace{}{section Brauer orthogonal}
that it also holds for the Brauer centralizer algebras acting on symplectic and orthogonal tensor spaces.  It is shown in ~\cite{BEG}
that the Jucys--Murphy elements act diagonally on the seminormal bases and dominance unitriangularly on the Murphy bases, generalizing a result of Murphy \cite[Theorem 4.6]{MR1194316}   for the Hecke algebras.

\newcommand{\sfR}{{R}}
\newcommand{\sfF}{{\mathbb F}}

\subsection{Cellularity and the Jones basic construction} \label{subsection: Jones}
In this section, we recall the framework of \cite{MR2794027,MR2774622,EG:2012}.  
This framework allows one to lift the cellular structure from a   coherent sequence $(H_\k)_{\k \ge 0}$ of cyclic cellular algebras to a second sequence $(A_\k)_{\k \ge 0}$, related to the first sequence by ``Jones basic constructions".   
Most importantly, we will recall how the branching factors and cell generators for the tower $(A_\k)_{\k \ge 0}$ can be explicitly constructed from those of the tower $(H_\k)_{\k \ge 0}$.

   The list of assumptions regarding the two sequence of algebras, from ~\cite[Section 5]{EG:2012}, is the following:    
$(H_\k)_{\k \ge 0}$ and $(A_\k)_{\k \ge 0}$ are  both sequences of algebras over an integral domain $\sfR$ with field of fractions $\sfF$.  The inclusions are unital, and both sequences of algebras have consistent algebra involutions $*$.   Moreover:
\begin{enumerate}[leftmargin=*,label=(J\arabic{*}), ref=J\arabic{*},  series =Jones]
\item  \label{J: A0 and A1}  \label{J-1}
$A_0 = H_0 = \sfR$   and $A_1 = H_1$  (as algebras with involution).
\item \label{J: idempotents}  \label{J-2}
 There is a $\delta \in S$ and for each $\k \ge 2$,  there is an element $e_{\k-1} \in A_\k$  satisfying $e_{\k-1}^* = e_{\k-1}$ and $e_{\k-1}^2 = \delta e_{\k-1}$.  
For $\k \ge 2$, 
$
e_{\k-1} e_\k e_{\k-1} = e_{\k-1}$ and  $ e_\k e_{\k-1} e_\k = e_\k $.
\item \label{J-3}
 For $\k \ge 2$, 
 $A_\k/(A_\k e_{\k-1} A_\k)  \cong H_\k$ as algebras with involution.

\item \label{J: en An en} For $\k \ge 1$,   $e_{\k}$ commutes with $A_{\k-1}$ and $e_{\k} A_{\k} e_{\k} \subseteq  A_{\k-1} e_{\k}$.
\item  \label{J:  An en}
For $\k \ge 1$,  $A_{\k+1} 	e_{\k} = A_{\k} e_{\k}$,  and the map $x \mapsto x e_{\k}$ is injective from
$A_{\k}$ to $A_{\k} e_{\k}$.

\item  \label{J:  Delta J}  \label{J-6}
For $\k \ge 2$,   $e_{\k-1} A_\k  e_{\k-1}  A_\k  =   e_{\k-1}  A_\k$.
\item  \label{J: semisimplicity}
For all $\k$,  $A_\k^\sfF : = A_\k \otimes_\sfR \sfF$   is split semisimple.  
\item \label{axiom Hn coherent}  \label{J-8}
$(H_\k)_{\k \ge 0}$ is a   coherent tower of cyclic cellular algebras.
\end{enumerate}

The conclusion (\cite[Theorem 5.5]{EG:2012}) is that $(A_\k)_{\k \ge 0}$  is a coherent tower of cyclic cellular algebras over $\sfR$ (in particular the tower $(A_\k)_{\k \ge 0}$ satisfies conditions  \eqref{diagram 1},  \eqref{diagram cyclic cellular},  \eqref{diagram semisimplicity}, and   \eqref{diagram restriction coherent}. 
We let $(\widehat H_\k, \trianglerighteq)$ denote the partially ordered set in the cell datum for $H_\k$.   Then
the partially ordered set in the cell datum for $A_\k$ is
$$
\widehat A_\k = \{ (\lambda, l) \suchthat  0 \le l \le \lfloor \k/2 \rfloor   \text{ and }  \lambda \in \widehat H_{\k-2 l}\},
$$   
with partial order $(\lambda, l) \trianglerighteq  (\mu, m)$ if $l > m$ or if $l = m$ and $\lambda \trianglerighteq \mu$.   The branching diagram for the tower $(A_\k)_{\k \ge 0}$ is $\widehat A = \bigsqcup_{\k \ge 0} \widehat A_\k$  with the branching rule  $(\la, l) \to (\mu, m)$ if $l = m$  and $\la \to \mu$ in $\widehat H$ or if $m = l+1$ and 
$\mu \to \lambda$ in $\widehat H$.    We call this the branching diagram obtained {\sf by reflections} from $\widehat H$.

We will now explain how the branching factors and cell generators for the tower $(A_\k)_{\k \ge 0}$ can be explicitly constructed from those of the tower $(H_\k)_{\k \ge 0}$.   
For $\k \ge 2$, let 
\begin{equation} \label{notation e sub i super ell}
e_{\k-1}^{(\el)} =  
\begin{cases}
1  & \text{ if }   \el = 0 \\
\underbrace{e_{\k-2\el+1}e_{\k-2\el+3}\cdots e_{\k-1}}_{\text{$\el$ factors}}
&\text{ if }  \el=1,\ldots,\lfloor \k/2\rfloor,  \text{ and }\\
0 &\text{ if }  \el>\lfloor \k/2\rfloor.
\end{cases}
\end{equation}
Let $d_{\la \to \mu}$ and $u_{\la \to \mu}$  denote down-- and up--branching factors, 
and let $m_\la$ denote cell generators   for $(H_\k)_{\k \ge 0}$.  Let 
 $\bar d_{\la \to \mu}$,  $\bar u_{\la \to \mu}$, and $\bar m_\la$   denote liftings of these elements in  the algebras $A_\k$.   Then we have the following two results:

\newcommand{\ddbar}[3]{\bar d_{{#1} \to {#2}}^{\power {#3}}}
\newcommand{\uubar}[3]{\bar u_{{#1} \to {#2}}^{\power {#3}}}
 
\begin{thm}[\cite{EG:2012}, Theorem 5.7] \label{theorem:  closed form determination of the branching factors}
The branching factors  for the tower $(A_\k)_{\k\ge 0}$ can be chosen to satisfy:
\begin{enumerate}[leftmargin=*,label=(\arabic{*}), font=\normalfont, align=left, leftmargin=*]
\item $\dd {(\la, \el)} {(\mu, \el)} {\k+1} = \ddbar \la \mu {\k + 1 - 2\el}  e_{\k-1} \power {\el}$.
\item $\uu {(\la, \el)} {(\mu, \el)} {\k+1} = \uubar \la \mu {\k + 1 - 2\el}  e_{\k} \power {\el}$.
\item $\dd {(\la, \el)} {(\mu, \el+1)} {\k+1}  = \uubar \mu \la {\k-2\el} e_{\k-1} \power \el$.  
\item   $\uu {(\la, \el)} {(\mu, \el+1)} {\k+1}  = \ddbar \mu \la {\k-2\el} e_{\k} \power {\el+1}$. 
\end{enumerate}
\end{thm}

 \begin{lem}[\cite{EG:2012}, Section 5.5]  \label{lemma lifting cell generators}
  For $(\la, l) \in \widehat A_\k$, the 
\cellgenerator\  $m_{(\la, l)}$  in $A_r^{\unrhd (\la, l)}$ can be chosen as
 $\M_{(\lambda, l)}   = \bar  \M_\lambda e_{\k-1}\power{l-1}$.
 \end{lem}

\begin{rmk}
Although these  results involve unspecified liftings of elements from $H_\k$ to $A_\k$,  in the examples, the liftings are chosen explicitly.  Moreover, the \cellgenerators \ $m_\la$ in $H_\k$ and
$m_{(\la, l)}$ in $A_\k$  are chosen  to be $*$--invariant, so that the tower $(A_r)$  satisfies axiom 
\eqref{diagram symmetry}.   Furthermore, in the examples, the branching factors and \cellgenerators \ in the algebras $H_r$ satisfy the compatibility relation \eqref{diagram 6},   and their liftings can be chosen to satisfy these relations as well.  It then follows  from \myref{Theorem}{theorem:  closed form determination of the branching factors} and \myref{Lemma}{lemma lifting cell generators}  that 
 the branching factors and \cellgenerators\ in the algebras $A_r$  also satisfy the compatibility relations  \eqref{diagram 6}.  
\end{rmk}

Now the tower $(A_\k)_{\k\ge 0}$ in particular satisfies the conditions \eqref{diagram 1}--\eqref{diagram compatibility} over $\sfR$,  so each $A_\k$ has a Murphy type  cellular basis obtained by the prescription of  
\myref{Theorem}{theorem abstract Murphy basis},
using ordered product of $d$--branching factors along paths on $\widehat A$.

\begin{rmk} For the standard examples of diagram algebras, for example the Brauer algebras, all this works not over the generic ground ring $R = \ZZ[\deltabold]$, but only over  $R[\deltabold\inv]$.  
  However, the branching factors and \cellgenerators \ obtained from \myref{Theorem}{theorem:  closed form determination of the branching factors} and \myref{Lemma}{lemma lifting cell generators} do lie in the algebras over the generic ground ring.  Furthermore, one can check  that the transition matrix between the diagram basis of the algebras and the Murphy type cellular basis is invertible over the generic ground ring; this step is case--by--case and somewhat {\em ad hoc}.  
It follows that  the tower of algebras $(A_\k^R)_{\k \ge 0}$ over the generic ground ring satisfies all of the conditions  \eqref{diagram 1}--\eqref{diagram compatibility}.  This is explained in detail in \cite[Sections 5 and 6]{EG:2012}.
  \end{rmk}

 \section{A framework for cellularity of quotient algebras}  \label{section quotient framework}
 
 \newcommand{\BbbK}{{\mathbb K}}
 
As explained in the introduction, cellularity does not pass to quotients in general, but nevertheless we intend to show that cellularity does pass to the quotients of certain abstract diagram or tangle algebras acting on tensor space.   In this section, we will develop an axiomatic framework for this phenomenon.
In the remainder of the paper, this framework will be applied   to Brauer's centralizer algebras acting on orthogonal or symplectic tensor space.  
In 
\iftoggle{arxiv}
 {\myref{Appendix}  {appendix wba tensor space}    }
 {Appendix C in the arXiv version of this paper}
we show that the walled Brauer algebras acting on mixed tensor space can be treated in an identical fashion.  
   
\subsection{A setting for quotient towers}\label{subsection setting quotient towers}
 We consider a tower of cellular algebras $(A_\k)_{\k \ge 0}$ over  an integral domain $R$ satisfying the properties
\eqref{diagram 1}--\eqref{diagram 6}   of 
\myref{Section}{subsection diagram algebras}.
 In particular, for each $r$, we have the cellular basis $$\{ d_\mfs^* \M_\lambda d_\mft \suchthat  \lambda \in \widehat A_\k \text{ and } \mfs, \mft \in \Std_\k(\lambda) \} $$  of $A_\k$ from 
 \myref{Theorem}{theorem abstract Murphy basis},
and we write
$m_{\mfs  \mft}^\lambda =   d_\mfs^*  \M_\lambda d_\mft$.

 Suppose that $S$  is an integral domain with field of fractions $\BbbK$ and that $\pi : R \to S$ is a surjective ring homomorphism.
We consider the  specialization $A_\k^S = A_\k \otimes_R S$ of the algebras $A_\k$.
Let  $(Q_\k^\BbbK)_{\k \ge 0}$ be a  tower  of unital algebras 
 over $\BbbK$, with common identity, 
 and  with surjective 
 homomorphisms 
 $\phi_\k : A_\k^\BbbK  \to  Q_\k^\BbbK$.      We denote $\phi_\k(A_\k^S) \subseteq  Q_\k^\BbbK$ by $Q_\k^S$.  
  
 We suppose that the homomorphisms are consistent with the inclusions of algebras,  $\phi_{\k+1} \circ \iota = \iota \circ \phi_\k$,  where $\iota$ denotes both the inclusions $\iota: A_\k^\BbbK \to A_{\k+1}^\BbbK$  and $\iota:  Q_\k^\BbbK \to Q_{\k+1}^\BbbK$.    In particular, this implies  that $\ker(\phi_\k) \subseteq \ker(\phi_{\k+1})$.  Because of this, we will  usually just write $\phi$ instead of writing $\phi_\k$.

 \begin{defn}\label{defn quotient tower}
We say that $(Q_\k^S)_{\k \ge 0}$ is a {\sf quotient tower} of $(A_\k^S)_{\k \ge 0}$ if the following axioms hold.  
\begin{enumerate}[leftmargin=*,label=(Q\arabic{*}), ref=Q\arabic{*},  series = QuotientAlgebras,leftmargin=*]
\item  \label{quotient axiom 1} There is a distinguished subset $\widehat A_{\k, \perm}$ of ``permissible" points in $\widehat A_\k$.  
The point    
  $\varnothing \in \widehat A_0$ is permissible,  and for each $\k$ and permissible $\mu$ in $\widehat A_\k$,  there exists  at least one permissible $\nu$ in $\widehat A_{\k+1}$ with $\mu \to \nu$ in 
$\widehat A$,  and (for $\k \ge 1$)  at least one permissible $\lambda$ in $\widehat A_{\k -1}$ with $\lambda \to \mu$ in $\widehat A$. 
\end{enumerate}

A path $\mft \in \Std_r(\nu)$ will be called permissible if $\mft(k)$ is permissible for all $0\leq k \leq r$.    Write 
$\Std_{r, \perm}(\nu)$ for the set of permissible paths in $ \Std_r(\nu)$.

\begin{enumerate}[leftmargin=*,label=(Q\arabic{*}), ref=Q\arabic{*},  resume = QuotientAlgebras]
\item  \label{quotient axiom 2} If $\mft \in \Std_r(\nu)$ is not permissible, let  $1\leq k \le r$  be the first index such that  $\mu = \mft(k)$ is not permissible.   Then  there exist elements $\mathfrak{b}_\mu$ and $ \mathfrak{b}'_\mu$  in $A_k^S$  such that 
\begin{enumerate}
\item  $\M_\mu =  \mathfrak{b}_\mu - \mathfrak{b}'_\mu$.
\item  $\mathfrak{b}_\mu \in \ker(\phi)$.
\item  $\mathfrak b'_\mu \in  \M_\mu A_k^\BbbK \cap (A_k^S)^{\rhd \mu}$.
\end{enumerate}

\item  \label{quotient axiom 3}  With $\Kbar$ the algebraic closure of $\BbbK$, we have
$$
\dim_{\Kbar}(Q_r^{\Kbar}) =  
\smashoperator[r]{\sum_{\nu \in \widehat A_{r, \perm}}}\ 
   (\sharp \Std_{r, \perm}(\nu))^2.
$$
\end{enumerate}
\end{defn}

 \begin{rmk}(Some notation and terminology)    \label{remark on evaluable elements}
  Let $\mathfrak p = \ker(\pi)$, a prime ideal in $R$, and  let $R_{\mathfrak p} \subset \FF$ be the localization of $R$ at $\mathfrak p$.    Thus $R_{\mathfrak p}$  is a local ring with unique maximal ideal $\mathfrak p R_{\mathfrak p}$ and residue field $\BbbK$, and
 $\pi: R \to S$ extends to a surjective ring homomorphism $\pi : R_{\mathfrak p} \to \BbbK$.  
  We have surjective evaluation maps, also denoted $\pi$  
   from $A_s^{R_{\mathfrak p}}$ to $A_s^\BbbK$ given by 
 $
 \pi(\sum  \alpha_{\mfu  \mfv}^\alpha    m_{\mfu  \mfv}^\alpha)   = 
 \sum  \pi(\alpha_{\mfu  \mfv}^\alpha)    m_{\mfu  \mfv}^\alpha.
 $
 and from $\Delta_{A_s}^{R_{\mathfrak p}}(\lambda)$  to $\Delta_{A_s}^{\BbbK}(\lambda)$ given by 
 $
 \pi(\sum_\mft  \alpha_\mft m^\lambda_\mft) =  \sum \pi(\alpha_\mft)  m^\lambda_\mft
 $.
 We often refer to  $R_{\mathfrak p}$,  or $A_s^{R_{\mathfrak p}}$,  or  $\Delta_{A_s}^{R_{\mathfrak p}}(\lambda)$  as the set of {\em evaluable} elements (in $\FF$,  or   $A_s^{\FF}$, or  $\Delta_{A_s}^\FF$, respectively). 
\end{rmk}
  
\subsection{Cellular bases of quotient towers}  \label{section Cellular bases of quotient towers}
 We are now going to show that under the assumptions \eqref{quotient axiom 1}--\eqref{quotient axiom 3}, 
 the quotient algebras $Q_\k^S$ are cellular algebras with a cellular basis $\{\phi(d_\mfs^* \M_\lambda  d_\mft) \mid \lambda \in \widehat A_{\k, \perm} \text{ and } \sts,\stt\in \Std_{\k, \perm}(\lambda)\}$.      Furthermore, we will produce a cellular basis $\{\nn_{\sts \stt}^\lambda\}$ of $A_\k^S$, equivalent to the  cellular basis $\{m_{\sts \stt}^\lambda \}$,  with the
  properties that $\nn_{\sts \stt}^\lambda = m_{\sts \stt}^\lambda$ in case both $\sts$ and $\stt$ are 
  permissible, and $\nn_{\sts \stt}^\lambda \in \ker(\phi)$ otherwise.   In particular,  the set of $\nn_{\sts \stt}^\lambda$ 
  such that at least one of $\sts$ and $\stt$ is not permissible constitutes an $S$--basis of $\ker(\phi)$.

 \begin{lem}   \label{path basis dominance lemma}   Assume as in the discussion above that
 $S$ is an integral domain with field of fractions $\BbbK$ and that $\pi : R \to S$ is a surjective ring homomorphism.    Let  $0\leq s < r$,  $\mu \in \widehat A_s$,   and $x \in \M_\mu A_s^\BbbK \cap (A_s^S)^{\rhd \mu}$. 
 Let $\lambda \in \widehat A_r$ and suppose $\mft \in \Std_{s,r}(\lambda \setminus \mu)$.   Then there exist coefficients
 $\alpha_\mfz \in S$  such that
 $$
 u^*_{\mft} x  \equiv  \sum_\mfz  \alpha_\mfz  \M_\lambda d_\mfz  \quad \mod ((A_r^S)^{\rhd \lambda} \cap \M_\lambda A_r^S ),
 $$
 where the sum is over $\mfz \in \Std_r(\lambda)$ with $\mfz_{[s, r]} \rhd \mft$ and $\mfz(s) \rhd \mu$.
 \end{lem}
 
 \begin{proof}   We will apply 
 \myref{Proposition}{path basis dominance theorem},
 but we cannot do so directly.      Recall the notation from \myref{Remark}{remark on evaluable elements}.
  By hypothesis, $x  = \M_\mu \beta$,  where $\beta \in A_s^\BbbK$  and $x \in (A_s^S)^{\rhd \mu}$.  Lift $\beta$ to an element $\beta_0 \in A_s^{R_\mathfrak p}$ and let $x_0 = \M_\mu \beta_0$.    Since $x_0 \in \M_\mu  A_s^{R_\mathfrak p}$, we can write
 $$
 x_0 \equiv  
 \smashr{\sum_{\mfv \in \Std_s(\mu) }}\ 
  r_\mfv  \M_\mu  d_\mfv  \quad \mod (A_s^{R_{\mathfrak p}})^{\rhd \mu}.
 $$
 Since $\pi(x_0) = x \in (A_s^S)^{\rhd \mu}$ it follows that $\pi(r_\mfv) = 0$ for  all $\mfv\in \Std_s(\mu)$.   Now we can apply 
 \myref{Proposition}{path basis dominance theorem}
 to $x_0$, with $R$ replaced by $R_\mathfrak p$, which gives us
 $$
 u_\mft^* x_0 \equiv 
 \smashr{ \sum_{\mfv \in \Std_s(\mu) } }\ r_\mfv  \M_\lambda  d_\mft  d_\mfv  + 
 \smashr{\sum_{
 \begin{subarray}c
\mfz \in \Std_r(\lambda) \\ 
\mfz_{[s, r]} \rhd \mft \\
\mfz(s) \rhd \mu
 \end{subarray}
 } }\ 
  r'_{\mfz}  \M_\lambda d_{\mfz}   \quad \mod  (A_r^{R_\mathfrak p})^{\rhd \lambda},
 $$
   Applying the evaluation map $\pi$ and recalling that $\pi(r_\mfv) = 0$ gives
\begin{equation} \label{eqn triangularity lemma 3.3}
 u^*_{\mft} x   =   
\smashr{ \sum_{
 \begin{subarray}c
\mfz \in \Std_r(\lambda) \\ 
\mfz_{[s, r]} \rhd \mft \\
\mfz(s) \rhd \mu
 \end{subarray}
 } }\ 
 \alpha_\mfz  \M_\lambda d_\mfz    + z
 \end{equation}
 where $\alpha_\mfz \in \BbbK$ and $z \in (A_r^\BbbK)^{\rhd \lambda}$.  But since $ u^*_{\mft} x \in A_r^S$, we must have $\alpha_\mfz \in S$ and $z \in (A_r^S)^{\rhd \lambda}$.  Finally, since $u^*_\mft x \in \M_\lambda A_r^S$, it follows from 
\eqref{eqn triangularity lemma 3.3} that $z \in  \M_\lambda A_r^S$. 
 \end{proof}

 \newcommand{\f}{r}
 \begin{lem} \label{good paths lemma 1}  Assume \eqref{quotient axiom 1}--\eqref{quotient axiom 3}.
Let $\lambda \in \widehat A_\f$ and let $ \mft  \in \Std_\f(\lambda)$.  If $\mft$ is not  permissible,  then there exist coefficients  $r_\mfv \in  S$ such that
$$
\M_\lambda d_\mft = 
\smashr{
 \sum_{
\begin{subarray}c
\mfv \in \Std_{\f,\perm}(\la) \\
\mfv \succ \mft
\end{subarray}
} 
}\ 
r_\mfv \M_\lambda d_\mfv +  x_1 + x_2, 
$$
where   $x_1 \in \ker(\phi)$, and 
$x_2 \in (A_\f^S)^{\rhd \lambda} \cap \M_\lambda A_\f^S$.   Hence for all 
for all $\mfs \in  \Std_\f(\lambda)$,
$$
m_{\mfs \mft}^\lambda  \equiv 
\smashr{ \sum_{
\begin{subarray}c
\mfv \in \Std_{\f,\perm}(\nu) \\
\mfv \succ \mft
\end{subarray}
} } \ 
 r_\mfv m_{\mfs \mfv}^\lambda  \quad \mod ((A_\f^S)^{\rhd \lambda} + \ker(\phi)).
$$ 
\end{lem}

\begin{proof}  Since $\mft$ is not permissible,  by assumption \eqref{quotient axiom 2}  there exists $0\leq k \le \f$ such that $\mu = \mft(k)$ satisfies the following:  there are elements $\mathfrak{b}_\mu$ and $\mathfrak{b}'_\mu$ in 
$A_k^S$ such that $\M_\mu = \mathfrak{b}_\mu - \mathfrak{b}'_\mu$,    $\mathfrak{b}_\mu \in \ker(\phi)$,  and
\begin{align}\label{nbmvcxzbnxmzbvncxmzvbxnczm}\mathfrak{b}'_\mu \in  \M_\mu A_k^\BbbK \cap (A_k^S)^{\rhd \mu}.
\end{align}
Write $\mft_1 = \mft_{[0, k]}$ and  $\mft_2 = \mft_{[k, \f]}$
Using the branching compatibility relation \eqref{abstract branching compatibility for path},
\begin{equation} \label{eqn good paths 1:1}
\M_\lambda d_\mft  =  \M_\lambda d_{\mft_2} d_{\mft_1 }= u^*_{\mft_2} \M_\mu d_{\mft_1}
= u^*_{\mft_2} \mathfrak{b}_\mu d_{\mft_1} - u^*_{\mft_2} \mathfrak{b}'_\mu d_{\mft_1}
\end{equation}
The first term $u^*_{\mft_2} \mathfrak{b}_\mu d_{\mft_1}$   in \eqref{eqn good paths 1:1}  lies in $\ker(\phi)$.  
Recall $\mathfrak{b}'_\mu d_{\mft_1} \in \M_\mu A_k^\BbbK \cap (A_k^S)^{\rhd \mu}$ by  \eqref{nbmvcxzbnxmzbvncxmzvbxnczm}, and so we can apply
\myref{Lemma}{path basis dominance lemma}
to conclude that the second term $- u^*_{\mft_2} \mathfrak{b}'_\mu d_{\mft_1}$  in \eqref{eqn good paths 1:1}  satisfies
\begin{equation} \label{eqn good paths 1:2}
- u^*_{\mft_2} \mathfrak{b}'_\mu d_{\mft_1} \equiv  
\smashr{ \sum_
 {
 \begin{subarray}c
 \mfv \in  \Std_r(\lambda)\\
 \mfv_{[k, \f]} \rhd \mft_2
 \end{subarray}
 }  }\ 
 \alpha_\mfv  \M_\lambda d_\mfv  \quad \mod ((A_\f^S)^{\rhd \lambda} \cap \M_\lambda A_\f^S),
\end{equation}
where $\alpha_\stv \in S$.  Note that the condition $\mfv_{[k, \f]} \rhd \mft_2$  implies  that  $\mfv \succ \mft$.  
This gives us
$$
\M_\lambda d_\mft  \equiv 
\smashr{\sum_{
\begin{subarray}c 
\stv \in \Std_r(\lambda)\\ 
\mfv \succ \mft		 
\end{subarray}
}  }\ 
\alpha_\mfv \M_\lambda d_\mfv  \quad \mod (\ker(\phi) +   (A_\f^S)^{\rhd \lambda} \cap \M_\lambda A_\f^S).  
$$
    By  induction on the ordering, $\succ$, on $\Std_\f(\nu)$  we obtain
$$
\M_\lambda d_\mft  \equiv 
\smashr{\sum_{
\begin{subarray}c 
\stv \in \Std_{r,\perm} (\lambda)\\ 
\mfv \succ \mft		 
\end{subarray}
}   }\ 
r_\mfv \M_\lambda d_\mfv  \quad \mod (\ker(\phi) +   (A_\f^S)^{\rhd \lambda} \cap \M_\lambda A_\f^S),
$$
where now the sum is over permissible  paths only.   This gives the first assertion in the statement of the lemma. 
Finally, multiplying on the left by $d_\mfs^*$  yields the second statement.
\end{proof}

We are now going to produce the cellular basis $\{\nn_{\sts \stt}^\lambda\}$ of $A_\k$, equivalent to the original cellular basis $\{m_{\sts  \stt}^\lambda\}$ with the properties that $\nn_{\sts  \stt}^\lambda = m_{\sts \stt}^\lambda$ in case both $\sts$ and $\stt$ are permissible, and $\nn_{\sts  \stt}^\lambda \in \ker(\phi)$ otherwise. 

Let $\stt \in \Std_\f(\lambda)$ be a non--permissible path.  Let $1 \le k \le \f$ be the first index such that $\mu = \stt(k)$ is not permissible. 
It follows from \eqref{quotient axiom 2} that the element $\mathfrak b_\mu$ is in $\ker(\phi) \cap \M_\mu A_k^\BbbK$, so there exists a $\beta_\mu \in A_k^\BbbK$ with $\mathfrak b_\mu = \M_\mu \beta_\mu$.    Let $\stt_1 = \stt_{[0,k]}$ and $\stt_2 = \stt_{[k, \f]}$.    Following the proof of 
\myref{Lemma}{good paths lemma 1},
and using in particular \eqref{eqn good paths 1:1} and \eqref{eqn good paths 1:2}, we get
\begin{equation} \label{eqn n-basis 1}
\M_\lambda d_\stt \equiv  u^*_{\stt_2} \mathfrak{b}_\mu d_{\stt_1} + 
\smashr{\sum_{
\begin{subarray}c
\stv \in \Std_\f(\lambda) \\ 
\stv \succ \stt
\end{subarray}} 
}\ 
\alpha_\stv \M_\lambda d_\stv \quad \mod (A_\f^S)^{\rhd \lambda},
\end{equation}
 for $\alpha_\stv \in S$.   Since  $\mathfrak b_\mu = \M_\mu \beta_\mu$, we have
$u^*_{\stt_2} \mathfrak{b}_\mu d_{\stt_1} = \M_\lambda d_{\stt_2} \beta_\mu d_{\stt_1}$, using \eqref{abstract branching compatibility for path}.    Substitute this into \eqref{eqn n-basis 1} and transpose to get 
\begin{equation} \label{eqn n-basis 2}
\M_\lambda d_{\stt_2} \beta_\mu d_{\stt_1} \equiv \M_\lambda d_\stt - 
\smashr{\sum_{
\begin{subarray}c
\stv \in \Std_\f(\lambda) \\ 
\stv \succ \stt
\end{subarray}}  
}\ 
\alpha_\stv \M_\lambda d_\stv \quad \mod (A_\f^S)^{\rhd \lambda}.
\end{equation}
Note that the left hand expression is in $\ker(\phi)$.  For any  non-permissible path  $\stt$, we define $a_\stt = d_{\stt_2} \beta_\mu d_{\stt_1} $
to be the element which we arrived at in \eqref{eqn n-basis 2}.
Although $a_\stt$  is {\em a priori}  in $A_\f^\BbbK$, \eqref{eqn n-basis 2} shows that $\M_\lambda a_\stt  \in   \M_\lambda A_r^S$.
Passing to the cell module $\cell r {} \lambda $,
 we have
\begin{equation}  \label{eqn n-basis 3}
\M_\lambda a_\stt  + (A_\f^S)^{\rhd \lambda} = m^\lambda_\stt - 
 \smashr{\sum_{
\begin{subarray}c
\stv \in \Std_\f(\lambda) \\ 
\stv \succ \stt
\end{subarray}}  
}\ 
\alpha_\stv  m^\lambda_\stv .  
\end{equation}
For $\stt \in \Std_\f(\lambda)$ permissible, define $a_\stt = d_\stt$.  For any $\stu, \stv \in \Std_\f(\lambda)$, permissible or not, define $\nn^\lambda_\stv = \M_\lambda a_\stv +  (A_\f^S)^{\rhd \lambda} $, 
and   $\nn^\lambda_{\stu  \stv} = a_\stu^* \M_\lambda a_\stv$. 
We remark that in all   examples, the elements 
$ \mathfrak b_\mu, \mathfrak b_\mu',$ and $\beta_\mu$ 
 will be explicitly described as elementary sums of Brauer-type diagrams.

\begin{thm}  \label{basis theorem}  
Assume \eqref{quotient axiom 1}--\eqref{quotient axiom 3}.
The set 
$$
\mathbb B_\f = 
\leftbrace \nn_{\mfs  \mft}^\lambda  \suchthat 
 \nn_{\mfs  \mft}^\lambda:=
a_\sts^* \M_\lambda a_\stt, 
 \lambda \in \widehat A_{\f } \text{ and }   
   \mfs, \mft \in \Std_{\f}(\lambda) \rightbrace 
$$
is a cellular basis of $A_\f^S$ equivalent to the original cellular basis.  It has the property that
$\nn_{\mfs  \mft}^\lambda  = m_{\mfs  \mft}^\lambda $ if both $\sts$ and $\stt$ are permissible and 
$\nn_{\mfs  \mft}^\lambda \in \ker(\phi)$ otherwise.
  \end{thm}
 \begin{proof}
  Equation \eqref{eqn n-basis 3} shows that  $\{\nn^\lambda_\stv \suchthat \stv \in \Std_\f(\lambda)\}$  is related to the $S$--basis $\{m^\lambda_\stv \suchthat \stv \in \Std_\f(\lambda)\}$  of the cell module $\cell r S \lambda$ by a unitriangular transformation with coefficients in $S$, and therefore $\{\nn^\lambda_\stv \suchthat \stv \in \Std_\f(\lambda)\}$ is also an
 $S$--basis of the cell module.
 
For $\stu$ and $\stv$  arbitrary elements of $\Std_\f(\lambda)$  
  we have 
 $\alpha_\lambda(\nn^\lambda_{\stu  \stv}  + (A_\f^S)^{\rhd \lambda})  = (\nn^\lambda_\stu)^* \otimes \nn^\lambda_\stv$.  It follows from 
  \cite[Lemma 2.3]{MR3065998}  that  $\{\nn^\lambda_{\stu  \stv} \suchthat  \lambda \in \widehat A_\f, \stu, \stv \in \Std_\f(\lambda)\}$  is a cellular basis of $A_\f^S$ equivalent to the original cellular basis 
  $\{m^\lambda_{\stu  \stv} \}$.

   It is evident from the construction that
  $\nn_{\mfs  \mft}^\lambda  = m_{\mfs  \mft}^\lambda $ if both $\sts$ and $\stt$ are permissible and 
$\nn_{\mfs  \mft}^\lambda \in \ker(\phi)$ otherwise.
  \end{proof}

 \begin{defn} \label{marginal point}
Call  $\mu \in \widehat A_s$ a {\sf marginal point} if $\mu$ is not permissible and there exists a path
$\mft \in \Std_s(\mu)$ such that $\stt(k)$ is permissible for all $k < s$.
\end{defn}

\begin{thm}  \label{good paths basis theorem}  
Assume \eqref{quotient axiom 1}--\eqref{quotient axiom 3}. Then
\begin{enumerate}[leftmargin=*,label=(\arabic{*}), font=\normalfont, align=left, leftmargin=*]
\item \label{good paths 1}
 $\ker(\phi_r)$ is globally invariant under the involution $*$.  Hence one can define an algebra involution on  $Q_\f^S = \phi(A_\f^S)$  by $(\phi(a))^* = \phi(a^*)$.
\item   \label{good paths 2}
 The algebra   $Q_\f^S = \phi(A_\f^S)$ is a cellular algebra over $S$ with  cellular basis 
$$
\mathbb A_\f = \leftbrace \phi( m_{\mfs \mft}^\lambda) \suchthat  \lambda \in \widehat A_{\f, \perm} \text{ and }  \mfs, \mft \in \Std_{\f, \perm}(\lambda) \rightbrace .
$$
More precisely, the cell datum is the following: the   involution *  on $Q_r^S$ defined in part  {\rm (1)};   the partially ordered set  $(\widehat A_{\f, \perm}, \unrhd)$   of permissible points in $\widehat A_\f$;   for each $\lambda \in  \widehat A_{\f, \perm}$,   the index set $\Std_{\f, \perm}(\lambda)$ of permissible paths of shape $\lambda$; and finally the basis $\mathbb A_\f$.
\item    \label{good paths 3}
The set 
$$
 \kappa_r = 
\{\nn^\lambda_{\sts \stt} \suchthat  \lambda \in \widehat A_\f \text{ and } \sts  \text{ or  } \stt \text{ is not permissible}\}
$$
is an $S$--basis of $\ker(\phi_\f)$.   
\item   \label{good paths 4}
$\ker(\phi_\f)$ is the ideal $\mathfrak I_\f$  in $A_\f^S$  generated by the set 
of $\mathfrak b_\mu$,  where $\mu$ is a marginal point of $\widehat A_s$ for some $0 <s \le \f$.  
\end{enumerate}
 \end{thm}

\begin{proof}  
Since $\mathbb B_\f$ is a basis of $A_\f^S$, by 
\myref{Theorem}{basis theorem},
it follows that 
$\phi(\mathbb B_\f)$ spans $Q_\f^S = \phi(A_\f^S)$ over $S$.  But
$\phi(\nn^\lambda_{\sts \stt} )  = \phi(m^\lambda_{\sts \stt} )$ if both $\sts$ and $\stt$ are permissible, and 
$\phi(\nn^\lambda_{\sts \stt} )  = 0$ otherwise.   It follows that $\mathbb A_\f$ spans 
$Q_\f^S $ over $S$, hence spans $Q_\f^{\Kbar} $ over $\Kbar$.  Since by assumption 
  \eqref{quotient axiom 3},  $\dim_{\Kbar}(Q_\f^{\Kbar}) = \sharp (\mathbb A_\f),$ it follows that $\mathbb A_\f$ is linearly independent over $\Kbar$.    Thus $\mathbb A_\f$ is an $S$--basis  of $Q_r^S$.  
 
The $S$--span of  $\kappa_r$  is contained in $\ker(\phi_r)$ by  
\myref{Theorem}{basis theorem}.   
On the other hand, it  follows from the linear independence of $\mathbb A_\f$ that  $\ker(\phi_\f)$ has trivial intersection with the $S$--span of $\{\nn^\lambda_{\sts \stt} \suchthat \lambda \in \widehat A_\f \text{ and } \sts, \stt \in \Std_{\f, \perm}(\lambda)\}$.   It follows from this that  
 $\kappa_r$ 
spans, and hence is a basis of, $\ker(\phi_\f)$. 

The cellular basis $\mathbb B_r$ of $A_r^S$ satisfies $ (\widetilde m_{\sts\stt}^\la )^\ast = \widetilde m_{\stt\sts}^\la$, and it follows that 
$\ker(\phi_r)$, namely  the  $S$-span of $\kappa_r$,   is globally invariant under $*$.       Hence one can define an algebra involution on  $Q_\f^S = \phi(A_\f^S)$  by $(\phi(a))^* = \phi(a^*)$.

So far, we have proved points  \ref{good paths 1} and  \ref{good paths 3}, and shown that  $\mathbb A_r$ is an $S$--basis of $Q_r^S$. Next we check that $\mathbb A_r$ is a cellular basis of $Q_r^S$,  by appealing to  \myref{Theorem}{basis theorem}. 
   For $\lambda \in \widehat A_{\f, \perm}$ and
$\mfs, \mft \in \Std_{\f, \perm}(\lambda)$, and for $a \in  A_\f^S$, we have by cellularity of $A_\f^S$  with respect to the basis   $\mathbb B_r$,
$$
\nn_{\mfs  \mft}^\lambda a \equiv 
\smashr{\sum_{\mfv \in \Std_\f(\lambda)} } \ 
 r_\mfv  \nn_{\mfs  \mfv}^\lambda  \quad \mod (A_\f^S)^{\rhd \lambda},
$$
where the coefficients are in $S$ and independent of $\mfs$, and the sum goes over all $\stv \in \Std_\f(\lambda)$.   When we apply $\phi$, only those terms with permissible $\stv$ survive:
$$
\phi(m_{\mfs  \mft}^\lambda) \phi( a)  \equiv 
\smashr{\sum_{\mfv \in \Std_{\f, \perm}(\lambda)} } \ 
 r_\mfv  \phi(m_{\mfs  \mfv}^\lambda)  \quad \mod  \phi((A_\f^S)^{\rhd \lambda}).
$$
Again by \myref{Theorem}{basis theorem},
 we have
$\phi((A_\f^S)^{\rhd \lambda}) = (Q_\f^S)^{\rhd \lambda}$.  This verifies the multiplication axiom for a cellular basis.  The involution axiom is easily verified using part \ref{good paths 1}, namely
$ 
\phi(m_{\mfs \mft}^\lambda) ^* = \phi({(m_{\mfs \mft}^\lambda)}^*)  =
\phi(m_{\mft  \mfs}^\lambda).  
$   This completes the proof of part \ref{good paths 2}.

It remains to check part \ref{good paths 4}.   By construction of the basis $\mathbb B_\f$, we have that
$$ 
\kappa_r \subseteq \mathfrak I_\f \subseteq \ker(\phi_\f).$$  Therefore it follows from part \ref{good paths 3} that
 $\ker(\phi_\f) = \mathfrak I_\f$.  
\end{proof}
  
 Since $Q_r^S$ is a quotient of $A_r^S$, in particular its cell modules are $A_r^S$--modules.  We observe that the cell modules of  $Q_r^S$  are quotients of  cell modules of $A_r^S$, when regarded as $A_r^S$--modules. 
    
\begin{cor}   \label{good paths cor}
Assume \eqref{quotient axiom 1}--\eqref{quotient axiom 3}. Then
\begin{enumerate}[leftmargin=*,label=(\arabic{*}), font=\normalfont, align=left, leftmargin=*]
\item \label{good paths 5}  For $\lambda$  a permissible point in $\widehat A_r$, 
$\kappa(\lambda) = \Span_S\{ \nn^\lambda_\stt \suchthat \stt \text{ is not permissible} \}$  is an $A_r^S$--submodule of the cell module $\cell {A_r} S \lambda$, and   $\cell {Q_r} S \lambda \cong \cell {A_r} S \lambda/ \kappa(\lambda)$ as $A_r$--modules.
\item  \label{good paths 6}  Assume $Q_r^\BbbK$ is split semisimple and $\lambda$ is a permissible point in $\widehat A_r$.  Then $\kappa(\lambda) \otimes_S K = \Span_\BbbK\{ \nn^\lambda_\stt \suchthat \stt \text{ is not permissible} \}$  is the radical of $\cell {A_r} \BbbK \lambda$.
\end{enumerate}
\end{cor}

\begin{proof}    By \myref{Theorem}{basis theorem},
 we have
$\phi((A_\f^S)^{\unrhd \lambda}) = (Q_\f^S)^{\unrhd \lambda}$ and $\phi((A_\f^S)^{\rhd \lambda}) = (Q_\f^S)^{\rhd \lambda}$, so $\phi$ induces an $A_r^S$--$A_r^S$ bimodule homomorphism
$(A_\f^S)^{\unrhd \lambda}/(A_\f^S)^{\rhd \lambda} \to (Q_\f^S)^{\unrhd \lambda}/(Q_\f^S)^{\rhd \lambda}$.
For a fixed permissible $\sts \in \Std_{r, \perm}(\lambda)$, 
$$
\widetilde m^\lambda_{\sts \stt} + (A_\f^S)^{\rhd \lambda} \mapsto \phi(\widetilde m^\lambda_{\sts \stt}) + (Q_\f^S)^{\rhd \lambda} 
$$
 defines   a right  $A_r$--module homomorphism from $\cell {A_r} S \lambda$ to $\cell {Q_r} S \lambda$, with kernel $\kappa(\lambda)$.     It follows that
$\cell {Q_r} \BbbK \lambda \cong \cell {A_r} \BbbK \lambda/ (\kappa(\lambda) \otimes_S \BbbK)$.   If 
$Q_r^\BbbK$ is split semisimple, then its cell modules are simple, so $\cell {Q_r} \BbbK \lambda$ is the simple head of $\cell {A_r} \BbbK \lambda$ and $\kappa(\lambda) \otimes_S K$ is the radical.
\end{proof}  

Let us review what we have accomplished here, with a view towards our applications in   
\myref{Sections}{section Brauer symplectic} and \myrefnospace{}{section Brauer orthogonal}.
Suppose we have a tower $(A^R_\k)_{\k \ge 0}$ of diagram algebras, satisfying axioms \eqref{diagram 1} to \eqref{diagram 6} of 
\myref{Section}{subsection diagram algebras}.
and specializations $A_\k^S$ together with maps $\phi_\k: A_\k^S \to Q_\k^S$ which satisfy the conditions of  \myref{Definition}{defn quotient tower}.
Then we can produce all of the following:
\begin{enumerate}
\item A modified Murphy basis $\{\widetilde m_{\sts \stt}^\lambda\}$  of each of the algebras $A_\k^S$ which is equivalent to the basis $\{ m_{\sts \stt}^\lambda\}$ of 
\myref{Theorem} {theorem abstract Murphy basis}.
If $\sts$ and $\stt$ are both permissible, then 
$\widetilde m_{\sts \stt}^\lambda=  m_{\sts \stt}^\lambda$.  However, if either  $\sts$ or $\stt$ is impermissible, then $\widetilde m_{\sts \stt}^\lambda $ belongs to the  kernel of $\phi_\k$.  
\item An $S$-linear basis of $\ker(\phi_\k)$ consisting of  those $\widetilde m_{\sts \stt}^\lambda$ with at least one of $\sts$ or $\stt$ not permissible.
\item A (small) generating set for $\ker(\phi_\k)$ as an ideal in $A_\k^S$.
\item A cellular basis of  $Q_\k^S$  consisting of the image under $\phi_\k$ of Murphy basis elements
$m_{\sts \stt}^\lambda$ such that both $\sts$ and $\stt$ are permissible.
\end{enumerate}

In the applications,  points (2) and (3) of this list are two different versions of a second fundamental theorem of invariant theory, while point (4) shows that the classical centralizer algebras -- Brauer's centralizer algebras on orthogonal or symplectic tensor space, or the image of the  
walled Brauer algebras on mixed tensor space -- are cellular algebras over the integers. 

 \section{Supplements on quotient towers}
 In this section, we provide some supplementary material on quotient towers.  This material is not strictly needed to appreciate the applications in the subsequent sections, so it could be safely skipped on the first reading.
  
 \subsection{Quotient towers are themselves towers of diagram algebras}

 In this section, we show that the tower  $(Q_\k^S)_{\k \ge 0}$  is  restriction coherent, and that the $d$--branching factors associated to restrictions of cell modules in this tower are just those obtained by applying $\phi$ to the $d$--branching factors of the tower $(A_\k^S)_{\k \ge 0}$.  It follows  that the 
 tower  $(Q_\k^S)_{\k \ge 0}$ satisfies all of the axioms 
 \eqref{diagram 1} to \eqref{diagram 6} of 
 \myref{Section}{subsection diagram algebras},
 with the possible exception of axiom   \eqref{diagram semisimplicity}.   If we assume that the quotient algebras
 $Q_\k^\BbbK$ are split semisimple -- and this will be valid in our applications -- then all the consequences of  
  \eqref{diagram 1} to \eqref{diagram 6} are available to us;  see 
  \myref{Section}{section some consequences from BEG} 
 and ~\cite{BEG}.

First we demonstrate that the tower of cellular algebras $(Q_\k^S)_{\k\ge 0}$ is restriction coherent.   
  We write $Q_k$ for $Q_k^S$.   Write $\widehat Q_\k$ for $\widehat A_{\k, \perm}$.    We have the branching diagram
$\widehat Q = \bigsqcup_\k \widehat Q_\k$, with the branching rule $\lambda \to \mu$ for $\lambda  \in \widehat{Q}_{r-1}$ and $\mu    \in \widehat{Q}_r$ if and only if $\lambda \to \mu$ in $\widehat A$.   For $\nu \in \widehat Q_\k$,  the set of $\mu \in \widehat  Q_{\k-1}$ such that
$\mu \to \nu$ is totally ordered, because it is a subset of  the set of $\mu \in \widehat  A_{\k-1}$ such that
$\mu \to \nu$.  For $\nu \in \widehat Q_\k$ let $\Delta_{Q_\k}(\nu)$ denote the corresponding cell module of $Q_\k$,
$$
 \Delta_{Q_\k}(\nu) = \Span_R \leftbrace  \phi(\M_\nu d_\mft) +  Q_\k^{\rhd \nu} \suchthat  \mft \in \Std_{\k, \perm}(\nu)             \rightbrace.
$$

\begin{lem} \label{coherence of Q tower 1}    Let  $r \ge 1$, let $\nu \in \widehat Q_r$ and $\mu \in \widehat Q_{r-1}$ with $\mu \to \nu$.   Let $u = u_{\mu \to \nu}$.
Let $x \in \phi(\M_\mu) Q_{r-1} \cap Q_{r-1}^{\rhd \mu}$.   Then
$$
\phi(u^*) x \equiv \sum_\mfz  r_\mfz \phi(\M_\nu d_\mfz)  \mod Q_r^{\rhd \nu},
$$
where the sum is over $\mfz \in \Std_{r, \perm}(\nu)$ such that $\mfz(r-1) \rhd \mu$.  
\end{lem}

\begin{proof}  Since $x \in \phi(\M_\mu) Q_{r-1} $,  there exists $b \in A_{r-1}^S$ such that $x = \phi(\M_\mu b)$.   Using 
\myref{Lemma}{good paths lemma 1},
we can write
$$
\M_\mu b =   
\smashr{\sum_{\mfs \in \Std_{r-1, \perm}(\mu)} }\ 
 r_\mfs \M_\mu d_\mfs + y_1 + y_2, 
$$
where $y_1 \in \ker(\phi)$  and $y_2 \in (A_{r-1}^S)^{\rhd \mu} \cap \M_\mu A_{r-1}^S$.   But since
$x = \phi(\M_\mu b) \in Q_{r-1}^{\rhd \mu}$, all the coefficients $r_s$ are zero and 
$$
\M_\mu b =   y_1 + y_2
$$
Now apply 
\myref{Lemma}{path basis dominance lemma}
to $y_2$,  
$$
u^* \M_\mu b =  u^* y_1 +  \sum_\mfv  \alpha_\mfv \M_\nu d_\mfv + y_3, 
$$
where the sum is over $\mfv \in \Std_{r}(\nu)$ such that  $\mfv(r-1) \rhd \mu$, and $y_3 \in (A_r^S)^{\rhd \nu}$.       The $\mfv$ appearing in the sum may not be permissible,  but we can apply 
\myref{Lemma}{good paths lemma 1}
to any term $\M_\nu d_\mfv$ with $\mfv$ not permissible, to replace it with a linear combination of terms $\M_\nu d_\mfz$, modulo
$(\ker(\phi)  +  (A_r^S)^{\rhd \nu})$, where  $\mfz \in \Std_r(\nu)$ is permissible and satisfies $\mfz \succ \mfv$.  But if $\mfz \succ \mfv$, then $\mfz(r-1) \unrhd \mfv(r-1) \rhd \mu$.
\end{proof}

\begin{cor}   \label{coherence of Q tower 2}   Let $r \ge 1$,  $\nu \in \widehat Q_r$, $\mft \in \Std_{r, \perm}(\nu)$,  and $a \in Q_{r-1}$.    Write $\mu = \mft(r-1)$  and $\mft' =  \mft_{[0, r-1]}$.  Suppose
$$
\phi(\M_\mu d_{\mft'}) a \equiv  
\smashr{\sum_{\mfs \in \Std_{r-1, \perm}(\mu)}  }\ 
r_\mfs \phi(\M_\mu d_\mfs)   \mod  Q_{r-1}^{\rhd \mu}.
$$ 
Then
$$
\phi(\M_\nu d_\mft) a  \equiv  
\smashr{\sum_{\mfs \in \Std_{r-1, \perm}(\mu)}} \ 
 r_\mfs \phi(\M_\nu d_{ \mu \to \nu}
d_{\mfs} )+  
\sum_\mfz  r_\mfz \phi( \M_\nu d_\mfz )  \mod  Q_r^{\rhd \nu},
$$
where the sum is over $\mfz \in \Std_{r, \perm}(\nu)$ such that $\mfz(r-1) \rhd \mu$.  
\end{cor}

\begin{proof}  Write 
$$
 \phi(\M_\mu d_{\mft'}) a \equiv  
\smashr{ \sum_{\mfs \in \Std_{r-1, \perm}(\mu)}  }\ 
 r_\mfs \phi(\M_\mu d_\mfs)  + y, 
$$ 
where $y \in Q_{r-1}^{\rhd \mu} \cap \phi(\M_\mu) Q_{r-1}$.   Multiply both sides  on the left by $\phi(u^*)$,  where 
$u = u_{\mu \to \nu}$, and apply 
\myref{Lemma}{coherence of Q tower 1}
to $\phi(u^*) y$.  
\end{proof}

\begin{prop}  \label{coherence of Q tower 3} 
  Let $r \ge 1$, let $\nu \in \widehat Q_r$, and let 
$\mu(1) \rhd \mu(2) \rhd \cdots \rhd \mu(s)$  be the list of $\mu \in \widehat Q_{r-1}$ such that $\mu \to \nu$.    Define
$$
M_j = \Span_R \leftbrace  \phi(\M_\nu d_\mft) +  Q_r^{\rhd \nu} \suchthat  \mft \in \Std_{r, \perm}(\nu), \mft(r-1) \unrhd \mu(j)                         \rightbrace.
$$
and $M_0 = (0)$.   Then
$$
M_0 \subseteq M_1 \subseteq \cdots  \subseteq M_s = \Delta_{Q_r}(\nu)
$$
is a filtration of $ \Delta_{Q_r}(\nu)$ by $Q_{r-1}$ submodules, and $M_j/M_{j-1} \cong \Delta_{Q_{r-1}}(\mu(j))$.  
\end{prop}

\begin{proof} Immediate from  \myref{Corollary}{coherence of Q tower 2}.
\end{proof}

\begin{cor} \label{coherence of Q tower 4}   The branching factors associated to the filtrations in 
\myref{Proposition}{coherence of Q tower 3}
 are 
$\phi(d_{\mu \to \nu})$  for $\mu \in \widehat Q_{r-1}$ and   $\nu \in \widehat Q_{r}$ with $\mu \to \nu$.
\end{cor}

\begin{proof}  The isomorphism 
$
M_j/M_{j-1} \to \cell {Q_{r-1}} {} {\mu(j)}
$
is
$$
\phi(\M_\nu d_{\mu(j) \to \nu} d_\mfs) + M_{j-1} + Q_r^{\rhd \nu} \mapsto
\phi(\M_{\mu(j)} d_\mfs) + Q_{r-1}^{\rhd \mu(j)}.
$$
\end{proof}

\begin{prop}   \label{coherence of Q tower 5}
Assume \eqref{quotient axiom 1}--\eqref{quotient axiom 3}, and that $Q_r^\BbbK$ is split semisimple for all $r$.   It follows that the tower $(Q_r^S)_{r \ge 0}$  satisfy axioms \eqref{diagram 1}--\eqref{diagram 6} of \myref{Section}{subsection diagram algebras}.  
\end{prop}

\begin{proof}  One only has to observe that the branching coefficients satisfy
$$
\phi(\M_\nu) \phi(d_{\mu \to \nu}) =  \phi(u_{\mu \to \nu})^*  \phi(\M_\mu).
$$
\end{proof}

\begin{rmk} \label{coherence of Q tower 6}  The assumption that $Q_r^\BbbK$ is split semisimple for all $r$
implies that $\cell {Q_{r-1}} {S} \mu$ appears as a subquotient in a cell filtration of $\Res(\cell {Q_r} {S} \nu)$  if and only if the simple  $Q_{r-1}^\BbbK$--module $\cell {Q_{r-1}} {K} \mu$  is a direct summand of 
$\Res(\cell {Q_r} {K} \nu)$, if and only if  $\mu \to \nu$ in $\widehat Q$.  See ~\cite[Lemma 2.2]{MR2794027}.
\end{rmk}

\subsection{Jucys--Murphy elements in quotient towers}  \label{subsection JM elements quotients}
   Consider a sequence of cellular algebras $(A_\k)_{\k \ge 0}$ over  an integral domain $R$ with field of fractions $\FF$,  satisfying the properties
\eqref{diagram 1}--\eqref{diagram 6}   of 
\myref{Section}{subsection diagram algebras}.  In ~\cite[Definition 4.1]{BEG}, following ~\cite{MR2774622}, one defines a sequence $\{L_i\}_{i \ge 1}$  of {\em additive Jucys--Murphy elements}  by the two conditions:
\begin{enumerate}[leftmargin=*,label=(JM\arabic{*}), ref=JM\arabic{*},  series = JM conditions]
\item \label{cond JM1} For $r \ge 1$, $L_r \in A_r^R$,  $L_r = L_r^*$, and $L_r$ commutes pointwise with $A_{r-1}^R$.
\item  \label{cond JM2} For all $r \ge 1$ and   $\la \in\widehat{A}_r$,   there exists  $d(\la) \in R$  such that    $ L_1 +  \cdots + L_r$   acts as the scalar   $d(\la)$  on the cell module $\Delta^R_r(\la)$.
\end{enumerate}
Using that $A_r^\FF$ is split semisimple for all $r$, condition \eqref{cond JM2}  is equivalent to 
\begin{enumerate}[leftmargin=*,label=(JM\arabic{*}), ref=JM\arabic{*},  resume = JM conditions]
\item \label{cond JM3} For  all $r \ge 1$, $ L_1 +  \cdots + L_r$  is in the center of $A_r^R$.
\end{enumerate}
For each edge  $\la \to \mu$ in the branching diagram $\hat A$ for $(A_r^\FF)_{r \ge 0}$, write 
 $\kappa(\la \to \mu) = d(\mu) - d(\la) \in R$, where by convention $d(\varnothing) = 0$.  For  a path $\mft \in \Std_r$ and $i \le r$, write
 $ \kappa_\mft(i) $ for $\kappa(\mft(i-1) \to \mft(i))$.
 The elements  $\kappa(\la \to \mu)$  or  $ \kappa_\mft(i) $  are called {\sf contents} since they generalize the contents of standard tableaux in the theory of the symmetric group.
It is shown in ~\cite[Section 4]{BEG}, strengthening  results of ~\cite{MR2774622}, that 
\begin{enumerate}[leftmargin=*,label=(JM\arabic{*}), ref=JM\arabic{*},  resume = JM conditions]
\item  \label{cond JM4} $f^\la_\mft L_i = \kappa_\mft(i) f_\mft$  for all paths $\mft \in \Std_r(\la)$ and for $i \le r$, and
\item \label{cond JM5} $m^\la_\mft L_i = \kappa_\mft(i) m^\la_\mft  + \sum_{\mfs \rhd \mft}  r_\mfs m^\la_\mfs$, for some coefficients $r_\mfs \in R$.
\end{enumerate}
  Condition  \eqref{cond JM5}  is an instance of Mathas's abstraction of Jucys--Murphy elements from ~\cite{MR2414949}.  We note that conditions \eqref{cond JM4} and \eqref{cond JM5}  do not depend on Mathas's separation condition being satisfied, as in ~\cite[Section 3]{MR2414949}.

Suppose we are given an additive sequence of JM elements.    Then
 conditions  \eqref{cond JM1}, \eqref{cond JM2}, \eqref{cond JM3}, and \eqref{cond JM5}
  remain valid in any specialization $A_r^S = A_r^R \otimes_R S$, where $L_i$ is  replaced by $L_i \otimes 1_S$ and $d(\la)$ and $\kappa(\la \to \mu)$ by their images in $S$, so in particular every specialization has JM elements in the sense of Mathas.

 Now, finally, suppose the hypotheses \eqref{quotient axiom 1}--\eqref{quotient axiom 3} are satisfied and that   the quotient  algebras $Q_r^\BbbK$ are split semisimple, so that the quotient tower $(Q_r^S)_{r \ge 0}$  is a sequence of diagram algebras  satisfying the properties
\eqref{diagram 1}--\eqref{diagram 6}.  
 Clearly, the defining conditions \eqref{cond JM1} and \eqref{cond JM2}  for JM elements are satisfied, with $L_i$ replaced by $\phi(L_i \otimes 1_S)$, and \eqref{cond JM3}   follows.     (Of course, versions of \eqref{cond JM4} and \eqref{cond JM5} must hold as well, but this is not very useful in this generality as we cannot relate the seminormal bases of the quotient tower with that of the original tower.  This defect is removed in \myref{Section}{subsection seminormal quotient}.)

 This discussion holds just as well for {\em multiplicative Jucys--Murphy elements}, see ~\cite[Definition 4.3]{BEG} and  ~\cite{MR2774622}.

\subsection{Seminormal bases of quotient towers}  \label{subsection seminormal quotient}
In this section we examine seminormal bases and seminormal representations in quotient towers.  We work in the following setting:  we assume the setting of 
\myref{Section}{subsection setting quotient towers}, in particular that  \eqref{quotient axiom 1}--\eqref{quotient axiom 3} are satisfied.   We assume this existence of additive or multiplicative JM elements for the tower $(A_r^R)_{r \ge 0}$, as in \myref{Section}{subsection JM elements quotients}; in particular, conditions \eqref{cond JM4}  and \eqref{cond JM5}  of \myref{Section}{subsection JM elements quotients} hold.  We assume, moreover, that the separation condition of Mathas holds; that is if $\mfs$ and $\mft$ are distinct paths in $\Std_r$, then there exists $i \le r$ such that $\kappa_\mft(i) \ne \kappa_\mfs(i)$.  We assume that the quotient
algebras $Q_r^\BbbK$ are split semisimple, so that the quotient tower $(Q_r^S)_{r \ge 0}$ is a tower of diagram algebras satisfying  \eqref{diagram 1}--\eqref{diagram 6}.  As a tower of diagram algebras, $(Q_r^S)_{r \ge 0}$ has its own seminormal bases.  Finally, we assume the following condition, which will   allow us to connect these seminormal bases  to those of the original tower:

\begin{enumerate}[label = (SN), ref = SN, leftmargin=*]
\item  \label{assumption SN} Whenever $\mft \in \Std_{r, \perm}$, it follows that
$F_\mft$ is evaluable.
\end{enumerate}

It follows from \eqref{assumption SN}   that if $\mfs, \mft \in \Std_r(\la)$ are permissible, then $f^\la_\mft$  and $F^\la_{\mfs \mft}$ are evaluable, and also that $\langle f^\la_\mfs, f^\la_\mft \rangle \in R_{\mathfrak p}$.

We are interested in verifying these assumptions, and in particular condition \eqref{assumption SN},  for quotients of diagram algebras acting on tensor spaces.
The following lemma provides a sufficient condition for \eqref{assumption SN} to hold.  Associate to each path $\mft \in \Std_r$ its  {\em content} sequence $\kappa_\mft(i)$  for ${1\le i \le r}$ and its {\em residue} sequence $r_\mft(i) = \pi(\kappa_\mft(i))$.    Say two paths $\mfs, \mft$ are {\em residue equivalent}, and write $\mft \approx \mfs$,  if they have the same residue sequences. 

\begin{lem}  \label{suff cond for SN}
Suppose that each permissible path in $\Std_r$ is residue equivalent only to itself.  Then \eqref{assumption SN} holds.
\end{lem}

\begin{proof}  This follows from \cite[Lemma~4.2]{MR2414949}.  It is also easy to prove directly using the following recursive  formula for the idempotents $F_\stt$.  For $\stt \in \Std_{k+1}$,  let $\stt'$ denote the truncation $\stt' = \stt_{[0, k]}$.  Then
 \begin{equation*} \label{interpolation}
 F_\stt = F_{\stt'} \prod_{\begin{subarray} c  \sts \ne \stt \\  \sts' = \stt' \end{subarray}}  
 \frac{L_{r} - \kappa_\sts(r)}{\kappa_\stt(r)-\kappa_\sts(r)}.
 \end{equation*}
\end{proof}

\begin{rmk} \label{remark: permissible paths residue class}   \label{suff cond for SN 1}
 We will use the hypothesis of  \myref{Lemma}{suff cond for SN} in the following equivalent form: For all $r \ge 1$ and for  all $\mft, \mfs \in \Std_{r}$   with $\mfs' = \mft'$, if  at least one of $\mfs, \mft$  is permissible,  then
$\pi(\kappa_t(r)) \ne \pi(\kappa_s(r))$. 
\end{rmk}

\begin{eg}  In the case of the symmetric group algebras acting on ordinary or mixed tensor space,  no specialization of the ground ring is involved, so condition \eqref{assumption SN} is vacuous.   The sufficient condition of \myref{Lemma}{suff cond for SN}  is verified for the Brauer algebras acting on symplectic tensor space in \myref{Section}{symplectic Brauer seminormal}, and for the walled Brauer algebras in 
 \iftoggle{arxiv}
 {\myref{Appendix}  {subsection wba seminormal}   }
 {Appendix C.4 in the arXiv version of this paper}.  It is verified for the partition algebras, with an appropriate permissibility condition, in  \cite[Section 6]{BDE}.   For the Brauer algebra acting on orthogonal tensor space, the sufficient condition of \myref{Lemma}{suff cond for SN} holds for odd integers values, but fails for even integer values; in the latter case   a more subtle argument is needed in order to verify \eqref{assumption SN}; for this, see  \cite{EG:2017}.
\end{eg}

\begin{lem} Let $a \in A_r^R$ and let $a(\mfs, \mft)$ denote the matrix coefficients of $a$ with respect to the seminormal basis $\{f_\mft^\la\}$ of $\Delta_{A_r}^\FF(\la)$, 
$$
f^\la_\mft a = \sum a(\mfs, \mft)  f^\la_\mfs.
$$   
If $\mfs, \mft$ are permissible paths in $\Std_r(\lambda)$,  then $a(\mfs, \mft) \in R_{\mathfrak p}$.
\end{lem}

\begin{proof}  We have $f^\la_\mft a F_\mfs =  a(\mfs, \mft) f^\la_\mfs = a(\mfs, \mft)  m^\la_\mfs + \sum_{\mfv \rhd \mfs}  \gamma_\mfv m^\la_\mfv$, using \myref{Theorem}{dominance triangularity}.
By assumption (SN),  the element on the left side of the equation is evaluable, and hence the coefficients on the right side  lie in $R_{\mathfrak p}$.
\end{proof}

We will now construct a cellular basis $\{h^\la_{\mfs \mft} \suchthat \la \in \hat A_r \text{ and } \mfs, \mft \in \Std_r(\la)\}$  with the properties that 
\begin{enumerate}
\item If both $\mfs, \mft$ are permissible paths, then $h^\la_{\mfs \mft} =  \pi(F^\la_{\mfs \mft})$.
\item If at least one of $\mfs, \mft$ is not permissible, then  $h^\la_{\mfs \mft} \in \ker(\phi)$.
\end{enumerate}
If $\mft \in \Std_r$ is permissible, define $[\mft] = \{\mft\}$.   If $\mft$ is not permissible, then 
let  $[\mft]$ denote the set of paths $\mfs \in \Std_r$ such that   $\mfs \approx \mft$. 
 By assumption (SN) and \cite[Lemma~4.2]{MR2414949}, 
 $F_{[\mft]} :=  \sum_{\mfs \in [\mft]} F_\mfs$ is an evaluable idempotent.   
For any $\mfs, \mft \in \Std_r(\lambda)$, permissible or not, define $h^\la_\mft =  \nn^\la_\mft \pi(F_{[\stt]})$, 
and $h^\la_{\mfs \mft} =  \pi(F_{[\mfs]}) \nn^\la_{\mfs \mft} \pi(F_{[\mft]})$.

\begin{lem} \label{lemma h basis triangular transition 1}
For $\la \in \hat A_r$ and $\mft \in \Std_r(\la)$  there exist coefficients $\beta_\mfs \in \BbbK$ such that
\begin{align}  \label{unitriangular transition 3}
h^\la_\mft=
m^\la_\mft
+\sum_{\substack{\mfs\succ\mft}} \beta_\mfs m^\la_\mfs.
\end{align}
\end{lem}

\begin{proof}  We know from  Equation \eqref{eqn n-basis 3} that  the basis elements $\nn^\lambda_\stt $
of $\Delta_{A_r}^S(\lambda)$ 
 are related to  the basis elements  
$m^\lambda_\stt $  
by a unitriangular transformation with coefficients in $S$, 
$$
\nn^\la_\mft =  m^\la_\mft   +\sum_{\substack{\mfs\succ\mft}} \alpha_\mfs m^\la_\mfs.
$$
Lift this relation to $\Delta_{A_r}^R(\lambda)$, defining $\nn^\la_\mft $ by
$
\nn^\la_\mft =  m^\la_\mft   +\sum_{\substack{\mfs\succ\mft}} \alpha'_\mfs m^\la_\mfs
$,
where $\alpha'_\mfs  \in R$ and $\pi(\alpha'_\mfs) = \alpha_\mfs$.  Apply  \myref{Theorem}{dominance triangularity}  to get 
$
\nn^\la_\mft =  f^\la_\mft   +\sum_{\substack{\mfs\succ\mft}}
\gamma_\mfs f^\la_\mfs,
$
where the coefficients are now in $\FF$.   Applying  $F_{[\stt]}$  yields
$
\nn^\la_\mft  F_{[\stt]} =  f^\la_\mft   +\sum_{\begin{subarray}c
\mfs\succ\mft ,  \mfs \in  [\mft]
\end{subarray}}
\gamma_\mfs f^\la_\mfs,  
$
and using \myref{Theorem}{dominance triangularity}  again produces
$$
\nn^\la_\mft  F_{[\stt]} =  m^\la_\mft   +\sum_{\mfs \succ \mft} \beta'_\mfs m^\la_\mfs.
$$
The coefficients are {\em a priori} in $\FF$, but since $\nn^\la_\mft  F_{[\stt]}$ is evaluable, the coefficients are  evaluable.   Finally applying $\pi$ yields  \eqref{unitriangular transition 3}, with $\beta_\mfs = \pi(\beta'_\mfs)$.
\end{proof}

 \begin{cor}   \label{seminormal basis lemma 1} \mbox{}
\begin{enumerate}[leftmargin=*,label=(\arabic{*}), font=\normalfont, align=left, leftmargin=*]
\item  For $\la \in \hat A_r$, 
 $\{h^\la_\mft \suchthat \mft \in \Std_r(\lambda)\}$   is a basis of the cell module $\Delta_{A_r}^\BbbK(\lambda)$ and  $\{h^\la_{\mfs \mft} \suchthat  \la \in \hat A_r \text{ and } \mfs, \mft \in \Std_r(\la)  \}$ is a cellular basis of $A_r^\BbbK$ equivalent to the Murphy basis. 
 \item
 If both $\mfs, \mft$ are permissible paths, then $h^\la_{\mfs \mft} =  \pi(F^\la_{\mfs \mft})$.   If at least one of $\mfs, \mft$ is not permissible, then  $h^\la_{\mfs \mft} \in \ker(\phi)$.
 \end{enumerate}
\end{cor}
 
\begin{proof}  The first assertion follows from  \myref{Lemma}{lemma h basis triangular transition 1} by  familiar argument, compare \myref{Theorem}{basis theorem}.   The second assertion is evident from the construction and properties of $\{\nn^\la_{\mfs \mft}\}$.  
\end{proof}

When $\mfs, \mft \in \Std_r(\la)$ are permissible, 
define $\bar F^\la_{\mfs \mft} =  \phi\circ\pi(F^\la_{\mfs \mft}) = \phi(h^\la_{\mfs \mft})$  and   $\bar F_\mft =  \phi\circ\pi(F_\mft)$.

\begin{cor}\label{seminormal basis lemma 2} \mbox{}
\begin{enumerate}[leftmargin=*,label=(\arabic{*}), font=\normalfont, align=left, leftmargin=*]
\item  The set $\{ \bar F^\la_{\mfs \mft} \suchthat \la \text{ is permissible and }  \mfs, \mft \in \Std_{r, \perm}(\la)\}$  is a basis of $Q_r^\BbbK$.
\item  The set 
$\{h^\lambda_{\sts \stt} \suchthat  \lambda \in \widehat A_\f \text{ and } \sts  \text{ or  } \stt \text{ is not permissible}\}$ is a $\BbbK$--basis of  $\ker(\phi_r)$.  
\end{enumerate}
\end{cor}

\begin{proof}  Adapt the proof of \myref{Theorem}{good paths basis theorem}  parts \ref{good paths 2} and
\ref{good paths 3}. 
\end{proof}

In \myref{Corollary}{good paths cor}, for $\lambda \in \hat A_r$ permissible, we identified $\Delta_{Q_r}^\BbbK(\lambda)$ with the simple head of $\Delta_{A_r}^\BbbK(\lambda)$, and we showed that the radical of $\Delta_{A_r}^\BbbK(\lambda)$ is  $\kappa^\BbbK(\la) := \Span_\BbbK\{ \nn^\lambda_\stt \suchthat \stt \text{ is not permissible} \}$.   Overloading notation, let us write $\phi$ for the quotient map
$\phi: \Delta_{A_r}^\BbbK(\lambda) \to \Delta_{A_r}^\BbbK(\lambda)/\kappa^\BbbK(\lambda)$.   

For $\mft \in \Std_r(\lambda)$ permissible, write $\bar f^\la_\mft = \phi(h^\la_\mft) = \phi(\nn^\la_\mft \pi(F_\mft))$.  

\begin{cor}\label{seminormal basis lemma 3} \mbox{}
\begin{enumerate}[leftmargin=*,label=(\arabic{*}), font=\normalfont, align=left, leftmargin=*]
\item  For $\la \in \hat A_r$ permissible,  $\{\bar f^\la_\mft \suchthat \mft \text{ is permissible}\}$ is a basis
of $\Delta_{Q_r}^\BbbK(\la)$. 
\item The set $\{h^\la_\mft \suchthat \mft \text{ is  not permissible}\}$ is a basis of $\rad(\Delta_{A_r}^\BbbK(\la))$.
\end{enumerate}
\end{cor}

\begin{proof}  Write $\Delta$ for $\Delta_{A_r}^\BbbK(\la)$.  For $\stt$ not permissible,
we have $h^\la_\mft = \nn^\la_\mft \pi(F_{[\mft]})  \in \rad(\Delta)$ because $\nn^\la_\mft  \in \rad(\Delta)$ by \myref{Corollary}{good paths cor}, and $\rad(\Delta)$ is a submodule.  The two conclusions follow by a dimension argument as in the proof of
 \myref{Theorem}{good paths basis theorem}  parts \ref{good paths 2} and
\ref{good paths 3}. 
\end{proof}

The following result says that the seminormal representations of the quotient algebras $Q_r^\BbbK$ are
obtained as truncations of the seminormal representations of the diagram algebras $A_r^\FF$.   The 
application of this result to the Brauer algebras and their quotients acting on orthogonal or symplectic 
tensor space recovers  a well-known phenomenon, which is implicit in \cite{MR1398116, MR1427801, MR1144939}, and explicit in ~\cite[Theorem 5.4.3]{MR3581213}.  See also \cite{EG:2017}.

\begin{thm}\label{theorem quotient seminormal}       \mbox{}
\begin{enumerate}[leftmargin=*,label=(\arabic{*}), font=\normalfont, align=left, leftmargin=*, series = quotient sn]
\item  \label{quotient sn 1}
The family of idempotents $\bar F_\mft  = \phi\circ\pi(F_\mft)$, where $r \ge 1$ and  $\mft \in \Std_r$ is permissible, 
  is the family of Gelfand-Zeitlin idempotents for the tower $Q_r^\BbbK$.
\end{enumerate}
In the following statements, let $\la$ and $\mu$ be permissible points in $\hat A_r$ for some $r$,  and let
$\mfs, \mft \in \Std_{r, \perm}(\la)$, and $\mfu, \mfv \in \Std_{r, \perm}(\mu)$.
\begin{enumerate}[leftmargin=*,label=(\arabic{*}), font=\normalfont, align=left, leftmargin=*, resume = quotient sn]
\item  \label{quotient sn 2}
 $\bar f^\la_\mft \bar F^\mu_{\mfu \mfv} = \delta_{\la, \mu} \delta_{\mft, \mfu}\pi(\langle  f^\la_\mft,   f^\la_\mft \rangle)  \bar f^\la_\mfv$,  \ \ and  \ \
 $\bar F^\la_{\mfs \mft} \bar F^\mu_{\mfu \mfv} = \delta_{\la, \mu} \delta_{\mft, \mfu}\pi(\langle f^\la_\mft,  f^\la_\mft \rangle ) \bar F^\la_{\mfs \mfv}$.  
 \item  \label{quotient sn 3}  $\pi(\langle f^\la_\mft, f^\la_\mfs\rangle) \ne 0$ if and only if $\mft = \mfs$. 
  \item  \label{quotient sn 4} $ \pi(\langle f^\la_\mft, f^\la_\mft \rangle)\inv  \bar F^\la_{\mft \mft} = \bar F_\mft$
    \item  \label{quotient sn 5}  
    The set of elements $\bar E^\la_{\mfs \mft} = \pi(\langle f^\la_\mfs,  f^\la_\mfs \rangle)\inv  \bar F^\la_{\mfs \mft} $ for $\la \in \widehat A_r$ and $\mfs, \mft \in \Std_{r, \perm}(\la)$  is a complete family of matrix units with
 $
 \bar E^\la_{\mfs \mft} \bar E^\mu_{\mfu \mfv} = \delta_{\la, \mu} \delta_{\mft, \mfu}  \bar E^\la_{\mfs \mfv}$,  and 
$\bar E^\la_{\mfs \mfs}  =  \bar F_\mfs$.
 \item  \label{quotient sn 6}  
 Suppose that $a \in A_r^R$ has matrix coefficients $a(\mfs,\mft)$ with respect to the seminormal basis  $\{f_\mft^\la\}$ of $\Delta_{A_r}^\FF(\la)$, 
$$
f^\la_\mft a = \sum_{\mfs \in \Std_r(\la)}  a(\mfs, \mft)  f^\la_\mfs. 
$$  
Then for $\mft \in \Std_r(\la)$ permissible, we have
\begin{equation} \label{eqn: quotient sn 1}
\bar f^\la_\mft \phi\circ \pi(a) = 
\smashr{\sum_{\mfs \in \Std_{r, \perm}(\la )} }\  
\pi(a(\mfs, \mft) ) \bar f^\la_\mfs. 
\end{equation}

\end{enumerate}
\end{thm}

\begin{proof}
We remark that if $\stt$ is permissible, then 
  statement \ref{quotient sn 2} follows from the definitions and the corresponding properties of the elements $f^\la_\mft$ in $\Delta_{A_r}^\FF(\la)$  and $F^\mu_{\mfu \mfv}$ in $A_r^\FF$, see  ~\cite[Lemma 3.8]{BEG}.  
  (This follows because  $[\stt]=\{\stt\}$ by definition, when  $\stt$ is permissible.)
  We already know from ~\cite[Lemma 3.8]{BEG} that $\langle f^\la_\mft, f^\la_\mfs\rangle = 0$ if $\mfs \ne \mft$ and also that for all $\mft$,  $F^\la_{\mft \mft} = \langle f^\la_\mft, f^\la_\mft \rangle F_\mft$.      It follows that 
$\bar F^\la_{\mft \mft} = \pi(\langle f^\la_\mft, f^\la_\mft \rangle) \bar F_\mft$.   If for some $\mft$, 
$ \pi(\langle f^\la_\mft, f^\la_\mft \rangle) = 0$, then $\bar F^\la_{\mft \mft} = 0$, contradicting \myref{Corollary}{seminormal basis lemma 2}.   This proves points \ref{quotient sn 3} and  \ref{quotient sn 4} and point \ref{quotient sn 5} also follows from the previous statements.    

For $\mft \in \Std_{r, \perm}(\la)$ and $\mfv \in \Std_{r, \perm}(\mu)$, we have 
 $\bar f^\mu_\mfv  \bar F_\mft = \delta_{\la, \mu} \delta_{\mfv, \mft} \bar f^\mu_\mfv$,   and it follows that
 $\sum_{\mft \in \Std_{r, \perm}(\la)} \bar F_\mft$ is the minimal central idempotent in $Q_r^\BbbK$ corresponding to the simple module $\Delta_{Q_r}^\BbbK(\la)$.    
 Let  $s \le r$, $\mfs \in  \Std_{s, \perm}$, $\mft \in \Std_{r, \perm}$.
We have $F_\mfs F_\mft = \delta_{ \mfs,\stt_{[0, s]}}F_\mft$,  in $A_r^\FF$ from the definition of the Gelfand-Zeitlin idempotents.   But then the corresponding property  
$\bar F_\mfs  \bar F_\mft = \delta_{ \mfs,\stt_{[0, s]}}\bar F_\mft$ holds in $Q_r^\BbbK$.    By ~\cite[Lemma 3.10]{MR2774622}, these properties characterize the family of Gelfand-Zeitlin idempotents, so point \ref{quotient sn 1} holds.

 For point \ref{quotient sn 6}, suppose that $a \in A_r^R$ and that $a$ has matrix coefficients $a(\mfs, \mft)$ with respect to the  seminormal basis  $\{f_\mft^\la\}$ of $\Delta_{A_r}^\FF(\la)$.   Then when $\mft$ and $\mfs$ are both permissible elements of $\Std_r(\la)$, we have $m^\la_\mft F_\mft a F_\mfs = a(\mfs, \mft) m^\la_\mfs F_\mfs$.  As this equality involves evaluable elements, we can apply $\phi\circ \pi$ to get
 $\bar f^\la_\mft \phi\circ\pi(a) \bar F_\mfs =  \pi(a(\mfs, \mft))\bar f^\la_\mfs $. Now sum over $\mfs$ and use that  $\sum_{\mfs \in \Std_{r, \perm}(\la)} \bar F_\mfs$ acts as the identity on the cell module
 $\Delta_{Q_r}^\BbbK(\la)$.  This yields \eqref{eqn: quotient sn 1}.
\end{proof}

\section{The   Murphy and dual Murphy bases of the symmetric groups}
\label{dual Murphy basis of S n}\label{subsection Murphy basis of S n}\label{section Murphy symm}

   A \emph{partition} $\lambda$ of $r$, denoted $\lambda\vdash r$, is defined to be a weakly decreasing  sequence $\lambda=(\lambda_1,\lambda_2,\dots,\lambda_\ell)$ of non-negative integers such that the sum $|\lambda|=\lambda_1+\lambda_2+\dots +\lambda_\ell$ equals $r$.   
   Let $\widehat{\mathfrak S}_r$ denote the set of all partitions of $r$.
   With a partition, $\lambda$, is associated its \emph{Young diagram}, which is the set of nodes
\[[\lambda]=\left\{(i,j)\in\mathbb{Z}_{>0}^2\ \left|\ j\leq \lambda_i\right.\right\}.\]   We identify partitions with their Young diagrams. 
There is a unique partition of size zero, which we denote $\varnothing$.   We let $\lambda'$ denote the \emph{conjugate partition} obtained by flipping the Young diagram $[\lambda]$ through the diagonal.  
Let $\lambda,\mu \in   \widehat{\mathfrak{S}}_r$,  we say that $\lambda$
dominates $\mu$ and write $\lambda \unrhd  \mu$ if, for all $1\leq k\leq r$, we have 
$$  \sum_{i=1}^k \lambda_i  \geq  
\sum_{i=1}^k \mu_i .$$
Define column dominance order, denoted by $\coldomeq$,  by $\la \coldomeq \mu$ if and only if $  \la' \unrhd \mu'$. It is known that column dominance order is actually the opposite order to dominance order.

 {\sf Young's graph} or {\sf  lattice},  $\widehat{\mathfrak S}$, is the branching diagram with vertices 
$\widehat{\mathfrak S}_r$
on level $r$ and a directed edge $\lambda \to \mu$  if $\mu$ is obtained from $\lambda$ by adding one box.   We define a {\sf standard tableau} of shape $\lambda$ to be a directed path  on $\widehat{\mathfrak S}$ from $\varnothing$ to $\lambda$.  (Such paths are easily identified with the usual picture of standard tableaux as fillings of the Young diagram of $\lambda$ with the numbers $1$ through $r$, so that the entries are increasing in rows and columns.) 
For $\lambda \vdash r$,  denote the set of standard tableaux of shape $\lambda$ by $\Std_r(\lambda)$.  
If $\mfs \in \Std_{r}(\lambda)$ is 
the path
 $$
 \varnothing= \sts(0)\to  \sts(1)\to  \sts(2)\to \dots \to  \sts(r)  = \lambda,
 $$
 then the conjugate standard tableaux
$\mfs' \in \Std_{ r}(\lambda')$  is 
the path
 $$
 \varnothing= \sts(0)\to  \sts(1)'\to  \sts(2)'\to \dots \to  \sts(r)'  = \lambda'.
 $$
 
For any ring $R$, and for all $r\geq 0$,   the group algebra $R \mathfrak S_r$ has an algebra involution determined by $w^*  = w\inv$ and an automorphism $\#$   determined by $w^\# = \sgn(w) w$  for $w \in \mathfrak S_r$.  The involution $*$,  the automorphism $\#$,  and the inclusions $R\mathfrak S_r \hookrightarrow R\mathfrak S_{r+1}$ are mutually commuting.  
Let $s_1, \dots, s_{r-1}$ be the usual generators of the symmetric group $\mathfrak S_r$,   $s_i = (i, i+1)$.
If $1 \le a \le i$, define
\begin{equation} \label{cyclic permutation}
s_{a, i} = s_a s_{a+1} \cdots s_{i-1} = (i, i-1, \dots, a)
\end{equation}
and
$
s_{i, a} = s_{a, i}^*  
$.  Therefore $s_{a, a} = 1$, the identity in the symmetric group.
 We let 
\begin{equation} \label{symmetrizer and antisymmetrizer}
    x_\lambda  = \sum_{w \in \mathfrak{S}_{\lambda}}   w
   \qquad
   y_\lambda=    \sum_{w \in \mathfrak{S}_{\lambda'}} \sgn(w)  w
\end{equation}
   where 
$\mathfrak{S}_{\lambda}=\mathfrak{S}_{\{1,2, \dots ,  \lambda_1\}} \times \mathfrak{S}_{\{\lambda_1+1 , \dots ,  \lambda_1+\lambda_2\}} \times \ldots$ is the Young subgroup labeled by $\lambda$ and $\mathfrak{S}_{\lambda'}$ is the Young subgroup labeled by $\lambda'$.  
 Given  $\mu \vdash i-1$ and $\lambda \vdash i$  with $\lambda=\mu\cup \{ j, \lambda_j\}$, we 
set $a=\sum_{r=1}^j \lambda_r$ and let 
  $b=\sum_{r=1}^{\lambda_j} \lambda'_r$.  
   We define the branching factors     as follows:
\begin{align}\label{branching factors for S n}
  {d}_{\mu \to \lambda}  = s_{a , i}  
\quad {u}_{\mu \rightarrow  \lambda  }= 
 s_{  i,a}  \sum_{ r=0}^{\mu_j} 
 s_{a,a-r}  \end{align} 
 and (conjugating and applying the automorphism 
 $\#$) we obtain the dual branching factors 
\begin{align} \label{dual branching coefficients for S n}
  {b} _{\mu \to \lambda}  = (-1)^{b-i}s_{b , i}  
\quad {v} _{\mu \rightarrow  \lambda  }= 
 s_{  i,b}  \sum_{ r=0}^{ j} 
(-1)^{r + b - i} s_{b,b-r}. \end{align}

For $ \lambda  \in \widehat{\mathfrak{S}}_r$ and $\stt \in \Std_r(\lambda )$  let $d_\stt$ be the ordered product of the $d$--branching factors along $\stt$ and let  $b_\stt$ be the ordered product of $b$--branching factors along $\stt$, i.e. $b_\stt =  (d_{\stt'})^\#$. 
Given $\sts,\stt \in \Std(\lambda)$ we let 
$$ x_{\sts\stt}  = d_\sts^\ast x_\lambda d_\stt
\qquad  y_{\sts\stt}  = \c_\sts^\ast y_\lambda \c_\stt.  
 $$

 \begin{thm}[\cite{EG:2012, MR1327362}]\label{thm Murphy basis of S n}
 The   algebra  $R\mathfrak{S}_r$ has cellular bases 
  $$ {\mathcal{X}}=
 \{  x_{\sts\stt}  \mid   \sts,\stt  \in {\rm Std}_r(\lambda ) \text{ for } \lambda \in\widehat{\mathfrak{S}}_r\}
\qquad
 {\mathcal{Y}}=
 \{  y_{\sts\stt}  \mid   \sts,\stt  \in {\rm Std}_r(\lambda ) \text{ for } \lambda \in \widehat{\mathfrak{S}}_r\}		 	$$
   with 
   the involution $\ast$  and the posets $(\widehat{\mathfrak{S}}_r , \unrhd)$ 
   and $(\widehat{\mathfrak{S}}_r , \unrhd_{\rm  col})$ respectively. These bases are the Murphy and dual-Murphy bases defined in \cite{MR1327362}.   
   \end{thm}
   
   \begin{proof}  It is shown in \cite[Corollary 4.8]{EG:2012} that $\mathcal X$ coincides with the Murphy cellular basis defined in \cite[Theorem 4.17]{MR1327362}.  Since $y_{\sts \stt} = (x_{\sts' \stt'})^\#$,  it follows that $\mathcal Y$ is also a cellular basis; the basis $\mathcal Y$ appears in 
   \cite{MR1327362} in a subsidiary role.
   \end{proof}

 It is shown in \cite[Section 4 and Appendix A]{EG:2012}, 
 following \cite{MR1327362,MR1461487,MR513828,MR812444,MR2531227},  that
the sequence  of symmetric group algebras $(\ZZ\mathfrak S_\k)_{\k \ge 0}$, endowed with the Murphy cellular bases,  satisfies axioms \eqref{diagram 1}--\eqref{diagram 6} of 
\myref{Section}{subsection diagram algebras}.
 In fact, the sequence is induction coherent, and the $u$--branching coefficients are those derived from cell filtrations of induced cell modules.   
The \cellgenerators \ are the elements $x_\lambda$.
 The branching diagram associated to the sequence is Young's lattice $\widehat{\mathfrak S}$.  The $d$-- and $u$--branching factors 
 satisfy the compatibility relations \eqref{abstract branching compatibility},  by
 \cite[Appendix A]{EG:2012}.  The corresponding results for the dual-Murphy basis follow by 
 conjugating and 
 applying the 
 automorphism $\#$.

\section{The Murphy and dual Murphy  bases of the Brauer algebra}\label{section: Murphy bases Brauer}

In this section we recall the definition of the Brauer algebra and the construction of its Murphy and dual Murphy bases.  
In subsequent sections, we will require the 
Murphy basis for examining Brauer algebras acting on symplectic tensor space,  
whereas we require the   
dual Murphy basis for examining Brauer algebras acting on orthogonal  tensor space.

{\sf An $\k$--strand  Brauer diagram} is a figure consisting of $\k$  points  on the top edge, and another $\k$ on the bottom edge of a rectangle $\mathcal R$   together with $\k$ curves in $\mathcal R$ connecting the $2\k$ points in pairs, with two such diagrams being identified if they induce the same matching of the $2\k$ points.  We call the distinguished points {\sf vertices} and the curves {\sf strands}.  A strand is  {\sf vertical} if it connects  a top vertex with a bottom vertex and {\sf horizontal} otherwise.   We label the top vertices by $\p 1, \dots, \p \k$  and the bottom vertices by $\pbar 1, \dots, \pbar \k$ from left to right.

Let $S$ be an integral domain with a distinguished element $\delta \in S$.   As an $S$-module, {\sf the $\k$--strand Brauer algebra}  $B_\k(S; \delta)$  is the free $S$--module with basis the set of $\k$--strand Brauer diagrams.   The product $a b$  of two Brauer diagrams is defined as follows:  stack $a$ over $b $ to obtain a figure $a\ast b$ consisting of a Brauer diagram $c$ together with some number $j$  of closed loops.   Then $ab$ is defined to be $\delta^j c$.  The product on  $B_\k(S; \delta)$ is the bilinear extension of the product of diagrams.

The Brauer algebra $B_\k(S; \delta)$ has an $S$--linear involution $*$ defined on diagrams by reflection in a horizontal line.  The $\k$--strand algebra embeds in the $\k+1$--strand algebra by the map defined on diagrams by adding an additional top vertex $\p {\k+1}$ and an additional bottom vertex $\pbar {\k+1}$ on the right, and connecting the new  pair of vertices by a vertical strand. 

The  $\k$--strand Brauer algebra is generated as a unital algebra by the following Brauer diagrams:
\begin{align} \label{figure brauer generators}
\scalefont{0.9} \begin{minipage}{13cm}\begin{tikzpicture}
\node at (-7.4,-0.5) {$\displaystyle {s_i}= \color{black}$};
\node at (-6.3,-0.4) {$\displaystyle \cdots \color{black}$};
\draw (-6.8,0) -- (-6.8,-0.8);
\draw (-5.8,0) -- (-5.8,-0.8);
\draw (-5,0) -- (-4.2,-0.8);
\draw (-4.2,0) -- (-5,-0.8);
\draw (-3.4,0) -- (-3.4,-0.8);
\draw (-2.4,0) -- (-2.4,-0.8);
\filldraw [black] (-6.8,0) circle (1.2pt);
\filldraw [black] (-6.8,-0.8) circle (1.2pt);
\filldraw [black] (-5.8,0) circle (1.2pt);
\filldraw [black] (-5.8,-0.8) circle (1.2pt);
\filldraw [black] (-5.0,0) circle (1.2pt);
\filldraw [black] (-5.0,-0.8) circle (1.2pt);
\filldraw [black] (-4.2,0) circle (1.2pt);
\filldraw [black] (-4.2,-0.8) circle (1.2pt);
\filldraw [black] (-3.4,0) circle (1.2pt);
\filldraw [black] (-3.4,-0.8) circle (1.2pt);
\filldraw [black] (-2.4,0) circle (1.2pt);
\filldraw [black] (-2.4,-0.8) circle (1.2pt);
\node at (-2.9,-0.4) {$\displaystyle \cdots \color{black}$};
\node at (-5.0,-1.1) {\scalefont{0.8}$\displaystyle i \color{black}$};
\node at (-4.1,-1.1) {\scalefont{0.8}$\displaystyle i+1 \color{black}$};
\node at (-1.4,-0.4) {$\displaystyle \text{and} \color{black}$};
\node at (0.0,-0.5) {$\displaystyle e_i= \color{black}$};
\node at (1.1,-0.4) {$\displaystyle \cdots \color{black}$};
\draw (0.6,0) -- (0.6,-0.8);
\draw (1.6,0) -- (1.6,-0.8);
\draw (2.4,0)  .. controls (2.6,-0.2) and (3.0,-0.2) .. (3.2,0);
\draw (2.4,-0.8)  .. controls (2.6,-0.6) and (3.0,-0.6) .. (3.2,-0.8);
\draw (4.0,0) -- (4.0,-0.8);
\draw (5.0,0) -- (5.0,-0.8);
\filldraw [black] (0.6,0) circle (1.2pt);
\filldraw [black] (0.6,-0.8) circle (1.2pt);
\filldraw [black] (1.6,0) circle (1.2pt);
\filldraw [black] (1.6,-0.8) circle (1.2pt);
\filldraw [black] (2.4,0) circle (1.2pt);
\filldraw [black] (2.4,-0.8) circle (1.2pt);
\filldraw [black] (3.2,0) circle (1.2pt);
\filldraw [black] (3.2,-0.8) circle (1.2pt);
\filldraw [black] (4,0) circle (1.2pt);
\filldraw [black] (4,-0.8) circle (1.2pt);
\filldraw [black] (5,0) circle (1.2pt);
\filldraw [black] (5,-0.8) circle (1.2pt);
\node at (4.5,-0.4) {$\displaystyle \cdots \color{black}$};
\node at (2.4,-1.1) {\scalefont{0.8}$\displaystyle i \color{black}$};
\node at (3.3,-1.1) {\scalefont{0.8}$\displaystyle i+1 \color{black}$};
\end{tikzpicture}\end{minipage}
 \end{align}
   We have $e_i^2 = \delta e_i$,   $e_i^* = e_i$ and $s_i^* = s_i$.    
   An $\k$--strand  Brauer diagram with only vertical strands can be identified with a permutation in $\mathfrak S_\k$, and the product of such diagrams agrees with composition of permutations.  The linear span of the permutation diagrams is thus a subalgebra of $B_\k(S; \delta)$ isomorphic to $S \mathfrak S_\k$.  This subalgebra is generated by the diagrams $s_i$ in \eqref{figure brauer generators}.

The linear span of   $\k$--strand  Brauer diagrams with at least one horizontal strand is an ideal $J_\k$ of $B_\k(S; \delta)$, and $J_\k$ is generated as an ideal by any of the elements $e_i$.  
 The quotient $B_\k(S; \delta)/J_\k$ is also isomorphic to the symmetric group algebra, as algebras with involution. 
 
 The {\sf rank} of a Brauer diagram is the number of its vertical strands;  the {\sf corank} is $1/2$ the number of horizontal strands.

 Denote the generic ground ring $\ZZ[\deltabold]$ for the Brauer algebras  by $R$, and let $R' = \ZZ[\deltabold, \deltabold\inv]$.   
   It was shown in ~\cite[Section 6.3]{EG:2012}  that the pair of towers of algebras
 $(B_\k(R'; \deltabold))_{\k\ge 0}$  and $(R' \mathfrak S_\k)_{\k\ge0}$  satisfy \eqref{J-1}--\eqref{J-8} of  
 \myref{Section}{subsection: Jones},
where $R' \mathfrak S_\k$ is endowed with the Murphy cellular basis.   But the same is true if  $R' \mathfrak S_r$  is endowed instead with the dual Murphy cellular basis.   Following through the work outlined in 
 \myref{Section}{subsection: Jones},
 based on the Murphy basis or the dual Murphy basis of the symmetric group algebras, one obtains two different cellular bases on the Brauer algebras    $B_\k(R; \deltabold)$ over the generic ground ring $R$, which we also call the Murphy and dual Murphy cellular bases.   The tower of Brauer algebras over $R$, with either cellular structure, is restriction--coherent, with branching diagram $\widehat B$ obtained by reflections from Young's lattice.  Explicitly, the branching rule is as follows:
 if $(\lambda, l) \in \widehat B_\k$ and $(\mu, m) \in \widehat B_{\k+1}$,  then $(\lambda, l) \to (\mu, m)$ if and only if  the Young diagram $\mu$ is obtained from the Young diagram $\lambda$ by either adding or removing one box.  
 The partial order  $\unrhd$ on $\widehat B_\k$ for the Murphy cell datum is:   $(\lambda, l) \unrhd (\mu, m)$  if $l > m$ or if $l = m$ and $\lambda \unrhd \mu$  The partial order $\coldomeq$  on 
 $\widehat B_\k$ for the  dual Murphy cell datum is analogous, but with dominance order replaced with column dominance order.

 The branching factors and \cellgenerators\ in the Brauer algebras, computed using \myref{Theorem}{theorem:  closed form determination of the branching factors} and \myref{Lemma}{lemma lifting cell generators},  involve liftings of elements of the symmetric group algebras to the Brauer algebras;  for any element $x \in R\mathfrak S_r$, we lift $x$ to the ``same" element in the span of permutation diagrams in $B_r(r; \deltabold)$.  Thus, 
 for $(\lambda, l) \in \widehat B_\k$,  define
 \begin{equation} \label{x and y generators Brauer}
 x_{(\lambda, l)} = x_\lambda e_{r-1}\power l \quad \text{and} \quad  y_{(\lambda, l)} = y_\lambda e_{r-1}\power l .
 \end{equation} 
 These are the \cellgenerators\  for the Murphy and dual Murphy cellular structures on $B_\k(R; \deltabold)$.
 The branching factors  $d_{(\lambda, l) \to (\mu, m)}$ and  $u_{(\lambda, l) \to (\mu, m)}$  for the Murphy basis are obtained using the formulas of 
 \myref{Theorem}{theorem:  closed form determination of the branching factors} from the $d$-- and $u$--branching factors of the symmetric group algebras; and similarly the $b_{(\lambda, l) \to (\mu, m)}$ and  $v_{(\lambda, l) \to (\mu, m)}$ branching factors for the dual Murphy basis are obtained from the $b$-- and $v$--branching factors of the symmetric group algebras.
These
 satisfy the compatibility relations
\begin{equation} \label{compatibility relations Brauer}
x_{(\mu, m)}  d_{(\lambda, l) \to (\mu, m)} =    u_{(\lambda, l) \to (\mu, m)}^* x_{(\lambda, l) }
\qquad
y_{(\mu, m)}  b_{(\lambda, l) \to (\mu, m)} =    v_{(\lambda, l) \to (\mu, m)}^* y_{(\lambda, l) }. 
\end{equation}
The compatibility relation is  established in \cite[Appendix A]{EG:2012} for the Murphy basis, and the argument holds just as well for the dual Murphy basis. 
For $(\lambda, l) \in \widehat B_\k$ and $\stt \in \Std_\k(\lambda, l)$  let $d_\stt$  be the ordered product of the $d$--branching factors along $\stt$, and  $b_\stt$ the ordered product of the $b$--branching factors along $\stt$.
For $(\lambda, l) \in \widehat B_\k$  and $\sts, \stt \in \Std_k(\lambda, l)$,  define 
$$x\power{\lambda, l}_{\sts \stt} =  d_{\sts}^* x_{(\lambda, l)} d_\stt \qquad y\power{\lambda, l}_{\sts \stt} =  b_{\sts}^* y_{(\lambda, l)} b_\stt.$$ 
We have that 
\begin{align} \label{Murphy basis Brauer}  
\{x\power{\lambda, l}_{\sts \stt}  \suchthat
(\lambda, l) \in \widehat B_\k \text{ and }  \sts, \stt \in \Std_k(\lambda, l) \}
\\
\label{dual Murphy basis Brauer} \{y\power{\lambda, l}_{\sts \stt}  \suchthat
(\lambda, l) \in \widehat B_\k \text{ and }  \sts, \stt \in \Std_k(\lambda, l) \}
\end{align}
are the Murphy and dual Murphy cellular bases of $B_\k(R; \deltabold)$.

\begin{thm}     \label{theorem:  Brauer algebras as diagram algebras}
 The sequence  of Brauer algebras $(B_\k(R; \delta))_{\k \ge 0}$ over the generic ground ring $R= \ZZ[\deltabold]$   with either the Murphy cellular bases  \eqref{Murphy basis Brauer} or the dual Murphy cellular bases \eqref{dual Murphy basis Brauer}, 
 satisfies \eqref{diagram 1}--\eqref{diagram 6} of  \myref{Section}{subsection diagram algebras}.
  The data entering into the definition of the Murphy bases and dual Murphy bases are explicitly determined using equation \eqref{x and y generators Brauer} and the formulas of 
 \myref{Theorem}{theorem:  closed form determination of the branching factors}.
  \end{thm}
  
  \begin{proof}  For both the Murphy and dual Murphy bases,  conditions \eqref{diagram 1},  \eqref{diagram cyclic cellular},  \eqref{diagram semisimplicity}, and   \eqref{diagram restriction coherent} follow from the general framework of \cite{EG:2012}, as outlined in  \myref{Section}{subsection: Jones}.  Condition \eqref{diagram symmetry} follows from \eqref{x and y generators Brauer} and condition   \eqref{diagram 6}  from \eqref{compatibility relations Brauer}.
   \end{proof}

We will require the following lemma in 
\myref{Sections}{section Brauer symplectic}
and
\myrefnospace{}{section Brauer orthogonal}.

\begin{lem} \label{Brauer block diagonal transition}
 If $D$ is an $r$--strand  Brauer diagram of corank $\ge m+1$,  then for all $\mu \vdash r- 2m$, 
$D$ is an element of the ideal $B_r(R; \deltabold)^{\rhd (\mu, m)}$  defined using the Murphy basis. Likewise, $D$ is an element of the ideal $B_r(R; \deltabold)^{\coldom (\mu, m)}$  defined using the dual Murphy basis.
\end{lem}

\begin{proof}  It follows from the computation of the transition matrix between the Murphy basis and the diagram basis in \cite[Section 6.2.3]{EG:2012} that for all $(\lambda, l) \in \widehat B_r$ and all 
$\sts, \stt \in \Std_r(\lambda, l)$,   the Murphy basis element $x_{\sts \stt}\power {\lambda, l}$  is an integer linear combination of Brauer diagrams with corank $l$.  Thus the transition matrix is block diagonal, and the inverse transition matrix is also block diagonal.  Hence if the corank of $D$ is $l  > m$,   then
$D$ is a linear combination of Murphy  basis elements $x_{\sts \stt}\power {\lambda, l}$ with 
$\lambda \vdash r - 2l$  and $\sts, \stt \in \Std_r(\lambda, l)$.   It follows that 
$D \in B_r(R; \deltabold)^{\rhd (\mu, m)}$ for any $\mu \vdash r - 2m$.     The same argument holds for the dual Murphy basis.
\end{proof}

\section{Bilinear forms and the action of the Brauer algebra on tensor space} \label{section tensor space}

Let $V$ be a finite  dimensional vector space over a field $\Bbbk$ with a non--degenerate bilinear form denoted $[\ , \ ]$.  For the moment we make no assumption about the symmetry of the bilinear form.   The non--degenerate form induces an isomorphism $\eta : V \to V^*$, defined by $\eta(v)(w) = [w, v]$ and hence a linear isomorphism $A : V \otimes V \to \End(V)$ defined by $A(v \otimes w)(x) = \eta(v)(x) \,w = [x, v]\, w$.   We write $x \cdot (v \otimes w) = A(v \otimes w)(x) = [x, v]\, w$, and also $(v \otimes w) \cdot x =
v [w, x]$.     

For all $r \ge 1$, extend the bilinear form to $V^{\otimes r}$ by $[x_1 \otimes x_2 \cdots x_r, y_1 \otimes y_2 \cdots y_r] = \prod_i [x_i, y_i]$.  Then this is also an non--degenerate bilinear form so induces
isomorphisms $\eta_r : V^{\otimes r} \to (V^*)^{\otimes r}$  and $A_r:  V^{\otimes 2r} \to \End(V^{\otimes r})$.   We will generally just write $\eta$ and $A$ instead of $\eta_r$ and $A_r$.      In the following, let $G$ denote the group of linear transformations of $V$ preserving the bilinear form.

\begin{lem} \label{lemma  properties of omega}
 Let $\{v_i\}$  and $\{v_i^*\}$  be dual bases of $V$ with respect to the bilinear form, i.e. bases such that $[v_i, v_j^*] = \delta_{i, j}$, and let $\omega = \sum_i v_i^* \otimes v_i$.    Then:
\begin{enumerate}[leftmargin=*,label=(\arabic{*}), font=\normalfont, align=left, leftmargin=*]
\item For all $x \in V$, $x \cdot \omega = \omega \cdot x = x$.  In particular, $\omega = A\inv(\id_V)$ and $\omega$ is independent of the choice of the dual bases.
\item  For all $x, y \in V$,  $[x \otimes y, \omega] = [y, x]$. 
\item $\omega$ is $G$--invariant.
\end{enumerate}
\end{lem}

\begin{proof}  For any $j$, $v_j \cdot \omega = \sum_i [v_j, v_i^*] v_i = v_j$.  Similarly, $\omega \cdot v_j^* = v_j^*$.    Hence, for all $x$,  $x \cdot \omega = \omega \cdot x = x$.    We have
$[x \otimes y, \omega] = \sum_i [x, v_i^*] [y, v_i] =   [ y,  x \cdot \omega] = [y, x]$.   For the $G$--invariance of $\omega$, note that for $g \in G$, $g\cdot \omega = \sum_i  g v_i^* \otimes g v_i = \omega$, because
$\{g v_i\}$  and $\{g v_i^*\}$ is another pair of dual bases.
\end{proof}

For the remainder of this section, we assume that the bilinear form $[\ , \ ]$ is either symmetric or skew--symmetric.  Note that in both cases the bilinear form induced on $V\otimes V$ is symmetric. 
 Because the bilinear form on $V^{\otimes r}$ is non--degenerate,  $\End(V\tensor r)$ has a $\Bbbk$--linear algebra involution $*$  defined by $[T^*(x), y] = [x, T(y)]$, for $T \in \End(V\tensor r)$ and $ x, y \in V\tensor r$.   (The involution property depends on the bilinear form being either symmetric or skew--symmetric.)

Define $E, S  \in \End(V \otimes V)$ by $(x \otimes y) E = [x, y] \omega$, and  $(x \otimes y) S = y \otimes x$.    These will be used to define a right action of  Brauer algebras on tensor powers of $V$.

\begin{lem} \label{lemma properties of E and S} \mbox{}
Write $\epsilon = 1$ if the bilinear form $[\ , \ ]$ on $V$ is symmetric and $\epsilon = -1$ if the bilinear form   is skew symmetric.

\begin{enumerate}[leftmargin=*,label=(\arabic{*}), font=\normalfont, align=left, leftmargin=*]
\item  $E S = S E = \epsilon E$
\item  $E^2 = (\epsilon \dim V) E$.
\item  $E = E^*$ and $S = S^*$ in $\End(V \otimes V)$.
\item  $E$ and $S$ commute with the action of $G$ on $V \otimes V$. 
\end{enumerate}
\end{lem}

\begin{proof}  These statements follow from  straightforward computations.    
%
The proof of the last statement, on $G$--invariance, uses \myref{Lemma}{lemma  properties of omega}, part (3).  
\end{proof}

For $r \ge 1$  and for $1 \le i \le r-1$  define $E_i$ and $S_i$ in $\End(V\tensor r)$  to be $E$  and $S$ acting in the $i$--th and $i+1$--st tensor places. 

\begin{prop}[Brauer, \cite{MR1503378}]   Let $V$ be a finite dimensional  vector space  over $\Bbbk$ with a non--degenerate symmetric or skew--symmetric bilinear form $[ \ , \ ]$, and let $G$ be the group of linear transformations of $V$ preserving the bilinear form.   Then  for $r \ge 1$, $e_i \mapsto E_i$  and 
$s_i \mapsto  \epsilon S_i$  determines a homomorphism 
$$
\Phi_r : B_r(\Bbbk;  \epsilon \dim V) \longrightarrow  \End_G(V \tensor r).
$$
\end{prop}

\begin{proof}  Brauer works over the complex numbers, but his argument in \cite[page 869]{MR1503378} is equally valid over any field.   Alternatively, one can use the presentation of the Brauer algebra, see \cite[Proposition 2.7]{MR2235339}  for example, and verify that the images of the generators $s_i, e_i$ satisfy the defining relations. 
\end{proof}

Note that the symmetric group (contained in the Brauer algebra)  acts  on $V \tensor r$ by place permutations if the bilinear form  $[\ , \ ]$ on $V$ is symmetric and by signed place permutations if the bilinear form is skew--symmetric.

We have $\Phi_{r+1} \circ \iota =  \iota \circ \Phi_\f$,   where we have used $\iota$ to denote both the embedding of $B_\f$ into $B_{r+1}$ and the embedding of  $\End(V^{\otimes r})$ into
$\End(V^{\otimes {(r+1)}})$,   namely $\iota: T \mapsto T \otimes \id_V$.  In particular, this implies $\ker(\Phi_\f) \subseteq  \ker(\Phi_{r+1})$.   Because of this, we will 
sometimes  just write $\Phi$ instead of $\Phi_\f$.

\begin{lem}  The homomorphism $\Phi_r$ respects the involutions on $B_r(\Bbbk;  \epsilon \dim V)$ and
$ \End(V \tensor r)$, i.e.  $\Phi_r(a^*) = \Phi_r(a)^*$.   Consequently,  the image $\im(\Phi_r)$ is an algebra with involution, and $\ker(\Phi_r)$ is a $*$--invariant ideal in $B_r(\Bbbk;  \epsilon \dim V)$. 
\end{lem}

\begin{proof}  Follows from \myref{Lemma}{lemma properties of E and S}, part (3). 
\end{proof}
 
\section{The Brauer algebra on symplectic tensor space} \label{section Brauer symplectic}

Let $V$ be a $2N$--dimensional vector space over a field $\Bbbk$ with a symplectic form $\langle \ , \ \rangle$, i.e., a non--degenerate, alternating (and thus skew--symmetric) bilinear form.    One can show that
$V$  has a {\sf Darboux basis}, i.e. a basis $\{v_i \}_{1 \le i \le 2N}$ such that the dual basis  $\{v_i^*\}$ with respect to the symplectic form is $v_i^* = v_{2N + 1 - i}$ if $1 \le i \le N$  and $v_i^* = -  v_{2N + 1 - i}$ if $N+1 \le i \le 2N$.   Hence,     one can assume without loss of generality that $V = \Bbbk^{2N}$ with the standard symplectic form
$$
\langle x,  y \rangle = \sum_{i = 1}^N (x_i y_{2N +1 -i} -  y_i x_{2N + 1 -i}).
$$
For $r \ge 1$, let  $\Phi_r : B_r(\Bbbk; -2N) \to \End(V\tensor r)$ be the homomorphism defined as in \myref{Section}{section tensor space} using the symplectic form.   When required for clarity, we write
$V_\Bbbk$ for $V = \Bbbk^{2N}$ and $\Phi_{r, \Bbbk}$ for $\Phi_r$.  The image $\im(\Phi_{r, \Bbbk})$ is known as the  {\sf (symplectic) Brauer centralizer algebra}.

\begin{thm}[\cite{MR2342000}]  \label{symplectic SW duality}
 Let $\Lambda : \Bbbk \Symp(V) \to \End(V\tensor r)$ denote the homomorphism corresponding to the diagonal action of the symplectic group $ \Symp(V)$ on $V\tensor r$.   
\begin{enumerate}[leftmargin=*,label=(\arabic{*}), font=\normalfont, align=left, leftmargin=*]
\item 
 If $\Bbbk$ is a quadratically closed infinite field, then  $\im(\Phi_r) = \End_{\Symp(V)}(V \tensor r)$ and
$\im(\Lambda) = \End_{B_r(\Bbbk; -2N)}(V \tensor r)$.  
\item  The dimension of  $\im(\Phi_r)$ is independent of the field and  of the characteristic, for infinite fields $\Bbbk$. 
\end{enumerate}
\end{thm}

\begin{rem}  The special case when $\Bbbk$ is the field of complex numbers  is due to Brauer  \cite{MR1503378}. The statement  of part (1) in \cite{MR2342000} is more general, allowing general infinite fields at the cost of replacing the symplectic group with the symplectic similitude group. 
\end{rem}

Let $\Phi_{r, \ZZ}$ denote the restriction of $\Phi_{r, \C}$ to $B_r(\ZZ; -2N)$; the image  $\im(\Phi_{r, \ZZ})$ is the $\ZZ$-subalgebra of $\End(V_\C\tensor r)$  generated by $E_i$ and $S_i$ for $1 \le i \le r-1$.   Let $V_\ZZ$ be the $\ZZ$-span of  the standard basis  $\{e_i \suchthat 1 \le i \le 2N\}$.  Thus $V_\ZZ \tensor r \subset V_\C \tensor r$ and  $\End_\ZZ(V_\ZZ \tensor r) \subset \End(V_\C \tensor r)$.   Since $E_i $ and $S_i$ leave $V_\ZZ \tensor r$ invariant, we can also regard $\im(\Phi_{r, \ZZ})$ as a $\ZZ$-subalgebra of 
 $\End_\ZZ(V_\ZZ \tensor r) $.   
 
 For any $\Bbbk$,  $B_r(\Bbbk; -2N) \cong B_r(\ZZ; -2N) \otimes_\ZZ \Bbbk$ and $V_\Bbbk \cong V_\ZZ \otimes_\ZZ \Bbbk$.     For $a \in B_r(\ZZ; -2N)$ and $w \in V_\ZZ$,  we have
 $\Phi_{r, \Bbbk}(a \otimes 1_\Bbbk) (w \otimes 1_\Bbbk) = \Phi_{r, \ZZ}(a)(w)\otimes 1_\Bbbk $.  Therefore, we are in the situation of \myref{Lemma}{lemma specialization cd}, and there exists a map $\theta : \im(\Phi_{r, \ZZ}) \to \im(\Phi_{r, \Bbbk})$ making the diagram commute:
  \!\!\!\!\!\!
  \begin{equation}  \label{symplectic CD}
\begin{minipage}{5cm}
\begin{tikzpicture}
 \draw (0,0) node {$B_r(\Z; -2N)$};\draw (3,0) node {\ \ $\im(\Phi_{r, \ZZ})$};
\draw[->](1.1,0) to node [above] {\scalefont{0.7} $\Phi_{r, \ZZ}$} (2.4,0);
\draw[->](0,-0.3) to node [right] {\scalefont{0.8}$\otimes 1_\Bbbk$} (0,-1.2);
\draw[->](3,-0.3) to node [right] {\scalefont{0.8}$\theta$} (3,-1.2);
\draw[->](1.1,-1.5) to node [below] {\scalefont{0.7} $\Phi_{r, \Bbbk}$} (2.4,-1.5);
\draw (0,-1.5) node {$B_r(\Bbbk; -2N)$};\draw (3,-1.5) node {\ \ $\im(\Phi_{r, \Bbbk})$};
\end{tikzpicture}\end{minipage} .
\end{equation}

  \subsection{Murphy basis over the integers}
 \label{symplectic integral BCA}

\begin{defn}
Write $A^{\mathrm s}_\f(**) = \Phi(B_\f( **; -2N))$, where $**$ stands for $\C$, $\Q$, or $\Z$. 
 Thus $A^{\mathrm s}_\f(\Z)$ is the $\Z$--algebra generated by $E_i = \Phi(e_i)$   and $S_i = -\Phi(s_i)$.  
 (The superscript ``s'' in this notation stands for ``symplectic".)  
 \end{defn}

Let $R = \Z[\deltabold]$.  Endow $B_r(R; \deltabold)$  with the Murphy 
cellular structure described in 
\myref{Section}{section: Murphy bases Brauer}
with the Murphy type basis
$$
\leftbrace x_{\mfs \mft}\power{\lambda, l} \suchthat  (\lambda, l) \in \widehat B_\k \text{ and } \mfs, \mft \in \Std_\k(\lambda, l) \rightbrace.
$$
By 
\myref{Theorem}{theorem:  Brauer algebras as diagram algebras},
the tower $(B_\k(R; \deltabold))_{\k \ge 0}$ satisfies the assumptions
\eqref{diagram 1}--\eqref{diagram 6} of  \myref{Section}{subsection diagram algebras}.

 We want to show that the maps   $\Phi_{r, \ZZ} : B_\f(\Z; -2N) \to  A^{\mathrm s}_\f(\Z)$ satisfy the assumptions
 \eqref{quotient axiom 1}--\eqref{quotient axiom 3} of 
 \myref{Section}{section quotient framework}.
 It will follow that the integral Brauer centralizer algebras $ A^{\mathrm s}_\f(\Z)$ are cellular over the integers. 

First we need the appropriate notion of permissibility for points in $\widehat B_r$ and for paths on $\widehat B$.   

\begin{defn}  A $(-2N)$--permissible partition $\lambda$ is a partition such that $\lambda_1 \le N$.  
We say that  an element
$(\lambda, l) \in \widehat B_\f$ is $(-2N)$--permissible if $\lambda$ is $(-2N)$--permissible.  
   We let  $\widehat B^{\rm s}_{\f,  \perm}\subseteq \widehat B_{\f}$   denote  the subset of $(-2N)$--permissible points.  
   
   A path $\mft \in \Std_\f({\lambda, l})$ is $(-2N)$--permissible if $\mft (k)$ is $(-2N)$--permissible for all $0\leq k \leq \f$.
  We let  $\Std^{\mathrm s}_{\f,  \perm}(\lambda, l) \subseteq \Std _{\f}(\lambda, l)$   denote the subset 
  of   $(-2N)$--permissible paths. 
\end{defn}

 Note that this set of permissible points satisfies condition \eqref{quotient axiom 1}.

\medskip
For any ring $U$ and any $\delta \in U$,  and any natural numbers $r$, $s$,  there is an  injective $U$--algebra homomorphism $B_r(U; \delta) \otimes B_s(U; \delta) \to B_{r+s}(U; \delta)$ defined on the basis of Brauer diagrams by placing diagrams side by side.  We also write $x\otimes y$ for the image of $x\otimes y$ in $B_{r+s}(U; \delta)$.  

 \begin{defn} \label{defn of diagrammatic Pfaffian}
 Define $\mathfrak b_r \in B_r(\Z; -2N)$ to be the sum of all Brauer diagrams on $r$ strands and $\mathfrak b_r'$ to be the sum of all Brauer diagrams on $r$ strands  of corank $\ge 1$.  
 For $\lambda = (\lambda_1, \dots, \lambda_s)$  a partition of $r$, write
\begin{equation} 
\label{symplectic a}
\mathfrak b_\lambda = \mathfrak b_{\lambda_1} \otimes x_{(\lambda_2, \dots, \lambda_s)} \quad \text{and} \quad
\mathfrak b'_\lambda = \mathfrak b'_{\lambda_1} \otimes x_{(\lambda_2, \dots, \lambda_s)}. 
\end{equation}
For $(\lambda, l) \in \widehat B_\k$, write
\begin{equation}     \label{symplectic b}
\mathfrak b_{(\lambda, l)} = \mathfrak b_\lambda e_{\k-1}\power l   \quad \text{and} \quad
\mathfrak b'_{(\lambda, l)} = \mathfrak b'_\lambda e_{\k-1}\power l. 
\end{equation}
\end{defn}

\begin{rmk}Thus, for all $\k\geq0$ and for all  $(\lambda, l) \in \widehat B_\k$,
\begin{equation} \label{symplectic 1}
x_{(\lambda, l)} =  \mathfrak b_{(\lambda, l)} - \mathfrak b'_{(\lambda, l)}.
\end{equation} 
\end{rmk}
\begin{lem}  \label{factorization of b'}
There exists $\beta'_r \in B_r(\Q; -2N)$ such that $\mathfrak b'_r = x_{(r)} \beta'_r$.
\end{lem}

 \begin{proof}  $\mathfrak S_r$ acts on the left on the set of Brauer diagrams on $r$ strands with corank $\ge 1$.   
Choose a representative of each orbit.  Then
$$
\mathfrak b'_r =  x_{(r)} \  \sum_x  \frac{1}{|\Stab(x)|} \, x,
$$
where the sum is over orbit representatives $x$ and $|\Stab(x)|$ is the cardinality of the stabilizer of $x$ in $\mathfrak S_r$.  
\end{proof}

It follows that for all $\k\geq0$ and for all $(\lambda, l) \in \widehat B_\k$,
 \begin{equation}  \label{symplectic 2}
  \mathfrak b'_{(\lambda, l)} =  x_{(\lambda, l)} \beta'_{\lambda_1}.
 \end{equation}

Fix $r \ge 1$.  The multilinear functionals on $V^{2r}$ of the form
$(w_1, \dots, w_{2r}) \mapsto \prod  \langle w_i, w_j\rangle$, where each $w_i$ occurs exactly once, are evidently 
$\rm {Sp}(V)$--invariant.  Moreover, there are some obvious relations among such functionals.  If we take 
$r = N+1$, then for any choice of  $(w_1, \dots, w_{2r})$, the  $2r$--by--$2r$ skew-symmetric matrix
$(\langle w_i, w_j \rangle)$ is singular and therefore the Pfaffian of this matrix is zero, which provides such a relation.
  These elementary observations are preliminary to  the first and second fundamental theorems of invariant theory for the symplectic  groups, see   
  \cite[Section 6.1]{MR0000255}.
The following proposition depends on the second of these observations.

\begin{prop}  \label{symplectic fund theorem 1}
The element $\mathfrak b_{N+1}$ is in $\ker(\Phi)$.  Hence if $r \ge N+1$ and  $(\lambda, l) \in \widehat B_\k$ with 
$\lambda_1 = N + 1$, then $\mathfrak b_{(\lambda, l)} \in \ker(\Phi)$.  
\end{prop}

\begin{proof}  Set $r = N+1$. There exist linear isomorphisms $A : V\tensor{2r} \to \End(V\tensor{r})$  and 
$\eta: V\tensor{2r} \to (V\tensor{2r} )^*$.  
The proof consists of showing that
$(\eta\circ A\inv\circ \Phi(\mathfrak b_{N+1}))(w_1 \otimes \cdots \otimes w_{2r})$ is up to a sign the Pfaffian of the singular matrix $(\langle w_i, w_j \rangle)$.  Hence $\mathfrak b_{N+1} \in \ker(\Phi)$.  
 This is explained in  \cite[Section 3.3]{MR1670662}.    
  We have also provided an  exposition in  
   \iftoggle{arxiv}
 { \myref{Appendix}{appendix symplectic}}
 {Appendix D in the arXiv version of this paper}.
  Another proof can be found in  \cite[Proposition 4.6]{MR2430306}. 
\end{proof}
  
We can now verify axiom \eqref{quotient axiom 2}.    Let $\mft \in \Std_\f(\lambda, l)$ be a path which is not $(-2N)$-permissible.  
Let $k \le r$ be the first index  such that $\mft(k) = (\mu, m)$  satisfies  $\mu_1 = N+1$. 
    It follows from
  \myref{Proposition}{symplectic fund theorem 1}
 that 
\begin{equation}  \label{symplectic 3}
\mathfrak b_{(\mu, m)} \in \ker(\Phi_r).  
\end{equation}
By \eqref{symplectic 2} and \eqref{x and y generators Brauer}, we have that 
  $$  \mathfrak b'_{(\mu, m)} =  x_{(\mu, m)} \beta'_{\mu_1} = x_\mu \beta'_{\mu_1} e_{k-1}\power m$$   is a linear combination of Brauer diagrams of corank at least $m+1$, and therefore
$ \mathfrak b'_{(\mu, m)}  \in B_\k(\Z; -2N)^{\rhd (\mu, m)}$, using 
   \myref{Lemma}{Brauer block diagonal transition}. 
 Hence
\begin{equation}\label{symplectic 4}
 \mathfrak b'_{(\mu, m)}  \in  x_{(\mu, m)} B_\k(\Q; -2N) \cap  B_\k(\Z; -2N)^{\rhd (\mu, m)} 
\end{equation}
as required.  
Taken together, \eqref{symplectic 1}--\eqref{symplectic 4}  show that  axiom \eqref{quotient axiom 2} holds.

Finally, it is shown in \cite[Theorem 3.4, Corollary 3.5]{MR951511} that 
$$
\dim_{\C}(A_\f^{\mathrm s}(\C)) =  
\smashr{\sum_{(\lambda, l) \in \widehat B^{\rm s}_{\f, \perm}}   } \ 
(\sharp \Std^{\mathrm s}_{\f, \perm}(\lambda, l))^2.
$$
Thus axiom \eqref{quotient axiom 3} holds.   

 Since assumptions
 \eqref{quotient axiom 1}--\eqref{quotient axiom 3}   of 
\myref{Section}{section quotient framework}
 are satisfied, we can produce a modified Murphy basis  of  $B_\f(\Z; -2N) $,
 $$
\leftbrace \widetilde x_{\mfs \mft}\power{\lambda, l} \suchthat  (\lambda, l) \in \widehat B_\k \text{ and } \mfs, \mft \in \Std_k(\lambda, l) \rightbrace,
$$
following the procedure described before 
\myref{Theorem}{basis theorem}.
   The following theorem gives a cellular basis for Brauer's centralizer algebra acting on symplectic tensor space, valid over the integers.  It also gives two descriptions  of the kernel  of the map
 $\Phi_{\f , \ZZ}: B_\f(\ZZ; -2N) \to \End(V^{\otimes r})$, one by providing a basis of $\ker(\Phi_{\f, \ZZ})$ over the integers, and the other by describing the kernel as the ideal generated by a single element.  Each of these statements is a form of the second fundamental theorem of invariant theory for the symplectic groups. 
 
\begin{thm}  \label{integral cellular basis 1}
  The integral (symplectic) Brauer centralizer algebra 
 $A^{\mathrm s}_\f(\Z)$ is a cellular algebra over $\Z$ with basis 
$$\mathbb A^{\mathrm s}_\f(\ZZ) = \leftbrace \Phi_{r, \ZZ}(\widetilde x_{\mfs \mft}\power{\lambda, l})\mid   
(\lambda, l) \in \widehat B^{\rm s}_{\f, \perm} \text{ and }  \mfs, \mft \in \Std^{\mathrm s}_{\f,  \perm}(\lambda, l)
 \rightbrace 
$$
with 
the involution $*$  determined by $E_i^* = E_i$ and $S_i^* = S_i$ and 
 the partially ordered set \break $(\widehat B^{\rm s}_{\f,   \perm}, \unrhd)$.  
The ideal  $\ker(\Phi_{\f, \ZZ})\subseteq  B_\f(\ZZ; -2N)$ has  $\ZZ$-basis 
$$ \kappa_r = 
\leftbrace  \widetilde x_{\mfs \mft}\power{\lambda, l}
\mid   
(\lambda, l) \in \widehat B_{r} \text{ and }  \mfs \text{ or } \mft  \text{ not  $(-2N)$--permissible}
 \rightbrace.
$$  
Moreover,    for $r>N$, $\ker(\Phi_{\f, \ZZ})$  is the ideal generated by the single element 
$\mathfrak b_{N+1} \in   B_\f(\ZZ; -2N)$.   For $r \le N$,    $\ker(\Phi_{\f, \ZZ})  = 0$. 
\end{thm}

\begin{proof}   The construction of the  cellular basis of  $A^{\mathrm s}_\f(\Z)$ and the basis of $\ker(\Phi_r)$   follows immediately from 
\myref{Theorem}{good paths basis theorem}
 since \eqref{quotient axiom 1}--\eqref{quotient axiom 3} have been verified.

Now, if $r\leq N$ then  $\ker(\Phi_{\f, \ZZ})=0$, since all paths on $\widehat B$ of length $\le N$ are $(-2N)$--permissible. 
  For $r > N$,  the kernel  is the ideal generated by all the $\mathfrak{b}_{(\mu, m)}$ such that $(\mu,m)$ is a   marginal point in $\widehat B_k$ for some $0< k \le r$,  using  
  \myref{Theorem}{good paths basis theorem}.
   But the marginal points  are all of the form $(\mu,m)$ for some $\mu$ with $\mu_1=N+1$. 
 Now, by \eqref{symplectic a}  and  \eqref{symplectic b}
 we have that 
 $$ \mathfrak{b}_{(\mu, m)} = \mathfrak{b}_\mu e_{\k-1}\power m =
 \mathfrak b_{N+1} \otimes x_{(\lambda_2, \dots, \lambda_s)}e_{k-1}\power m
$$
and so we are done.
 \end{proof}

\begin{rmk}\label{7point8}
Applying \myref{Corollary}{good paths cor},   for every $r \ge 0$ and for every permissible point $(\lambda, l) \in \widehat B_r$,   the cell module $\cell {A_r^{\mathrm s}} \C {(\lambda, l)}$, regarded as  a $B_r(\C; -2N)$--module,    is the simple head of the
cell module   $\cell {B_r(\C; -2N)} {} {(\lambda, l)}$.   Thus  the cell module 
$\cell {A_r^{\mathrm s}} \ZZ {(\lambda, l)}$ is an integral form of the simple $B_r(\C; -2N)$--module 
$L_{{B_r(\C; -2N)}} ((\lambda, l))$.    In this way, all the simple $B_r(\C; -2N)$--modules that factor through the representation on symplectic tensor space are provided with integral forms.
\end{rmk}

\subsection{Murphy basis  over an arbitrary field}
 \label{symplectic BCA over a field}

We return to the general situation described at the beginning of \myref{Section}{section Brauer symplectic}:
$V$ is a vector space of dimension $2N$ over an arbitrary field $\Bbbk$, with a symplectic form, and 
for $r \ge 1$,   $\Phi_{r, \Bbbk} : B_{r}(\Bbbk; -2N) \to \End(V\tensor r)$ is Brauer's homomorphism.
We assume without loss of generality that $V = \Bbbk^{2N}$, with the standard symplectic form.  We have the commutative diagram
\eqref{symplectic CD}.   

For the modified Murphy basis $\{\widetilde x_{\mfs \mft}\power{\lambda, l} \}$ of $B_r(\ZZ; -2N)$,  we also write $\widetilde x_{\mfs \mft}\power{\lambda, l}$ instead of $\widetilde x_{\mfs \mft}\power{\lambda, l}  \otimes 1_\Bbbk$ for the corresponding basis element of $B_r(\Bbbk; -2N)$.

  \begin{thm}  \label{field cellular basis 1}
  Let $V$ be a $2N$--dimensional vector space over a field $\Bbbk$.  Assume that $V$ has a symplectic form, and let
   $\Phi_{r, \Bbbk} : B_{r}(\Bbbk; -2N) \to \End(V\tensor r)$ be Brauer's homomorphism defined using the symplectic form.
   Write 
 $A^{\mathrm s}_\f(\Bbbk)$  for the Brauer centralizer algebra $\im(\Phi_{r, \Bbbk})$.  

  The Brauer centralizer algebra 
 $A^{\mathrm s}_\f(\Bbbk)$ is a cellular algebra over $\Bbbk$ with basis 
$$\mathbb A^{\mathrm s}_\f(\Bbbk) = \leftbrace \Phi_{r,\Bbbk}(\widetilde x_{\mfs \mft}\power{\lambda, l})\mid   
(\lambda, l) \in \widehat B^{\rm s}_{\f, \perm} \text{ and }  \mfs, \mft \in \Std^{\mathrm s}_{\f,  \perm}(\lambda, l)
 \rightbrace 
$$
with 
the involution $*$  
determined by $E_i^* = E_i$ and $S_i^* = S_i$ and 
the partially ordered set \break  $(\widehat B^{\rm s}_{\f,   \perm}, \unrhd)$.  
The ideal  $\ker(\Phi_{r, \Bbbk})\subseteq  B_\f(\Bbbk; -2N)$  
has basis  
$$
\kappa_{r} = 
\leftbrace  \widetilde x_{\mfs \mft}\power{\lambda, l} 
\mid   
(\lambda, l) \in \widehat B_{\f} \text{ and }  \mfs \text{ or } \mft  \text{ not  $(-2N)$--permissible}
 \rightbrace.
$$  
Moreover,  $\ker(\Phi_{r, \Bbbk})$  is the ideal generated by the single element 
$\mathfrak b_{N+1}   \in   B_\f(\Bbbk; -2N)$ for $r>N$ (and is zero for $r \leq N$).  

\end{thm}  

\begin{proof} 
Refer to the commutative diagram \eqref{symplectic CD}.
 If $\sts$ or $\stt$ is not $(-2N)$--permissible, then $\Phi_{r, \Bbbk}(\widetilde x_{\sts \stt}\power{\lambda, l} ) = \theta\circ \Phi_{r, \ZZ}(\widetilde x_{\sts \stt}\power{\lambda, l}   ) = 0$.   Thus
$\kappa_{r} \subset \ker(\Phi_{r, \Bbbk})$.   It follows from this that 
$\mathbb A^{\mathrm s}_\f(\Bbbk)  
 $ spans $A_r^{\mathrm s}(\Bbbk) $.  
Once we have established  that $ \mathbb A^{\mathrm s}_\f(\Bbbk)$ is linearly independent, the argument of 
\myref{Theorem}{good paths basis theorem} 
shows that 
$\kappa_{r}$ is a basis of 
$\ker(\Phi_{r, \Bbbk})$, 
and that 
$ \mathbb A^{\mathrm s}_\f(\Bbbk)$ is a cellular basis  of the Brauer centralizer algebra $ A^{\mathrm s}_\f(\Bbbk)$.

Suppose first that $\Bbbk$ is an infinite  field.  By \myref{Theorem}{symplectic SW duality} part (2), 
 the dimension of $\ker(\Phi_{r,\Bbbk})$ 
and the dimension of $ A_r^{\mathrm s}(\Bbbk)$ are
 independent of the (infinite) field $\Bbbk$ and of the characteristic.   Therefore 
 $ \mathbb A_r^{\mathrm s}(\Bbbk)$ is a basis of  $ A_r^{\mathrm s}(\Bbbk)$.
 
 Now consider the case that $\Bbbk$ is any field;  let $\overline \Bbbk$ be the algebraic closure of $\Bbbk$.
Applying   
 \myref{Lemma}{lemma specialization cd}
 again, we have a commutative diagram
$$
\begin{minipage}{5cm}
 \begin{tikzpicture}
 \draw (0,0) node {$B_r(\Bbbk; -2N)$};\draw (3,0) node {$A_r^{\rm s}(\Bbbk)$};
\draw[->](1.1,0) to node [above] {\scalefont{0.7} $\Phi_{r, \Bbbk}$} (2.4,0);
\draw[->](0,-0.3) to node [right] {\scalefont{0.8}$\otimes 1_{\overline\Bbbk}$} (0,-1.2);
\draw[->](3,-0.3) to node [right] {\scalefont{0.8}$\eta$} (3,-1.2);
\draw[->](1.1,-1.5) to node [below] {\scalefont{0.7} $\Phi_{r, \overline\Bbbk}$} (2.4,-1.5);
\draw (0,-1.5) node {$B_r(\overline{\Bbbk}; -2N)$};\draw (3,-1.5) node {$ A_r^{\mathrm s}(\overline\Bbbk )$};
\end{tikzpicture}
\end{minipage} .
$$
 \noindent We conclude that $\mathbb A^{\mathrm s}_r(\Bbbk)$ is linearly independent over $\Bbbk$, since
$\eta(\mathbb A^{\mathrm s}_r(\Bbbk))  = \mathbb A^{\mathrm s}_r(\overline \Bbbk)$ is linearly independent over $\overline \Bbbk$.  
As noted above, it now follows that
 $\mathbb A^{\mathrm s}_r(\Bbbk)$ is a cellular basis of 
$A^{\mathrm s}_r(\Bbbk)$, and  $\kappa_{r} $ is a basis of
 $\ker(\Phi_{r, \Bbbk})$.  For the final statement, it suffices to show that elements of $\kappa_r$ are in the ideal generated by 
 $\mathfrak b_{N+1} $ and this 
 follows from
 \myref{Theorem}{integral cellular basis 1}.
\end{proof} 

\begin{rmk}  
\myref{Theorem}{field cellular basis 1}
extends 
the constancy of  dimension of
Brauer's centralizer algebras $A^{\mathrm s}_\f(\Bbbk)$ and of $\ker(\Phi_{r, \Bbbk})$, from  
\myref{Theorem}{symplectic SW duality} part (2),   to all   fields $\Bbbk$.   
\ignore{
It also follows from \myref{Theorem}{field cellular basis 1} that 
$A^{\mathrm s}_r(\Bbbk) = A^{\mathrm s}_r(\ZZ) \otimes_\ZZ \Bbbk$, i.e. the Brauer centralizer algebra  $A^{\mathrm s}_\f(\Bbbk)$ acting on $V_\Bbbk^{\otimes r}$ is a specialization of 
the integral Brauer centralizer algebra $A^{\mathrm s}_\f(\ZZ)$. 
}
\end{rmk} 

We also have the following corollary.

\begin{cor} \label{corollary field cellular basis 1}
Adopt the hypothesis of \myref{Theorem}{field cellular basis 1}.  The Brauer centralizer algebra  $A^{\mathrm s}_\f(\Bbbk)$ acting on $V^{\otimes r}$ is the specialization of 
the integral Brauer centralizer algebra $A^{\mathrm s}_\f(\ZZ)$,   i.e.
$A^{\mathrm s}_r(\Bbbk)  \cong  A^{\mathrm s}_r(\ZZ) \otimes_\ZZ \Bbbk$.
\end{cor}

\begin{proof}  In general, if $A$ is and $R$--algebra which is free as an $R$--module with basis $\{b_i\}$ and structure constants $r_{i j}^k \in R$, then any specialization $A^S = A \otimes_R S$ is characterized by being free as an $S$--module with basis $\{b_i \otimes 1_S\}$  and structure constants $r_{i j}^k \otimes 1_S$. 
Now $A^{\mathrm s}_r(\ZZ)$ has $\ZZ$--basis $\{\Phi_{r,\ZZ}(\widetilde x_{\mfs \mft}\power{\lambda, l}) \mid  (\lambda, l), \mfs, \mft \text{ permissible}\}$;    and if
$$
\widetilde x\power{\lambda, l}_{\mfs \mft}  \widetilde x\power{\mu, m}_{\mfu \mfv}  =
\sum r(\nu, n, \alpha, \beta) \  \widetilde x\power{\nu, n}_{\alpha, \beta}
$$
in $B_r(\ZZ; -2N)$,  where the sum runs over all  $(\nu, n)$ and $\alpha, \beta$,  then
$$
\Phi_{r,\ZZ}(\widetilde x\power{\lambda, l}_{\mfs \mft} )  \Phi_{r,\ZZ}( \widetilde x\power{\mu, m}_{\mfu \mfv} ) =
\sum {'} \ r(\nu, n, \alpha, \beta) \ \Phi_{r,\ZZ}( \widetilde x\power{\nu, n}_{\alpha, \beta}),
$$
where now the sum is restricted to permissible $(\nu, n)$ and $\alpha, \beta$.    

$A^{\mathrm s}_r(\Bbbk)$ has $\Bbbk$--basis $\{\Phi_{r,\Bbbk}(\widetilde x_{\mfs \mft}\power{\lambda, l}) \mid  (\lambda, l), \mfs, \mft \text{ permissible}\}$.  Moreover, in $B_r(\Bbbk; -2N)$, we have
$$
\widetilde x\power{\lambda, l}_{\mfs \mft}  \widetilde x\power{\mu, m}_{\mfu \mfv}  =
\sum (r(\nu, n, \alpha, \beta) \otimes 1_\Bbbk ) \  \widetilde x\power{\nu, n}_{\alpha, \beta},
$$
and in $A^{\mathrm s}_r(\Bbbk)$,
$$
\Phi_{r,\Bbbk}(\widetilde x\power{\lambda, l}_{\mfs \mft} )  \Phi_{r,\Bbbk}( \widetilde x\power{\mu, m}_{\mfu \mfv} ) =
\sum {'} \ (r(\nu, n, \alpha, \beta) \otimes 1_\Bbbk ) \ \Phi_{r,\Bbbk}( \widetilde x\power{\nu, n}_{\alpha, \beta}),
$$
with the sum again restricted to permissible $(\nu, n)$ and $\alpha, \beta$.   This shows that $A^{\mathrm s}_r(\Bbbk)$ is a specialization of $A_r(\ZZ)$, as required.
\end{proof}

\subsection{Jucys--Murphy elements and seminormal representations} 
\label{symplectic Brauer seminormal}

 We verify that  the setting of \myref{Section}{subsection seminormal quotient} applies to the  Brauer algebras, their specializations   $B_{r}(\ZZ; -2N)$, and the quotients of these specializations acting on symplectic tensor space.
 
 We have $R = \ZZ[\bolddelta]$, the generic ground ring, and the quotient map $\pi : \ZZ[\bolddelta] \to \Z$ determined by $\bolddelta \mapsto  -2N$.  The kernel of this map is the prime ideal $\mathfrak p = (\bolddelta + 2 N)$.  The subring  of evaluable elements in $\FF = \Q(\deltabold)$ is $R_{\mathfrak p}$.    The subring of evaluable elements in $B_{r}(\FF; \bolddelta)$  is  $B_{r}(R_{\mathfrak p}; \bolddelta)$;  c.f.  \myref{Remark}{remark on evaluable elements}.   The sequence of Brauer algebras with the Murphy cellular basis satisfies  \eqref{diagram 1}--\eqref{diagram 6} according to  \myref{Theorem}{theorem:  Brauer algebras as diagram algebras}.  We have verified in \myref{Section}{symplectic integral BCA} that the quotient axioms \eqref{quotient axiom 1}--\eqref{quotient axiom 3} are satisfied by the maps  $\Phi_r$  from $B_r(\ZZ; -2N)$  to endomorphisms of symplectic tensor space.    Jucys--Murphy elements for the Brauer algebras were defined by Nazaorv \cite{MR1398116}.  It is shown in \cite{MR2774622} that these are an  additive  family of JM elements, with contents
\begin{align}\label{equation: contents for Brauer algebra}
\kappa((\lambda, l) \to (\mu, m)) = 
\begin{cases}
c(a), &\text{if   $\mu = \lambda \cup \{a\}$,}\\
1- \deltabold - c(a), &\text{if  $\mu = \lambda \setminus \{a\}$.}   
\end{cases}  
\end{align}
where the content $c(a)$ of a node $a$  of a Young diagram is the column index of $a$ minus the row index of $a$.   It is easy to check that the separation condition is satisfied.   It remains to check condition \eqref{assumption SN}.  

\begin{lem} \label{SN for symplectic Brauer}
 Let $\stt$ be an $-2N$--permissible path in $\Std_{r}$.  
Then $F_t$ is evaluable.
\end{lem}

\begin{proof}  We apply  \myref{Lemma}{suff cond for SN} and \myref{Remark}{suff cond for SN 1}.     Let $\mfs,  \mft$ be two paths of length $r \ge 1$  with $\mfs' = \mft'$ and with at least one of the paths $-2N$ permissible.    We have to show that
  $\kappa_\stt(r)-\kappa_\sts(r) \not\equiv 0  \mod (\bolddelta + 2 N)$.
  In order to reach a contradiction, assume  $\kappa_\stt(r)-\kappa_\sts(r) \equiv 0  \mod (\bolddelta + 2 N)$.
  This can only happen if one of the two edges $\stt(r-1) \to \stt(r)$  and $\stt(r-1) \to \sts(r)$ involves adding a node to $\stt(r-1)$ and the other involves removing a node. Assume wlog that
  $\stt(r-1) = (\lambda, l)$,  $\stt(r) = (\lambda \cup \{\alpha\}, l)$ for an addable node $\alpha$ of $\lambda$, and $\sts(r) = (\lambda \setminus \{\beta\}, l+1)$, for a removable node $\beta$ of $\lambda$.  Our condition is then $c(\alpha) + c(\beta) = 1 + 2N$.  But since $\lambda_1 \le N$, we have $c(a), c(\beta) \le N$ and $c(a) + c(\beta) \le 2N$. 
\end{proof}

\section{The  Brauer algebra on orthogonal tensor space}
\label{section Brauer orthogonal}

Let $V$ be an $N$--dimensional vector space over a field $\Bbbk$ with $\ch(\Bbbk) \ne 2$, with a non--degenerate, symmetric bilinear form $(\ , \ )$;  we will call such forms {\sf orthogonal forms} for brevity.   
For $r \ge 1$, let  $\Psi_r : B_r(\Bbbk; N) \to \End(V\tensor r)$ be Brauer's homomorphism defined  in \myref{Section}{section tensor space} using the orthogonal form.   The image $\im(\Psi_r)$ is known as the 
{\sf(orthogonal) Brauer centralizer algebra}.

Over general fields, the classification of orthogonal forms is complicated. However, if   the field is quadratically closed,  then one can easily show that $V$ has an orthonormal basis, or alternatively a basis
$\{v_i\}_{1 \le i \le N}$ whose dual basis $\{v_i^*\}$ with respect to the orthogonal form  is $v_i^* = v_{N+1 - i}$.  Therefore, we can  assume (when $\Bbbk$ is quadratically closed) that $V = \Bbbk^N$ with the standard orthogonal form $(x, y)  = \sum_i x_i y_{N+1 - i}$.

\begin{thm}[\cite{MR2500869}]  \label{orthogonal SW duality}
 Let $\Lambda : \Bbbk \Orth(V) \to \End(V\tensor r)$ denote the homomorphism corresponding to the diagonal action of the orthogonal group $ \Orth(V)$ on $V\tensor r$.   
\begin{enumerate}[leftmargin=*,label=(\arabic{*}), font=\normalfont, align=left, leftmargin=*]
\item 
 If $\Bbbk$ is a quadratically closed infinite field, with $\ch(\Bbbk) \ne 2$, then  $\im(\Psi_r) = \End_{\Orth(V)}(V \tensor r)$ and
$\im(\Lambda) = \End_{B_r(\Bbbk; N)}(V \tensor r)$.  
\item  The dimension of  $\im(\Psi_r)$ is independent of the field and  of the characteristic, for infinite quadratically closed fields $\Bbbk$ with $\ch(\Bbbk) \ne 2$.  
\end{enumerate}
\end{thm}

The special case of this  theorem when $\Bbbk=\mathbb{C}$  is due to Brauer  \cite{MR1503378}.  
Let us continue to assume, for now, that $\Bbbk$ is quadratically closed and that $V = \Bbbk^N$ with the standard orthogonal form.  When we need to emphasize the field we write $V_\Bbbk$ for $V$ and
$\Psi_{r, \Bbbk}$ for $\Psi_r$.

Let $\Psi_{r, \ZZ}$ denote the restriction of $\Psi_{r, \C}$ to $B_r(\ZZ; N)$; the image  $\im(\Psi_{r, \ZZ})$ is the $\ZZ$-subalgebra of $\End(V_\C\tensor r)$  generated by $E_i$ 
and $S_i$ for $1 \le i \le r-1$.   Let $V_\ZZ$ be the $\ZZ$-span of  the standard basis  $\{e_i \suchthat 1 \le i \le 2N\}$.  
Thus $V_\ZZ \tensor r \subset V_\C \tensor r$ and  $
\End_\ZZ(V_\ZZ \tensor r) \subset \End(V_\C \tensor r)$.   Since $E_i $ and $S_i$ leave $V_\ZZ \tensor r$ invariant, we 
can also regard $\im(\Psi_{r, \ZZ})$ as a $\ZZ$-subalgebra of 
 $\End_\ZZ(V_\ZZ \tensor r) $.   
 
 For any  (quadratically closed) $\Bbbk$,  $B_r(\Bbbk; N) \cong B_r(\ZZ; N) \otimes_\ZZ \Bbbk$ and $V_\Bbbk 
\cong V_\ZZ \otimes_\ZZ \Bbbk$.     For $a \in B_r(\ZZ; N)$ 
and $w \in V_\ZZ$,  we have
 $\Psi_{r, \Bbbk}(a \otimes 1_\Bbbk) (w \otimes 1_\Bbbk) = \Psi_{r, \ZZ}(a)(w)\otimes 1_\Bbbk $.  Therefore, we are in the 
situation of \myref{Lemma}{lemma specialization cd}, 
and there exists a map $\theta : \im(\Psi_{r, \ZZ}) \to \im(\Psi_{r, \Bbbk})$ making the diagram commute:
  \!\!\!\!\!\!
  \begin{equation}  \label{orthogonal CD}
\begin{minipage}{5cm}
\begin{tikzpicture}
 \draw (0,0) node {$B_r(\Z; N)$};\draw (3,0) node {\ \ $\im(\Psi_{r, \ZZ})$};
\draw[->](1.1,0) to node [above] {\scalefont{0.7} $\Psi_{r, \ZZ}$} (2.4,0);
\draw[->](0,-0.3) to node [right] {\scalefont{0.8}$\otimes 1_\Bbbk$} (0,-1.2);
\draw[->](3,-0.3) to node [right] {\scalefont{0.8}$\theta$} (3,-1.2);
\draw[->](1.1,-1.5) to node [below] {\scalefont{0.7} $\Psi_{r, \Bbbk}$} (2.4,-1.5);
\draw (0,-1.5) node {$B_r(\Bbbk; N)$};\draw (3,-1.5) node {\ \ $\im(\Psi_{r, \Bbbk})$};
\end{tikzpicture}\end{minipage} .
\end{equation}

\subsection{Murphy basis over the integers}
 \label{orthogonal integral BCA}
 
  \begin{defn}
Write $A^{\mathrm o}_\f(**) = \Psi_r(B_\f( **; N))$, where $**$ stands for $\C$, $\Q$, or $\Z$. 
 Thus $A^{\mathrm o}_\f(\Z)$ is the $\Z$--algebra generated by $E_i = \Psi_r(e_i)$   and $S_i = \Psi_r(s_i)$.  
 (The superscript ``o'' in this notation stands for ``orthogonal".)
\end{defn}

Let $R = \Z[\deltabold]$.  Endow $B_\k(R; \deltabold)$  with the  dual  Murphy 
cellular structure described in 
\myref{Section}{section: Murphy bases Brauer},
   with cellular basis
$$
\leftbrace y_{\mfs \mft}\power{\lambda, l} \suchthat  (\lambda, l) \in \widehat B_\k \text{ and } \mfs, \mft \in \Std_\k(\lambda, l) \rightbrace.
$$
By  
   \myref{Theorem}{theorem:  Brauer algebras as diagram algebras},
 the tower $(B_\k(R; \deltabold))_{\k \ge 0}$ satisfies the assumptions
\eqref{diagram 1}--\eqref{diagram 6} of 
\myref{Section}{subsection diagram algebras}.

 We want to show that the maps   $\Psi_{r, \ZZ} : B_\f(\Z; N) \to  A^{\mathrm o}_\f(\Z)$ satisfy the assumptions
 \eqref{quotient axiom 1}--\eqref{quotient axiom 3} of 
\myref{Section}{section quotient framework}.
 It will follow that Brauer's centralizer algebras $ A^{\mathrm o}_\f(\Z)$ are cellular over the integers. 

First we define  the appropriate permissible points in $\widehat B_\k$ and permissible paths in $\widehat B$.  

\begin{defn}  An $N$--permissible partition $\lambda$ is a partition such that $\lambda'_1  + \lambda'_2 \le N$.  
We say that  an element
$(\lambda, l) \in \widehat B_\f$ is $N$--permissible if $\lambda$ is $N$--permissible.  
   We let  $\widehat B^{\rm o}_{\f,  \perm}\subseteq \widehat B_{\f}$   denote  the subset of $N$--permissible points.  
   
   A path $\mft \in \Std_\f({\lambda, l})$ is $N$--permissible if $\mft (k)$ is $N$--permissible for all $0\leq k \leq \f$.
  We let  $\Std^{\mathrm o}_{\f,  \perm}(\lambda, l) \subseteq \Std_{\f}(\lambda, l)$   denote the subset 
  of   $N$--permissible paths. 

\end{defn}

Note that this set of permissible points satisfies condition \eqref{quotient axiom 1}.

 We will require the notion of a {\sf walled Brauer diagram}.    Consider Brauer diagrams with $a + b$ strands. 
 Divide the top vertices into a left cluster of $a$ vertices and a right cluster of $b$ vertices, and similarly for the bottom vertices.   The $(a,b)$--walled Brauer diagrams are those in which no vertical strand connects a left vertex and a right vertex,  and every horizontal strand connects a left vertex and a right vertex. 
For any ground ring $U$ and loop parameter $\delta$,  the $U$--span of $(a,b)$--walled Brauer diagrams is a unital  involution--invariant subalgebra of  $B_{a + b}(U; \delta)$, called the {\sf walled Brauer algebra}, and denoted $B_{a, b}(U; \delta)$. (One imagines a wall dividing the left and right vertices; thus the terminology ``walled Brauer diagram" and ``walled Brauer algebra".)

   \begin{defn} \label{sign of walled Brauer algebra}
If $d$ is an $(a, b)$-walled Brauer diagram, then the diagram obtained by exchanging the top and bottom vertices to the right of the wall is a permutation diagram.  Define the sign of $d$, denoted $\sgn(d)$ to be the sign of the permutation diagram.
\end{defn}

\begin{eg}
 For $a,b \in \mathbb{N}$,   let $e_{a, b}$ be the $(a, b)$--walled Brauer diagram with horizontal edges
 $\{\p 1, \p {a+b}\}$ and $\{\pbar 1, \pbar {a+b}\}$ and vertical edges $\{\p j, \pbar j\}$ for $j \ne 1, a+b$ . Then the permutation corresponding to $e_{a, b}$ is  the transposition $(1, a+b)$, and hence
 ${\rm sign}(e_{a,b})=-1$.  \end{eg}

\begin{defn} \mbox{}  \label{diagrammatic Pfaffians 1}
Let $a, b \ge 0$.  We define elements $\mathfrak d_{a, b}$ and $\mathfrak d'_{a, b}$ in $B_{a, b} \subseteq B_{a+b}$. 
\begin{enumerate}[leftmargin=*,label=(\arabic{*}), font=\normalfont, align=left, leftmargin=*]
\item  Let  $\mathfrak d_{a, b} = \sum_{d}  \sgn(d) \ d,$
 where the sum is over all $(a, b)$-walled Brauer diagrams.  
\item  Let  $\mathfrak d'_{a,  b} = \sum_d  \sgn(d) \ d,$
 where now the sum is over all $(a, b)$-walled Brauer diagrams of corank $\ge 1$. 
\end{enumerate}
\end{defn}

\begin{defn}  For a composition $\lambda$ , the row antisymmetrizer of $\lambda$ is
$A_\lambda = \sum_{\pi \in \mathfrak S_\lambda}  \sgn(\pi) \  \pi$.   In case $\lambda$ is a partition, we have $A_\lambda =  y_{\lambda'}$,  where $\lambda'$ is the conjugate partition.
\end{defn}

\begin{lem}  \label{factorization of d'}
For $a, b \ge 0$, there exists an element $\beta'_{a, b} \in  B_{a, b}(\Z; N)\subseteq  B_{a+b}(\Z; N)$ such that 
 $\mathfrak d'_{a, b} =
 A_{(a, b)} \beta'_{a, b}$.
\end{lem}

\begin{proof}   $\mathfrak S_a \times \mathfrak S_b$  acts freely (by multiplication on the left)  on the set $D'_{a, b}$ of $(a, b)$-walled Brauer diagrams with corank $\ge 1$, and 
$$
\sgn(w\, d) = \sgn(w) \sgn(d),
$$
for $w \in \mathfrak S_a \times \mathfrak S_b$ and $d \in D'_{a, b}$.   Choose a representative of each orbit of the action.  Then
$$
\mathfrak d'_{a, b} =   A_{(a, b)}  \sum_x  \sgn(x) x,
$$
where the sum is over the chosen orbit representatives.
\end{proof}

Given $\lambda$ a Young diagram with more than two columns, 
we vertically slice $\lambda$ into two parts
 after the second column. The left and right segments of the 
sliced partition are then defined as follows,
  $$\lambda^L= (\lambda'_1,\lambda_2')' \quad \lambda^R= 
  (\lambda'_{3}, \lambda_4',\ldots  )'.$$  
 
  \begin{defn} \mbox{}  \label{diagrammatic Pfaffians orthogonal case}
 \begin{enumerate}[leftmargin=*,label=(\arabic{*}), font=\normalfont, align=left, leftmargin=*]
\item  If $\lambda$ is a Young diagram with at most two columns, define $\mathfrak d_\lambda =
\mathfrak d_{\lambda'_1, \lambda'_2}$  and 
$\mathfrak d'_\lambda =
\mathfrak d'_{\lambda'_1, \lambda'_2}$.
 \item
For $\lambda$  a Young diagram, with more than 2 columns we write  
\begin{equation}  
\label{orthogonal a}
\mathfrak d_\lambda = \mathfrak d_{\la^L} \otimes 
y_{\lambda^R} \quad \text{and} \quad
\mathfrak d'_\lambda = \mathfrak 
d'_{\la^L}
 \otimes
y_{\lambda^R}.
\end{equation}
\item  For $(\lambda, l) \in \widehat B_r$, write
\begin{equation} \label{orthogonal b}
\mathfrak d_{(\lambda, l)} =  \mathfrak d_\lambda e_{r-1}\power l  
\quad \text{and} \quad
\mathfrak d'_{(\lambda, l)} =  \mathfrak d'_\lambda e_{r-1}\power l  .
\end{equation}
\end{enumerate}
\end{defn}

 \begin{rmk} It is immediate that  for all $r$ and for all  $(\lambda, l) \in \widehat B_r$,
\begin{equation} \label{orthogonal 1}
y_{(\lambda, l)} =  \mathfrak d_{(\lambda, l)} - \mathfrak d'_{(\lambda, l)}.
\end{equation} 
Moreover, it follows from 
\myref{Lemma}{factorization of d'}
that 
 \begin{equation}  \label{orthogonal 2}
  \mathfrak d'_{(\lambda, l)} =  y_{(\lambda, l)} \beta',
 \end{equation} 
 where $\beta' = \beta'_{\lambda'_1, \lambda'_2}$.
\end{rmk}

\ignore{
It follows that for a partition $\lambda$,  
\begin{equation}  \label{factorization of d' lambda}
\mathfrak d'_\lambda = y_\lambda \beta_{\widetilde \lambda_1, \widetilde \lambda_2}.
\end{equation}  
}

Fix $r \ge 1$.  The multilinear functionals on $V^{2r}$ of the form
$(w_1, \dots, w_{2r}) \mapsto \prod (w_i, w_j)$, where each $w_i$ occurs exactly once, are evidently 
${\rm O}(V)$--invariant.  Moreover, there are some evident relations among such functionals, stemming from the following observation.   If we take $r = N+1$ and fix disjoint sets $S, S' $ of size $N+1$ with
$S \cup S' = \{1, 2, \dots, 2N + 2\}$,  then
$(w_1, \dots, w_{2r}) \mapsto \det((w_i, w_j))_{i \in S, j \in S'}$  is zero, because the matrix
$((w_i, w_j))_{i \in S, j \in S'}$ is singular.   These elementary observations are preliminary to   the first and second fundamental theorems of invariant theory for the orthogonal groups.  
See the preamble to \cite[Theorem 2.17.A]{MR0000255}.
The following proposition depends on the second of these observations.

\begin{prop}  \label{orthogonal fund theorem 1}
If $a + b = N+1$, then $\mathfrak d_{a, b} \in \ker(\Psi)$.
Hence if $r \ge N+1$ and
$(\lambda, l) \in \widehat B_r$ with $\lambda'_1 + \lambda'_2 = N+1$, then
 $\mathfrak d_{(\lambda,l)} \in \ker(\Psi)$.  
\end{prop}

\begin{proof}    Set $r = N+1$. 
There exist linear isomorphisms $A : V\tensor{2r} \to \End(V\tensor{r})$  and 
$\eta: V\tensor{2r} \to (V\tensor{2r} )^*$.  
The proof consists of showing that
$\eta\circ A\inv \circ \Psi(\mathfrak d_{a, b})$ is a functional of the sort described above, 
$(\eta\circ A \inv \circ \Psi(\mathfrak d_{a, b}))(w_1 \otimes \cdots \otimes w_{2r}) = 
 \det((w_i, w_j))_{i \in S, j \in S'},$ 
 for suitable choice of $S, S'$.  Hence $\mathfrak d_{a, b} \in  \ker(\Psi)$.  
 This is explained in  \cite[Section 3.3]{MR1670662}  or 
  \cite[Lemma 3.3]{MR2979865}.  We have also provided an  exposition in 
   \iftoggle{arxiv}
 { \myref{Appendix}{appendix orthogonal}}
 {Appendix D in the arXiv version of this paper}.
\end{proof}

We can now verify axiom \eqref{quotient axiom 2}.    Let $\mft \in \Std_\f(\lambda, l)$ be a path which is not  $N$--permissible.  
Let $k\le r$ be the first index  such that $\mft(k) = (\mu, m)$  satisfies 
$\mu'_1 + \mu'_2 =   N+1$.
  It follows from
  \myref{Proposition}{orthogonal fund theorem 1}
that 
\begin{equation}  \label{orthogonal 3}
\mathfrak d_{(\mu, m)} \in \ker(\Psi).  
\end{equation}
By \eqref{orthogonal 2} and \eqref{x and y generators Brauer}, we have that 
$$  \mathfrak d'_{(\mu, m)} =  y_{(\mu, m)} \beta'= y_\mu \beta'  e_{k-1}\power m,$$ 
where $\beta' = \beta'_{\mu'_1, \mu'_2}$.   This exhibits $\mathfrak d_{(\mu, m)}$ as an element of
 $y_{(\mu, m)} B_\k(\Z; N)$.  Moreover, it shows that $\mathfrak d'_{(\mu, m)}$
  is a linear combination of Brauer diagrams of corank at least $m+1$, and therefore
$ \mathfrak d'_{(\mu, m)}  \in B_\k(\Z; N)^{\coldom (\mu, m)}$, using 
   \myref{Lemma}{Brauer block diagonal transition}. 
 Hence
\begin{equation}\label{orthogonal 4}
 \mathfrak d'_{(\mu, m)}  \in  y_{(\mu, m)} B_\k(\Z; N) \cap  B_\k(\Z; N)^{\coldom (\mu, m)}.
\end{equation}
Equations  \eqref{orthogonal 1}--\eqref{orthogonal 4}  show that  axiom \eqref{quotient axiom 2} holds.
It is shown in \cite[Theorem 3.4, Corollary 3.5]{MR951511} that 
$$
\dim_{\C}(A_r^{\mathrm o}(\C)) = 
\smashr{ \sum_{(\lambda, l) \in \widehat B_{r, \perm}^{\rm o}}   }\ 
 (\sharp \Std^{\mathrm o}_{r, \perm}(\lambda, l))^2.
$$
Thus axiom \eqref{quotient axiom 3} holds.

Since assumptions
 \eqref{quotient axiom 1}--\eqref{quotient axiom 3}   of 
\myref{Section}{section quotient framework} 
 are satisfied, we can produce a modified  dual Murphy basis  of  $B_\f(\Z; N) $,
 $$
\leftbrace \widetilde y_{\mfs \mft}\power{\lambda, l} \suchthat  (\lambda, l) \in \widehat B_r \text{ and } \mfs, \mft \in \Std_r(\lambda, l) \rightbrace,
$$
following the procedure described before 
\myref{Theorem}{basis theorem}.    The following theorem gives a cellular basis for Brauer's centralizer algebra acting on orthogonal  tensor space, valid over the integers.  It also gives two descriptions  of the kernel  of the map
 $\Psi_\f : B_\f(\ZZ; N) \to \End(V^{\otimes r})$, one by providing a basis of $\ker(\Psi_\f)$ over the integers, and the other by describing the kernel as the ideal generated by a small set of elements.  Each of these statements is a form of the second fundamental theorem of invariant theory for the orthogonal groups.

\begin{thm}    \label{theorem integral orthogonal BCA}
The integral (orthogonal) Brauer centralizer algebra    $A^{\rm o}_\f(\Z)$ is a cellular algebra over $\Z$ with basis 
$$\mathbb A^{\rm o}_\f = \leftbrace \Psi(\widetilde y_{\mfs \mft}\power{\lambda, l})\mid   
(\lambda, l) \in \widehat B^{\rm o}_{\f, \perm} \text{ and }  \mfs, \mft \in \Std^{\mathrm o}_{\f,  \perm}(\lambda, l)
 \rightbrace,
$$
with 
the involution $*$  determined by $E_i^* = E_i$ and $S_i^* = S_i$ 
and 
the partially ordered set $(\widehat B^{\rm o}_{\f,   \perm}, \coldomeq)$.  
 The ideal   $\ker(\Psi_\f) \subseteq B_r(\ZZ; N)$ has  $\ZZ$-basis 
 $$\kappa_r = \leftbrace  \widetilde y_{\mfs \mft}\power{\lambda, l}  \mid   
(\lambda, l) \in \widehat B_{\f, \perm} \text{ and   $\mfs$  or  $\mft$  is not $N$--permissible}
 \rightbrace.
$$
 Moreover,      for $r > N$,  $\ker(\Psi_\f)$ is the ideal generated by the set 
    $ \{\mathfrak{d}_{a, b} \mid  a+ b = N+1 \}.
 $
  For $r \le N$, $\ker(\Psi_\f) = 0$. 
\end{thm}

\begin{proof}    The construction of the cellular basis  of  $A^{\rm o}_\f(\Z)$ and of the basis of 
$\ker(\Psi_r)$  follows immediately from  
\myref{Theorem}{good paths basis theorem} 
 since \eqref{quotient axiom 1}--\eqref{quotient axiom 3} have been verified.

Now, if $r \leq N$ then  $\ker(\Psi_\f)=0$ since all paths on $\widehat B$ of length $\le N$ are $N$--permissible. 
For $r > N$,  the kernel is the ideal generated by all the $\mathfrak{d}_{(\mu, m)}$ such that $(\mu,m)$ is a   marginal point in $\widehat B_k$ for some $0 < k \le r$,  using  
\myref{Theorem}{good paths basis theorem}.
But the marginal points  are all of the form $(\mu,m)$ for some $\mu$ with $\mu'_1 + \mu'_2 = N+1$. 
Now by \eqref{orthogonal a} and \eqref{orthogonal b},
$$\mathfrak{d}_{(\mu, m)}=  \mathfrak{d}_{\mu} e_{k-1} \power m = 
(\mathfrak d_{\mu^L} \otimes y_{\mu^R} )  e_{k-1} \power m,
$$
and so the result follows.
 \end{proof}

\begin{rmk}
As in 
\myref{Remark}{7point8},
 our construction provides an integral form of the simple $B_r(\C;N)$-modules labeled by permissible partitions.   
\end{rmk}

\subsection{Murphy basis  over an arbitrary field}
 \label{orthogonal BCA over a field} 

We return to the general situation described at the beginning of \myref{Section}{section Brauer orthogonal}:
$V$ is a vector space of dimension $N$ over an arbitrary field $\Bbbk$,  with $\ch(\Bbbk) \ne 2$, 
with an orthogonal form, and 
for $r \ge 1$,   $\Psi_{r, \Bbbk} : B_{r}(\Bbbk; N) \to \End(V\tensor r)$ is Brauer's homomorphism.     The map $\Psi_{r, \Bbbk}$ actually depends upon the particular orthogonal form, and, in contrast to the symplectic case,  we may not assume in this generality that we are dealing with the standard orthogonal form on $\Bbbk^N$.  

For the modified Murphy basis $\{\widetilde y_{\mfs \mft}\power{\lambda, l} \}$ of $B_r(\ZZ; N)$,  we also write $\widetilde y_{\mfs \mft}\power{\lambda, l}$ instead of $\widetilde y_{\mfs \mft}\power{\lambda, l}  \otimes 1_\Bbbk$ for the corresponding basis element of $B_r(\Bbbk; N) \cong B_r(\ZZ; N) \otimes \Bbbk$.

  \begin{thm}  \label{field cellular basis 2}  Let $V$ be a vector space of dimension $N$ over a field $\Bbbk$ with $\ch(\Bbbk) \ne 2$.   Assume $V$ has an orthogonal form $(\ , \ )$, and let $\Psi_{r, \Bbbk} : B_{r}(\Bbbk; N) \to \End(V\tensor r)$ be Brauer's homomorphism defined using the orthogonal form. Write 
 $A^{\mathrm o}_\f(\Bbbk)$  for the Brauer centralizer algebra $\im(\Psi_{r, \Bbbk})$. 
  
  The algebra 
 $A^{\rm o}_\f(\Bbbk)$ is a cellular algebra over $\Bbbk$ with basis 
$$\mathbb A^{\rm o}_\f(\Bbbk) = \leftbrace \Psi_{r, \Bbbk}(\widetilde y_{\mfs \mft}\power{\lambda, l}  )\mid   
(\lambda, l) \in \widehat B^{\rm o}_{\f, \perm} \text{ and }  \mfs, \mft \in \Std^{\mathrm o}_{\f,  \perm}(\lambda, l)
 \rightbrace.
$$
with 
the involution $*$  determined by $E_i^* = E_i$ and $S_i^* = S_i$ 
and 
the partially ordered set $(\widehat B^{\rm o}_{\f,   \perm},  \coldomeq)$.  
The ideal  $\ker(\Psi_{r, \Bbbk})\subseteq  B_\f(\Bbbk;  N)$ has  basis  
$$ \kappa_{r} = 
\leftbrace  \widetilde y_{\mfs \mft}\power{\lambda, l}  
\mid   
(\lambda, l) \in \widehat B_{r} \text{ and }  \mfs \text{ or } \mft  \text{ not  $N$--permissible}
 \rightbrace.
$$  
Moreover,    for   $r > N$,  $\ker(\Psi_{r, \Bbbk})$ is the ideal generated by the set 
 $ \{\mathfrak{d}_{a, b} \mid  a+ b = N+1 \}.
 $
  For $r \le N$, $\ker(\Psi_{r, \Bbbk}) = 0$. 
\end{thm}
\begin{proof}
Assume first that $\Bbbk$ is infinite and quadratically closed.   In this case, we may assume that $V = \Bbbk^N$ with the standard orthogonal form, and moreover, we have the commutative diagram \eqref{orthogonal CD}.   Now we can argue exactly as in the proof of \myref{Theorem}{field cellular basis 1}, using the constancy of dimension of the orthogonal Brauer centralizer algebra from \myref{Theorem}{orthogonal SW duality}, to obtain the desired conclusions. 

Now consider the general case. Let $\overline \Bbbk$ be the algebraic closure of $\Bbbk$.  Extend the orthogonal form
to $V_{\overline \Bbbk} = V \otimes_\Bbbk \overline \Bbbk$ by $(v \otimes 1_{\overline \Bbbk},  w \otimes 1_{\overline \Bbbk}) = 
(v, w) \in \Bbbk \subset \overline \Bbbk$.   Let 
$\Psi_{r, \overline \Bbbk} : B_{r}(\overline \Bbbk; N) \to \End(V_{\overline \Bbbk}\tensor r)$ be the corresponding Brauer homomorphism.
It is easy to check that we are again in the situation of 
 \myref{Lemma}{lemma specialization cd}, and we 
have a commutative diagram:
\begin{equation}  \label{orthogonal CD 2}
\begin{minipage}{5cm}
 \begin{tikzpicture}
 \draw (0,0) node {$B_r(\Bbbk; N)$};\draw (3,0) node {$A_r^{\rm o}(\Bbbk)$};
\draw[->](1.1,0) to node [above] {\scalefont{0.7} $\Psi_{r, \Bbbk}$} (2.4,0);
\draw[->](0,-0.3) to node [right] {\scalefont{0.8}$\otimes 1_{\overline\Bbbk}$} (0,-1.2);
\draw[->](3,-0.3) to node [right] {\scalefont{0.8}$\eta$} (3,-1.2);
\draw[->](1.1,-1.5) to node [below] {\scalefont{0.7} $\Psi_{r, \overline\Bbbk}$} (2.4,-1.5);
\draw (0,-1.5) node {$B_r(\overline{\Bbbk}; N)$};\draw (3,-1.5) node {$ A_r^{\mathrm o}(\overline\Bbbk )$};
\end{tikzpicture}
\end{minipage} .
\end{equation}
Moreover, from the first paragraph of the proof, we know that $\mathbb A^{\rm o}_\f(\overline\Bbbk)$ is 
a $\overline \Bbbk$--basis of $A^{\rm o}_\f(\overline\Bbbk)$.
The map $\eta$ in \eqref{orthogonal CD 2} is injective,  by \myref{Remark}{remark injectivity of specialization}. 
If $y \in \kappa_r \subset B_r(\Bbbk; N)$,  then $0 = \Psi_{r, \overline\Bbbk}(y \otimes 1_{\overline \Bbbk}) =
\eta(\Psi_{r, \Bbbk}(y))$.  Since $\eta$ is injective, it follows that $\kappa_r \subseteq \ker(\Psi_{r, \Bbbk})$, and from this it follows that $\mathbb A^{\rm o}_\f(\Bbbk)$ spans $ A^{\rm o}_\f(\Bbbk)$.   But $\mathbb A^{\rm o}_\f(\Bbbk)$ is linearly independent over $\Bbbk$, because $\eta(\mathbb A^{\rm o}_\f(\Bbbk)) = \mathbb A^{\rm o}_\f(\overline\Bbbk)$ is linearly independent over $\overline \Bbbk$.   Now we conclude as in the proof of \myref{Theorem}{good paths basis theorem}  that
$\kappa_r$ is a basis of $\ker(\Psi_{r, \Bbbk})$ and that $\mathbb A^{\rm o}_\f(\Bbbk)$ is a cellular basis of 
 $A^{\rm o}_\f(\Bbbk)$.  To finish, it suffices to observe that $\kappa_r$ is contained in the ideal  generated by the elements 
 $\mathfrak{d}_{a, b}$, and this follows from  \myref{Theorem}{theorem integral orthogonal BCA}. 
\end{proof}  

 \begin{rmk}
Over  fields of characteristic not equal to 2, it is shown in \cite{MR2979865,MR2946824} that 
$\ker(\Psi_{r, \Bbbk})$ 
 is actually generated by the single element
   $ \mathfrak{d}_{ \lceil N/2 \rceil,   \lfloor N/2 \rfloor} $.
\end{rmk}

\begin{rmk}  
As in the symplectic case, 
\myref{Theorem}{field cellular basis 2}
extends the constancy of dimension  of the Brauer centralizer algebras 
$A^{\mathrm o}_\f(\Bbbk)$
and of $\ker(\Psi_{r, \Bbbk})$ from \myref{Theorem}{orthogonal SW duality} to all fields  $\Bbbk$ of characteristic different from $2$, 
and, moreover, to all orthogonal forms  on a finite dimensional $\Bbbk$--vector space.
\ignore{Moreover, as in the symplectic case, the Brauer centralizer algebra
$A^{\mathrm o}_\f(\Bbbk)$  acting on $V \tensor r$ is a specialization of the integral Brauer centralizer algebra, 
$A^{\mathrm o}_r(\Bbbk) \cong A^{\mathrm o}_r(\ZZ) \otimes_\ZZ \Bbbk$.}
\end{rmk}

We observe that 
the isomorphism type of  the Brauer centralizer algebra $A^{\mathrm o}_r(\Bbbk)$ is  independent of the choice of the orthogonal bilinear form;   for example, when $\Bbbk = \RR$, the field of real numbers, 
the isomorphism type of $A^{\mathrm o}_r(\RR)$  does not depend on the signature of the form:

\begin{cor}  Adopt the hypotheses of \myref{Theorem}{field cellular basis 2}.   The Brauer centralizer algebra  $A^{\mathrm o}_\f(\Bbbk)$ acting on $V^{\otimes r}$ is the specialization of 
the integral Brauer centralizer algebra $A^{\mathrm o}_\f(\ZZ)$,   i.e.
$A^{\mathrm o}_r(\Bbbk)  \cong  A^{\mathrm o}_r(\ZZ) \otimes_\ZZ \Bbbk$.  Consequently, the isomorphism type of $A^{\mathrm o}_r(\Bbbk)$ is independent of the choice of the orthogonal form on $V$.
\end{cor}

\begin{proof}  The proof is exactly the same as that of \myref{Corollary}{corollary field cellular basis 1}.  The final statement holds since $A^{\mathrm o}_r(\ZZ) \otimes_\ZZ \Bbbk$ doesn't depend on the choice of the bilinear form.
\end{proof}

\subsection{Jucys--Murphy elements and seminormal representations} 
\label{orthogonal Brauer seminormal}
The discussion in \myref{Section}{symplectic Brauer seminormal} before \myref{Lemma}{SN for symplectic Brauer} carries over to the orthogonal case with small changes.    However,  the condition of 
\myref{Lemma}{suff cond for SN} fails for even integer values of the parameter, so a different argument is needed to verify condition \eqref{assumption SN}.   This is done in   \cite{EG:2017}. 

\begin{eg}  Let $N = 2M$ be a positive even integer.   Let $\lambda = (M, M-1)'$,  $\mu^+ = (M+1, M-1)'$ and $\mu^- = (M-1, M-1)'$.  Then all three Young diagrams are $2M$--permissible, and 
$\kappa((\lambda, 0) \to (\mu^+, 0)) \equiv  \kappa((\lambda, 0) \to (\mu^-, 1)) \mod (\bolddelta - 2M)$.
\end{eg}


\nottoggle{arxiv}{}{

\appendix
\section{The symmetric group algebras and the Hecke algebras on tensor space}\label{symgroup}
 In this appendix, we review results of H\"arterich~\cite{MR1680384} regarding the action of the symmetric group algebra and the Hecke algebra on ordinary tensor space.
These results are a model for  our results  in  this paper
regarding the action of the Brauer algebras and walled Brauer algebras on tensor space.

\subsection{The action of the symmetric group on tensor space}
 Let $V$ be an $N$--dimensional complex vector space.  For $\f \ge 1$, the symmetric group 
$\mathfrak S_\f$
 acts on $V^{\otimes \f}$ by place permutations,   
 $$
( w_1  \otimes \dots \otimes 
 w_i \otimes w_{i+1} \otimes \dots \otimes w_\f
 )  \cdot s_i = 
 (w_1  \otimes \dots \otimes 
 w_{i+1} \otimes w_{i } \otimes \dots \otimes w_\f)
  $$
  Let $\sigma$ denote the resulting representation 
 $\sigma : \C \mathfrak{S}_\f \to  \End_\C(V^{\otimes \f})$.  We consider the restriction of $\sigma$ to $\ZZ \mathfrak S_\f \subset  \C \mathfrak S_\f$.

The following result is due to H\"arterich~\cite{MR1680384}.   
Actually,  H\"arterich proved this result more generally for 
 the action of the Hecke algebras $H_\f(\boldq)$ on tensor space.  We will provide a simple proof here   for the case of the symmetric group algebras, and then  indicate how this can  be quantized.  
 \def\n{N}
For a Young diagram $\lambda$,  let $\ell(\lambda) = \lambda'_1$ denote the length of $\lambda$, i.e. the number of rows.  Let $\widehat{\mathfrak S}_{\f, \n}$ denote the set of $\lambda \in \widehat{\mathfrak{S}}_\f$ with $\ell(\lambda) \le \n$.  

We need the following elementary remark about cellular algebras in general.

\begin{rmk} \label{quotient by order ideal}
 If $A$ is a cellular algebra with cell datum given as in 
 \myref{Definition}{c-d},
 and if $\Gamma$ is an order ideal in $\widehat A$,  then $A/A^\Gamma$ is a cellular algebra with respect to the natural involution $*$  inherited from $A$,  the partially ordered set $(\widehat A \setminus \Gamma,  \unrhd)$, and the cellular basis
$$
\{ c_{\sts \stt}^\lambda  + A^\Gamma \suchthat  \lambda \in \widehat A \setminus \Gamma \text{ and }  \sts, \stt \in \Std(\lambda)\}.
$$
\end{rmk}

\begin{thm}[\cite{MR1680384}]  \label{Harterich for S n}
The algebra $\sigma(\ZZ\mathfrak{S}_\f)$ is has cellular basis
$$
\{ \sigma (y^\lambda_{\sts\stt} ) \mid \lambda \in \widehat{ \mathfrak{S} }_{\f}  \text{ with } \ell(\lambda) \le \n \text{ and }  \sts, \stt \in \Std_{\f}(\lambda)\}, 
 $$
 with involution induced from the involution $\ast$  on $\ZZ\mathfrak{S}_\f$
and the partially ordered set $( \widehat{ \mathfrak{S} }_{\f, \n},  \coldomeq)$.  
The kernel of $\sigma$ in $\ZZ \mathfrak S_\f$  has $\ZZ$-basis
 $
\{ y^\lambda_{\sts\stt}   \mid  
\ell(\lambda) > \n \text{ and }   \sts, \stt \in \Std_{\f}(\lambda)
\}.  
 $
 In particular, $\sigma$ is faithful if $\f \le \n$.   For $\f >\n$, 
  $\ker(\sigma)$  is the ideal generated by the single element 
$y_{(1^{\n+1})}\in \ZZ\mathfrak{S}_\f$.  
\end{thm}

\begin{proof}  We could refer  to the framework of  
\myref{Section}{section quotient framework},
 but the situation here is simpler and it is more transparent to proceed directly.

 First we will show that if $\ell(\lambda) >\n$, then  all the dual Murphy basis elements
$y^\lambda_{\sts\stt}$ are in the kernel of $\sigma$.   Let  $\lambda$ be a Young diagram with $\ell(\lambda) > \n$ and let $\stt \in \Std_\f(\lambda)$.     Then there exists a $k \le \f$  such that 
$\mu = \stt(k)$ satisfies $\mu'_1 = \n+1$.    

Note that $y_{(1^{\n+1})}$ is the antisymmetrizer
$\sum_{w \in \mathfrak S_{N+1}} \sgn(w) w$.  Fix a basis $\{v_i\}$ of $V$.   
Then for any basis element
$v_{[i]} = v_{i_1} \otimes v_{i_2} \otimes \cdots \otimes v_{i_{N+1}}$ of $V^{\otimes N+1}$,  $v_{[i]}$ has repeated tensor factors, so $v_{[i]} \sigma(y_{(1^{\n+1})}) = 0$.   Thus $y_{(1^{\n+1})} \in \ker(\sigma)$.
 Since $y_{(1^{\n+1})}$ is a factor of $y_\mu$,  we have $y_\mu \in \ker(\sigma)$.   Let $\stt_1 = \stt_{[0,k]}$  and $\stt_2 =   \stt_{[k, \f]}$.   Then $y_\lambda \b_{\stt} = y_\lambda \b_{\stt_2} \b_{\stt_1} =
\v_{\stt_2}^* y_\mu  \b_{\stt_1}$,  using \eqref{abstract branching compatibility for path}, so $y_\lambda \b_{\stt}  \in \ker(\sigma)$.    Hence for all $\sts,\stt\in\Std_\k(\lambda)$,   $y^\lambda_{\sts\stt} =  (\b_\sts)^*  y_\lambda \b_{\stt} \in \ker(\sigma)$, as required. 

Now it follows that 
$$
\mathbb A_\f  = \{ \sigma (y^\lambda_{\sts\stt} ) \mid \lambda \in \widehat{ \mathfrak{S} }_{\f}  \text{ with } \ell(\lambda) \le N \text{ and }  \sts, \stt \in \Std_{\f}(\lambda)\}, 
 $$
 spans $\sigma(\ZZ \mathfrak S_\f)$ over $\ZZ$, and hence spans $\sigma(\C \mathfrak S_\f)$ over $\C$.  But it is known from Schur-Weyl duality that the dimension of  $\sigma(\C \mathfrak S_\f)$ is the cardinality of $\mathbb A_\f$, and consequently $\mathbb A_\f$ is linearly independent, and thus a $\ZZ$-basis of   $\sigma(\ZZ \mathfrak S_\f)$.  
     
 Now it follows, as in the proof of 
 \myref{Theorem}{good paths basis theorem},
 that $\ker(\sigma)$  is spanned by 
 $$
 \kappa_\f = \{ y^\lambda_{\sts\stt}   \mid  
\ell(\lambda) > N \text{ and }   \sts, \stt \in \Std_{\f}(\lambda)
\}, 
 $$
 and hence $\kappa_\f$ is a basis of $\ker(\sigma)$.    Finally,  it was shown above that every element of $\kappa_\f$ is in the ideal generated by $y_{(1^{N+1})}$, and therefore $\ker(\sigma)$ is generated as an ideal by this one element.
 
 Note that $\Gamma = \{ \lambda \in \widehat{\mathfrak S}_\f \suchthat  \ell(\lambda) > N\}$ is an order ideal 
 in $(\widehat{\mathfrak S}_\f, \coldomeq)$, and  
 $\ker(\sigma) = \spn(\kappa_r)$
 is just the corresponding involution--invariant two sided ideal  $(\ZZ \mathfrak S_\f)^\Gamma$.
  Therefore, cellularity of the quotient algebra 
 $(\ZZ \mathfrak S_\f)/\ker(\sigma)$ 
 follows from 
 \myref{Remark}{quotient by order ideal}.
\end{proof}

\subsection{The Hecke algebras}
We now give a brief sketch of how the results from the last section may be 
quantized.  We will 
assume that the reader is familiar  with
the  first properties  of Hecke algebras of symmetric groups (see for example \cite{MR812444}).

Let $S$ be an integral domain and $q\in S$ a unit.  The Hecke algebra $H_\f(S; q)$ is the unital $S$--algebra with generators $T_1, \dots, T_{\f-1}$  satisfying the braid relations and the quadratic relation
$(T_i - q) (T_i +1) = 0$.  For any $S$, the specialization  $H_\f(S; 1)$ is isomorphic to $S \mathfrak S_\f$.    
The generic ground ring for the Hecke algebras is the Laurent polynomial ring $R = \ZZ[\qbold, \qbold\inv]$, where $\qbold$ is an indeterminant.    We will write $H_\f(\qbold)$ for $H_\f(R; \qbold)$.   

The algebra $H_\f(\qbold)$ has an $R$--basis $\{T_w \suchthat w \in \mathfrak S_\f\}$, defined as follows:  if $w = s_{i_1} s_{i_2} \cdots s_{i_l}$ is a reduced expression for $w$ in the usual generators of $\mathfrak S_\f$,  then $T_w = T_{i_1} T_{i_2} \cdots T_{i_l}$,  independent of the reduced expression.   
Define $T_w^* = T_{w\inv}$,  $T_w^\dagger = (-\qbold)^{\ell(w)} T_w\inv$,  and  $T_w^\#  = (-\qbold)^{\ell(w)} T_{w\inv}\inv$.    Thus $\# = * \circ \dagger = \dagger \circ *$.    The operations $*$ and $\dagger$ are algebra involutions and $\#$ is an algebra automorphism.   For the symmetric group algebras, $\#$ agrees with the automorphism previously defined by $w^\# = \sgn(w) w$.

 For   $\lambda\vdash \f$,  let $x_\lambda =  \sum_{w \in \mathfrak S_\lambda} T_w$.
 Let $\stt^\lambda$ be the row reading tableaux of shape $\lambda$.  For any $\lambda$--tableau $\stt$,  there is a unique $w(\stt) \in \mathfrak S_\f$ with $\stt = \stt^\lambda w(\stt)$. 
   For $\lambda$ a partition of $\f$ and $\sts,\stt\in\Std_r(\la)$, let $x^\lambda_{\sts \stt} =  (T_{w(\sts)})^*  x_\lambda T_{w(\stt)}$.  Let $y_\lambda = (x_{\lambda'})^\#$,   and let $y^\lambda_{\sts \stt} =   (x_{\sts', \stt'}^{\lambda'})^\#$.

  \begin{thm}[The Murphy basis and dual Murphy basis, \cite{MR1327362}]  \label{thm Murphy basis of H n}  \mbox{}
  \begin{enumerate}[leftmargin=*,label=(\arabic{*}), font=\normalfont, align=left, leftmargin=*]
  \item
The set \ 
$
\mathcal X = \{  x_{\mfs  \mft}^\lambda  \suchthat  \lambda \in \widehat{\mathfrak S}_\f \text{ and } \sts, \stt  \in \Std_\f(\lambda) \}
$
is a cellular basis of $H_\f(R; \qbold)$,   with respect to the involution $*$ and the partially ordered set
$(\widehat{\mathfrak S}_\f, \unrhd)$.  
\item  The set \ 
$
\mathcal Y =  \{ y_{\mfs  \mft}^\lambda  \suchthat  \lambda \in \widehat{\mathfrak S}_\f \text{ and } \sts, \stt  \in \Std_\f(\lambda)\}$
is a cellular basis of $H_\f(R; \qbold)$,   with respect to the involution $*$ and the partially ordered set
$(\widehat{\mathfrak S}_\f, \coldomeq)$.  
\end{enumerate}
\end{thm}

\begin{thm}[\cite{EG:2012}]    The sequence of Hecke algebras $(H_\k(R; \qbold)_{\k \ge 0})$  with either the Murphy cellular bases or the dual Murphy cellular bases satisfies axioms \eqref{diagram 1}--\eqref{diagram 6}.   The branching diagram in each case is Young's lattice.  
The branching factors can be chosen so that the cellular basis obtained from ordered products of branching factors as in 
\myref{Theorem}{theorem abstract Murphy basis} 
 agrees with the Murphy basis (respectively the dual Murphy basis). 
\end{thm}

We will write $\b_\stt$ for the ordered product of the down--branching factors for the dual Murphy basis along a standard tableaux $\stt$,  as for the symmetric group algebras.   Thus the dual Murphy basis becomes $y_{\sts \stt}^\lambda = (\b_{\sts})^*  y_\lambda \b_\stt$.

Now for any composition $\lambda$ of $\f$, we extend the notation above so that $x_\lambda =  \sum_{w \in \mathfrak S_\lambda} T_w$. 
The permutation module $M_\lambda$ of $H_\f(\qbold)$ is 
$M_\lambda =  x_\lambda H_\f(\qbold)$.   Murphy showed that $M_\lambda$ has an $R$--basis
$\{x_\lambda T_{w(\stt)} \suchthat  \stt \text{ is a row standard $\lambda$--tableau}\}$.  

For any ring $S$, let $V_S =  S^N$, with the standard basis $e_1, \dots, e_N$ of unit vectors.
   There is a right action of 
 $H_\f(\qbold)$ on $V_R^{\otimes \f}$ with the following properties:  as an $H_\f(\qbold)$--module, 
 $V_R^{\otimes \f}$ is isomorphic to the direct sum of $M_\lambda$ over all compositions $\lambda$ of $\f$ with no more than $N$ parts, and  the specialization $(\ZZ \otimes_R  V_R^{\otimes \f},  \ZZ \otimes_R H_\f(\qbold))$,  with $\qbold$ acting as $1$,   is the place permutation action of $\ZZ \mathfrak S_\f$ on
 $V_\ZZ^{\otimes \f}$.    This is explained in \cite[Section 3]{MR1680384}  (although there the action is twisted by $\#$). 
  Write $\sigma : H_\f(\qbold)) \to \End_R(V_R^{\otimes \f})$ for the  representation corresponding to the action of the Hecke algebra $ H_\f(\qbold))$ on tensor space $V_R^{\otimes \f}$.  
 
 Now we have H\"arterich's theorem:
 
\begin{thm}[\cite{MR1680384}]    \myref{Theorem}{Harterich for S n}
remains valid with $\ZZ \mathfrak S_\f$ replaced by the Hecke algebra $H_\f(\qbold)$ over $R = \ZZ[\qbold, \qbold\inv]$ and  $\sigma$  by the representation of $H_\f(\qbold)$ on $V_R^{\otimes \f}$.
 \end{thm}
 
 \begin{proof}  The proof is essentially the same  as that of  the special case
  \myref{Theorem}{Harterich for S n}
 once we verify that $y_{(1^{N+1})} \in \ker(\sigma)$.   We have to show that  $y_{(1^{N+1})}$ is in the annihilator of $M_\mu$ for all compositions $\mu$ of $\f$  with no more than $N$ parts.   By  \cite[Lemma 4.12]{MR1327362},   for a composition $\mu$ of $\f$ and a partition
 $\nu$ of $\f$,  $M_\mu y_\nu = 0$ unless  $\mu \unlhd \nu$.    Take $\nu =  (N+1, 1^{\f- N -1})'$.  
 Then when $\mu$ has no more than $N$ parts,  $\mu \not\unlhd \nu$ and therefore
  $M_\mu y_\nu = 0$.   But $y_\nu = y_{(1^{n+1})}$.  
  
  One other point may deserve attention, namely that 
  $$
\mathbb A_\f  = \{ \sigma (y^\lambda_{\sts\stt} ) \mid \lambda \in \widehat{ \mathfrak{S} }_{\f}  \text{ with } \ell(\lambda) \le N \text{ and }  \sts, \stt \in \Std_{\f}(\lambda)\}, 
 $$
 is linearly independent over $R$.  But for this, it suffices that the specialization when $R \to \ZZ$ and $\qbold \mapsto 1$ is linearly independent over $\ZZ$, and that was shown in the proof of 
  \myref{Theorem}{Harterich for S n}.
 \end{proof}

\section{The walled Brauer algebras}  \label{section: walled Brauer}
\def\br{B}

 The walled Brauer algebras arise in connection with the invariant theory of the general linear group  $\GL(V)$ acting on mixed tensors space $V^{\otimes (r, s)} := V^{\otimes r} \otimes {V^*}^{\otimes s} $.     
The walled  Brauer algebras were studied by Turaev ~\cite{MR1024455},  Koike ~\cite{MR991410},   Benkart et  al.\  ~\cite{MR1280591},  and  Nikitin ~\cite{MR2251346}.   
Cellularity of walled Brauer algebras was proved by Green and Martin ~\cite{MR2293327} and by Cox et.\ al.\  ~\cite{MR2417984};  the latter authors show that
 walled Brauer algebras can be arranged into coherent cellular  towers.

\subsection{Definition of the walled Brauer algebras}
Let $S$ be a commutative ring with identity, with a distinguished element $\delta$.
The walled (or rational, or oriented)  Brauer algebra $\br_{r, s}(S; \delta)$
 is a unital subalgebra of the Brauer algebra $\br_{r+s}(S; \delta)$ spanned by certain $(r+s)$--strand Brauer diagrams.  
Divide the top vertices into a left cluster of $r$ vertices and a right cluster of $s$ vertices, and similarly for the bottom vertices.   The $(r,s)$--walled Brauer diagrams are those in which no vertical strand connects a left vertex and a right vertex,  and every horizontal strand connects a left vertex and a right vertex. 
 One can  check that the span of $(r,s)$--walled Brauer diagrams is a unital  involution--invariant subalgebra of $\br_{r+s}(S; \delta)$.\footnote{One imagines a wall dividing the left and right vertices; thus the terminology ``walled Brauer diagram" and ``walled Brauer algebra".}
 
   We have inclusions $\iota: \br_{r, s} \hookrightarrow \br_{r, s+1}$  by adding a  strand on the right, and 
$\iota' :  \br_{r, s} \hookrightarrow \br_{r+ 1, s}$ by adding a   strand on the left.    

 Label the top vertices of $(r, s)$--walled Brauer diagrams from left to right by
  $$\p{-r},  \dots,  \p {-2}, \p{-1}, \  \p1 , \p 2, \dots, \p s$$  and the bottom vertices by 
 $$\pbar{-r},  \dots,  \pbar {-2}, \pbar{-1},  \ \pbar1 , \pbar 2, \dots, \pbar s.$$  
 For $1\leq a \leq r$ and $1\leq b \leq s$ we let  $e_{a, b}$ denote the walled Brauer diagram with horizontal strands connecting $\p {-a}$ to $\p b$  and $\pbar {-a}$ to  $\pbar b$ and vertical strands connecting $\p j$ to $\pbar j$ for $j  \ne  -a, b$. 
  For $1 \le i < r$,  let $s'_i$ denote the transposition $s'_i = (-i, -(i+1))$,  regarded as an element of $B_{r, s}$  and for $1\leq i < s$   let $s_i = (i, i+1)$. 
 The walled Brauer algebra is generated by the elements $s'_i$,  $s_j$  and any single element $e_{a, b}$.

 \begin{lem}\label{walled LR isomorphism}
 $$
 B_{r, s}(S; \delta) \cong B_{s, r}(S; \delta).
 $$
 \end{lem}

\begin{proof}  Reflect walled Brauer diagrams left to right.
\end{proof}

 \subsection{The quotient algebras  $ \wbq_{r, s}$}
 We let 
 $$
 \mathfrak S_{r, s} = \mathfrak S\{-1, -2, \dots, -r\}  \times  \mathfrak S\{1, 2, \dots, s\},
 $$
 and 
 $$\wbq_{r, s} = \wbq_{r, s}(S) =  S\mathfrak S_{r, s} \cong  S\mathfrak S_r \otimes_S  S \mathfrak S_s.
 $$
  with left and right inclusion maps $\iota' :  \wbq_{r, s} \hookrightarrow \wbq_{r+ 1, s}$ and $\iota:  \wbq_{r, s} \hookrightarrow \wbq_{r, s + 1}$ as for the walled Brauer algebras.   
    The walled Brauer algebra $B_{r, s}(S; \delta)$ contains a subalgebra isomorphic to $\wbq_{r, s}$, spanned by the walled Brauer diagrams with rank  $r+s$,  i.e. with no horizontal strands.  
Moreover,  the ideal  in $B_{r, s}$ generated by any element $e_{a, b}$ is equal to the ideal  $J_{r,s}$ spanned by walled Brauer diagrams with rank strictly less than $r + s$,  and the quotient
 $B_{r, s}/J_{r,s} $ is isomorphic to $\wbq_{r, s}$.  
 \newcommand{\jj}{{\sf j}}
 Let $\jj$ or $x \mapsto x'$  denote the isomorphism  $\jj: \mathfrak S\{1, 2, \dots, r\} \to  \mathfrak S\{-1, -2, \dots, -r\}$ determined by $s_i \mapsto s'_i$.  
We endow $\wbq_{r, s}$ with the tensor product cell datum arising from the dual Murphy cellular bases of $S \mathfrak S_r$ and  $S \mathfrak S_s$:
 
 \begin{prop}
  The   algebra  $ \wbq_{r, s}$ has cellular basis 
  $$  
  \{ \jj(y^{\lambda(1)}_{\mfs \mft} )\otimes   y^{\lambda(2)}_{\mfu  \mfv}
   \mid   \sts,\stt  \in {\rm Std}_r(\lambda (1)),
   \stu,\stv  \in {\rm Std}_s(\lambda (2))
    \text{ for } \lambda=((\lambda(1), \lambda(2)) \in \widehat{\mathfrak{S}}_r\times\widehat{\mathfrak{S}}_s\}		.	$$
   with respect to the involution $\ast$  and the poset   $(\widehat{\mathfrak{S}}_r \times \widehat{\mathfrak{S}}_s , \coldomeq)$.  Here the product ordering is given by $(\lambda(1), \lambda(2)) \coldomeq  (\mu(1), \mu(2))$  if $\lambda(1)  \coldomeq \mu(1)$ and  $\lambda(2) \coldomeq\mu(2)$.  
 \end{prop}

\begin{rmk}
One can also lift the usual Murphy basis to a cellular basis of this algebra in an obvious fashion.  
\end{rmk}

\subsection{Cellularity of the walled Brauer algebras} \label{subsection: cellularity walled Brauer}
 Let $\bolddelta$ be an indeterminant over $\Z$.   Let $R = \Z[\bolddelta]$ be the generic ground ring for the Brauer algebras and $R' = \Z[\bolddelta^{\pm 1}]$.   Let $\F = \Q(\deltabold)$ denote the field of fractions of $R'$. 
We are going to show that  the algebras $\br_{r, s}(R'; \deltabold)$  are cyclic cellular algebras, with a cellular structure lifted from the dual Murphy cellular structure on the algebras $H_{r, s}(R')$.   (We use this cellular structure rather than  the ``ordinary"  Murphy cellular structure, because it is adapted to the action of the walled Brauer algebras on mixed tensor space.)   We show moreover that all the inclusions
$$
\iota: B_{r, s-1}(R'; \deltabold) \hookrightarrow B_{r, s}(R'; \deltabold) \quad \text{and} \quad
 \iota': B_{r-1, s}(R'; \deltabold)  \hookrightarrow B_{r, s}(R'; \deltabold)
$$
are coherent, that is 
a cell module of the larger algebra, when restricted to the smaller algebra, has an order preserving cell filtration; and likewise, a cell module of the smaller algebra, when induced to the larger algebra, has an order preserving cell filtration.   
 Since the algebras 
$\br_{r, s}(\F; \deltabold)$ are split semisimple, there is a well defined branching diagram governing the appearance of cell modules in these cell filtrations.
  Moreover, since the algebras are cyclic cellular,  there exist branching factors associated to the edges in the branching diagrams.

To obtain the results outlined above, one has to extract various single sequences from the double sequence of walled Brauer algebras, and apply the technique of the Jones basic construction, as in 
 \myref{Section}{subsection: Jones}.    Concurrently, one verifies axioms \eqref{diagram 1} to \eqref{diagram 6}.    
 
  Fix some $t \ge 0$.  
For $0 \le k \le t$,  take $\wbs_k = \br_{k, 0}(S; \delta) \cong  S \mathfrak S_k$, with inclusions $\iota': \wbs_k \hookrightarrow \wbs_{k+1}$.    For $j \ge 0$, take 
 $$\wbs_{t + 2j} = \br_{t+ j, j}(S; \delta) \quad \text {and} \quad  \wbs_{t + 2j + 1} =  \br_{t + j + 1, j}(S; \delta)
 .$$  with inclusions:
 $$
 \wbs_{t+2j}  \stackrel{\iota'}{\hookrightarrow} \wbs_{t+2j +1}  \stackrel{\iota}{\hookrightarrow}  \wbs_{t+2j + 2}.
 $$
  For $j \ge 1$, define
 $$
 e_{t+ 2j -1}  = e_{t + j ,  j} \in \wbs_{t+ 2j}  \quad \text{and}  \quad e_{t+ 2j} = e_{t+ j + 1, j} \in \wbs_{t+ 2j + 1}.
 $$
These elements satisfy the Temperley--Lieb relations $e_i e_{i\pm 1} e_i = e_i$, and $e_i e_j = e_j e_i$ if  $|i-j| \ge 2$ and so    \eqref{J-2} is satisfied.

 As with the walled Brauer algebras, we extract  a single sequence from the double sequence $H_{r, s}$.   Fix $t \ge 0$ as above, and let $\mathcal H_k = H_{k, 0}$ for
 $0 \le k \le t$.    For $j \ge 0$, set $\mathcal H_{t + 2 j} = H_{t+ j, j}$ and $\mathcal H_{t + 2j + 1} = H_{t+ j + 1, j}$.  
For $k \ge 2$,  with $J_{t+k} $ the ideal in $\wbs_{t + k}$  generated by 
$e_{t+k -1}$, we have $\wbs_{t+k}/J_{t+k} \cong \mathcal H_{t+k}$ and so  \eqref{J-3} is satisfied.

 The pair of towers of algebras  $(\wbs_{t + k}(R'; \deltabold))_{k \ge 0}$ and $(\mathcal H_{t + k}(R'))_{k \ge 0}$, where the algebra $\mathcal H_{t + k}$ are endowed with the dual Murphy cellular structure, 
satisfies  axioms \eqref{J-2}--\eqref{J-8} of 
 \myref{Section}{subsection: Jones},
 but \eqref{J-1} has to be replaced by:

\medskip
\noindent (J1$'$) \ \ $\wbs_{t} = \mathcal H_{t}$  and $\wbs_{t+1} = \mathcal H_{t+1}$.
\medskip

\noindent  This modified axiom (J1$'$)  is evident and the rest of axioms \eqref{J-2}--\eqref{J-8} are verified  just as in  \break ~\cite[Section 5.7]{MR2794027}, with the following additional observations:    Axiom \eqref{J-6} states that  \break  
$e_{i} \wbs_{i+ 1}  e_{i}  \wbs_{i+ 1} =  e_{i}  \wbs_{i+ 1}$ for $i \ge t+ 1$.    But we have 
$e_{i} \wbs_{i+ 1}  e_{i}  \wbs_{i+ 1} \supseteq e_{i}^2 \wbs_{i+ 1} = \deltabold e_{i} \wbs_{i+ 1}$,    Since $\deltabold$ is invertible in $R'$,  axiom \eqref{J-6} is verified.   Finally, to complete the verification of \eqref{J-8}, we have to check that 
 that each 
$\mathcal H_{i}$ for $i\geq t+1$ is a cyclic cellular algebra,  but this is evident.

\begin{lem}  \label{cellular tower for walled Brauer} For any choice of $t \ge 0$, 
the tower $(\wbs_{t + k}(R'; \deltabold))_{k \ge 0}$  is a coherent tower of cyclic cellular algebras, with branching diagram $\widehat A$ obtained by reflections from the branching diagram $\widehat{\mathcal H}$ of the tower $(\mathcal H_{t + k}(R'))_{k \ge 0}$. 
\end{lem}
\begin{proof}  Follows from the discussion above together with   \myref{Section}{subsection: Jones}.
\end{proof}

\begin{notation}    For $l \le \min(r, s)$,  write
$$
e_{r, s}\power l =  
   e_{r-l + 1, s-l + 1}  \cdots e_{r-1, s-1}   e_{r, s}.
$$
For  $\lambda = (\lambda(1), \lambda(2)) \in \widehat{\mathfrak S}_{r-l} \times \widehat{\mathfrak S}_{s-l}$, let 
\begin{equation} \label{walled Brauer antisymmetrizer}
y_{\lambda} =  \jj(y_{\lambda(1)}) \otimes y_{\lambda(2)} \quad \text{and} \quad
y_{(\lambda, l)} =  \jj(y_{\lambda(1)}) \otimes y_{\lambda(2)} e_{r, s}\power l.
\end{equation}
\end{notation}

The result regarding cellularity, coherence, and branching factors for the walled Brauer algebras is the following: 

\begin{prop} \label{B r s cellularity}
\mbox{}
\begin{enumerate}[leftmargin=*,label=(\arabic{*}), ref=\arabic{*},  font=\normalfont, align=left, leftmargin=*]
\item \label{brs cellularity 1}
 For $r, s \ge 0$,  $\br_{r, s}(R'; \bolddelta)$  is a cyclic cellular algebra.
\item \label{brs po set}
 The partially ordered set  $\widehat \br_{r, s}$  in the cell datum for $\br_{r, s}(R'; \bolddelta)$ is 
\begin{equation} \label{poset for B r s}
\widehat \br_{r, s} = \leftbrace(\lambda, l) \mid   
l \le \min(r, s),   \lambda = (\lambda(1), \lambda(2)) \in \widehat{\mathfrak S}_{r-l} \times \widehat{\mathfrak S}_{s - l}
   \rightbrace,
\end{equation}
with the partial order $(\lambda, l) \coldomeq(\mu, m)$  if $l > m$ or if $l = m$  and $\lambda \coldomeq \mu$ in
$\widehat{\mathfrak S}_{r-l} \times \widehat{\mathfrak S}_{s - l}$.   
\item   \label{brs generator}
 The element  $y_{(\lambda, l)}$ defined in \eqref{walled Brauer antisymmetrizer} lies in 
$B_{r, s}^{\coldomeq(\lambda, l)}$, and a 
 model for  the cell module $\Delta({\lambda, l})$ of $B_{r, s}$ is 
$$
\Delta({\lambda, l}) =  (y_{(\lambda, l)}  B_{r, s}  + B_{r, s}^{\coldom (\lambda, l)} )/  B_{r, s}^{\coldom (\lambda, l)}.
$$
\item  \label{brs coherence}
The inclusions
$$
\iota : B_{r, s-1}(R'; \deltabold) \subseteq B_{r, s}(R'; \deltabold) \quad \text{and} \quad  \iota':  B_{r-1, s}(R'; \deltabold) \subseteq B_{r, s}(R'; \deltabold)
$$
are   coherent.  
\item  \label{brs branching rules}
The branching rules associated with the inclusions $\iota$ and $\iota'$  are as follows.   
\begin{enumerate}[leftmargin=*,label=(\alph{*}), font=\normalfont, align=left, leftmargin=*]
\item
For the inclusion $\iota$,  if  $(\mu, m) \in \widehat B_{r, s-1}$ and $(\lambda, l) \in \widehat B_{r, s}$,  then 
$(\mu, m) \to (\lambda, l)$ if and only if $m = l$,  $\mu(1) = \lambda(1)$, and  $\mu(2) \subset \lambda(2)$;   or $m = l-1$,  $\mu(1) \supset \lambda(1)$, and $\mu(2) = \lambda(2)$.  
\item For the inclusion $\iota'$,  if  $(\mu, m) \in \widehat B_{r-s, s}$ and $(\lambda, l) \in \widehat B_{r, s}$,  then 
$(\mu, m) \to (\lambda, l)$ if and only if $m = l$,  $\mu(1) \subset \lambda(1)$, and  $\mu(2) =\lambda(2)$;   or $m = l-1$,  $\mu(1) = \lambda(1)$, and $\mu(2) \supset \lambda(2)$.  
\end{enumerate}   
\item \label{brs branching factors}
The branching factors associated with the inclusions $\iota$ and $\iota'$  can be chosen as follows.  We denote the branching factors arising from  restricting  a cell module by 
$\bb  {(\mu, l)} {(\lambda, l)} {r,s} $  and those arising from inducing a cell module by 
$\vv  {(\mu, l)} {(\lambda, l)} {r,s} $.
\begin{enumerate}[leftmargin=*,label=(\alph{*}), font=\normalfont, align=left, leftmargin=*]
\item For the inclusion $\iota$,
\begin{itemize}
\item  $\bb  {(\mu, l)} {(\lambda, l)} {r,s} =   \bb {\mu(2)} {\lambda(2)} {s-l}  e_{r, s-1}\power l $.
\item   $\vv  {(\mu, l)} {(\lambda, l)} {r,s} =   \vv {\mu(2)} {\lambda(2)} {s-l}  e_{r, s}\power l $.
\item   $\bb  {(\mu, l-1)} {(\lambda, l)} {r,s} = \jj(\vv {\lambda(1)} {\mu(1)} {r-l +1} )    e_{r, s-1} \power {l-1}$.
\item  $\vv  {(\mu, l-1)} {(\lambda, l)} {r,s} = \jj(\bb {\lambda(1)} {\mu(1)} {r-l +1} )    e_{r, s} \power {l}$.
\end{itemize}
\item  For the inclusion $\iota'$,
\begin{itemize}
\item  $\bb  {(\mu, l)} {(\lambda, l)} {r,s} =   \jj(\bb {\mu(1)} {\lambda(1)} {s-l}) \ e_{r-1, s}\power l $.
\item   $\vv  {(\mu, l)} {(\lambda, l)} {r,s} =   \jj(\vv {\mu(1)} {\lambda(1)} {s-l})\   e_{r, s}\power l $.
\item   $\bb  {(\mu, l-1)} {(\lambda, l)} {r,s} = \vv {\lambda(2)} {\mu(2)} {r-l +1}     e_{r-1, s} \power {l-1}$.
\item  $\vv  {(\mu, l-1)} {(\lambda, l)} {r,s} = \bb {\lambda(2)} {\mu(2)} {r-l +1}     e_{r, s} \power {l}$.
\end{itemize}
\end{enumerate}
\end{enumerate}
\end{prop}

\begin{proof} 
 If $r = 0$ or $s = 0$, all the statements are obvious.   For example, if $ r = 0$,  then
$B_{0, s}(R'; \delta) \cong R' \mathfrak S_s$,   and $\widehat B_{0, s} = \{\emptyset\} \times \widehat{\mathfrak S}_s \cong \widehat{\mathfrak S}_s$.    The inclusion $\iota$ is $R' \mathfrak S_{s-1} \subseteq R' \mathfrak S_s$,  and statements \eqref{brs branching rules} and \eqref{brs branching factors} give the known branching diagram and branching factors  for this inclusion. (For point \eqref{brs branching factors} we remark that  there are no steps in the branching diagram of the form $   (\mu, l-1) \to (\lambda,l)$ and we have that  $e^{(0)}_{r,0}=1$, $e^{(0)}_{0,s}=1$    and so the branching factors are all equal to those from the tower of symmetric groups.)  Thus, we can assume $r, s > 0$ for the remainder of the proof.

 For points \eqref{brs cellularity 1}--\eqref{brs generator},  assume first that $r \ge s$.   Set $t = r - s$, and construct the tower $(\wbs_k)_{k \ge t}$ as above.  Then $\wbs_{r + s}  = \br_{r, s}(R'; \deltabold)$, and points  \eqref{brs cellularity 1}--\eqref{brs generator} follow from 
 \myref{Lemma}{cellular tower for walled Brauer}.
   If $s > r$,  use the isomorphism $\br_{r, s} \cong \br_{s, r}$.  

For points  \eqref{brs coherence}--\eqref{brs branching factors},   first assume that $r \ge s >0$, and set $t = r - s$. Construct the tower $(\wbs_k)_{k \ge t}$ as above.   The inclusion $\iota$ is $\wbs_{r+s -1} \subseteq \wbs_{r + s}$, and so the assertions in  \eqref{brs coherence} and \eqref{brs branching rules} regarding $\iota$,  follow immediately  from 
 \myref{Lemma}{cellular tower for walled Brauer}.
 For the branching factors for $\iota$  in \eqref{brs branching factors}(a),  use in addition 
\myref{Theorem}{theorem:  closed form determination of the branching factors}
 as well as the definition of the idempotents $e_{t + k}$.  

It remains to consider the inclusion $\iota'$.  
 If $r = s$, we also obtain the assertions in \eqref{brs coherence}, \eqref{brs branching rules} and \eqref{brs branching factors}(b)  regarding the inclusion $\iota'$ by using the isomorphism $\br_{r-1, s} \cong \br_{s, r-1}$ of 
 \myref{Lemma}{walled LR isomorphism}.
 Now,  if $r > s \ge 0$,  construct the tower  $(\wbs_k)_{k \ge t}$ with 
$t = r-(s+1)$ instead of $t = r-s$.   Now the inclusion $\iota'$ is  $\wbs_{r+s -1} \subseteq \wbs_{r + s}$, so the assertions in \eqref{brs coherence}, \eqref{brs branching rules}, and \eqref{brs branching factors}(b)  regarding this inclusion follows from 
\myref{Lemma}{cellular tower for walled Brauer}
and
\myref{Theorem}{theorem:  closed form determination of the branching factors}.

At this point we have the assertions  in \eqref{brs coherence}--\eqref{brs branching factors}  regarding both inclusions $\iota$ and $\iota'$ in case $r \ge s$.   To handle the case $s >r$,  use the isomorphism
$B_{r, s} \cong B_{s, r}$ of 
\myref{Lemma}{walled LR isomorphism}.  
\end{proof}

\begin{lem}   \label{b and v branching factor compatibility walled case}
Let $(\lambda, l) \in \widehat B_{r, s}$,  and let $(\mu, m)$ be an element of 
$\widehat B_{r-1, s}$ or $\widehat B_{r, s-1}$ such that $(\mu, m) \to (\lambda, l)$.    
The $b$-- and $v$--branching factors satisfy the following compatibility relation:
\begin{equation}\label{branching factor intertwining walled case}
y_{(\lambda, l)} \bb {(\mu, m)}{(\lambda, l)}{r,s} = (\vv {(\mu, m)}{(\lambda, l)}{r, s})^* y_{(\mu, m)},
\end{equation}
\end{lem}

\begin{proof}  Straighforward computation using the compatibility relations for the branching factors
\eqref{dual branching coefficients for S n}, the definition
 \eqref{walled Brauer antisymmetrizer}, and 
 \myref{Proposition}{B r s cellularity}
 part \eqref{brs branching factors}.
\end{proof}

We can now produce many different analogues of the (dual) Murphy cellular basis for the walled Brauer algebras over $R'$.   Fix $r, s \ge 0$.    Consider any path  $\varepsilon = ((r_i, s_i))_{0 \le i \le r+s}$  in $\Z_{\ge 0} \times \Z_{\ge 0}$ from $(0,0)$ to $(r, s)$.    Thus  $(r_i , s_i) - (r_{i-1}, s_{i-1}) = (1, 0)$  or $(0, 1)$.    The sequence of algebras $(B_{\varepsilon_i}(R'; \deltabold))_{0 \le i \le r+s}$   is a  coherent tower of cyclic cellular algebras, with $B_{\varepsilon_i}(\F; \delta)$ split semisimple.   The branching diagram $\widehat B_\varepsilon$  and suitable branching factors for the tower of algebras
$(B_{\varepsilon_i}(R'; \deltabold))_{0 \le i \le r+s}$ 
are determined by   \myref{Proposition}{B r s cellularity}.

Fix $(r, s)\in \Z_{\ge 0} \times \Z_{\ge 0}$ and $\varepsilon$  a path from $(0, 0)$ to $(r, s)$ in  $\Z_{\ge 0} \times \Z_{\ge 0}$.  
For $1 \le k \le r+s$ and for $(\lambda, l) \in \widehat B_{\varepsilon_k}$, define $\Std_{\varepsilon}(\lambda, l)$ to be the set of paths on $\widehat B_\varepsilon$  from the root $\emptyset$ to $(\lambda, l)$.   Define $\Std_{\varepsilon, k}$ to be the set of all paths on  $\widehat B_\varepsilon$ from $\emptyset$ to some point of $\widehat B_{\varepsilon_k}$.   Similarly,  for $0 \le j < k \le r+s$,  let $\Std_{\varepsilon; j, k}$ denote the set of all paths on $\widehat B_\varepsilon$ from some point  $(\mu, m) \in \widehat B_{\varepsilon_j}$ to some point $(\lambda, l) \in \widehat B_{\varepsilon_k}$.

For a path   $
\mft  \in \Std_{\varepsilon; j, k}$   with initial vertex $\stt(j) = (\mu, m)$ and final vertex $\stt(k) = (\lambda, l)$,
 define $b_\mft$ and $v_\mft$ as ordered products of $b$-- and $v$--branching factors,  by
\begin{equation}\label{b t for truncated path walled case}
b_\mft =  \bb {\mft\powerr {k-1}} {\mft\powerr k} k    \bb {\mft\powerr {k-2}} {\mft\powerr {k-1}} {k-1} \cdots
\bb {\mft\powerr j} {\mft \powerr {j+1}} {j+1},
\end{equation}
and
\begin{equation}  \label{v t for truncated path walled case}
v_\mft = \vv {\mft\powerr j} {\mft \powerr {j+1}} {j+1} \cdots  \vv {\mft\powerr {k-2}} {\mft\powerr {k-1}} {k-1} 
\vv {\mft\powerr {k-1}} {\mft\powerr k} k ,
\end{equation} 
It follows from 
  \myref{Lemma}{b and v branching factor compatibility walled case}
  and induction on 
  $k-j$  that 
\begin{equation}\label{b and v compatibility for paths walled case}
y_{(\lambda, l)} b_\mft        =    v_\mft^*     y_{(\mu, m)}.  
\end{equation}
In particular, we have just defined $b_\mft$ and $v_\stt$ for $\mft \in \Std_{\varepsilon, k}(\lambda, l)$.

\begin{cor} \label{walled cellular basis 1}
 Let $\varepsilon$  be a path from $(0, 0)$ to $(r, s)$ in  $\Z_{\ge 0} \times \Z_{\ge 0}$.
 The tower  
$(B_{\varepsilon_i}(R'; \deltabold))_{0 \le i \le r+s}$ satisfies conditions  \eqref{diagram 1} to \eqref{diagram 6}.  
  The set
  \begin{equation}  \label{eqn walled Brauer Murphy basis 1}
\mathscr B^{ \varepsilon}_{r, s} = 
\leftbrace  y_{\mfs \mft}\power{\lambda, l} = 
 b_\mfs^*  y_{(\lambda, l)}  b_\mft \suchthat    
(\lambda, l) \in \widehat B_{r, s} \text{ and }  
\mfs, \mft \in \Std_{\varepsilon}({\lambda, l})
   \rightbrace
\end{equation}
is a cellular basis of  $B_{r, s}(R'; \deltabold)$.  The partially ordered set in the cell datum is
$(\widehat B_{r,s},  \coldomeq)$.  
\end{cor}

\begin{proof} Follows from  
\myref{Proposition}{B r s cellularity} and
\myref{Theorems}{theorem abstract Murphy basis} and
\myrefnospace{}{theorem:  closed form determination of the branching factors}.
\end{proof}

\begin{rmk}  For emphasis, we remind the reader that all the data entering into the definition of the cellular bases \eqref{eqn walled Brauer Murphy basis 1} are explicitly determined.
\end{rmk}

For the remainder of the paper, we fix our choice of $\varepsilon$ as follows:
\begin{equation}  \label{standard lattice path}
\varepsilon:  (0, 0) \to (1, 0) \to \cdots \to (r, 0) \to (r, 1) \to \cdots \to (r, s). 
\end{equation}With this convention  in place, we set $\mathscr B_{r,s}:=\mathscr B^\varepsilon_{r,s}$ for the remainder of the paper.  
This is consistent with the choice made in   ~\cite{MR3473643} and ~\cite{Werth:2015}.   An example of the resulting branching graph is depicted in \myref{Figure}{figureBRAUER22222} below, for $r = s = 2$.   

\begin{figure}[ht!]
 $$  \scalefont{0.8}
\begin{tikzpicture}[scale=0.7]
 \draw (0,0) -- (0,-2);  
\draw(0,-2) --(2,-4); 
 \draw(0,-2) --(-2,-4); 
\draw(-2,-4)--(0,-6);
 \draw(2,-4)--(0,-6);
  
  \draw(2,-4)--(2+2,-6);
  \draw(-2,-4)--(-4,-6);

\draw(4,-6)--(1,-8);\draw(-4,-6)--(1,-8);
\draw(0,-6)--(-1,-8);
\draw(0,-6)--(1,-8);

\draw(4,-6)--(4,-8);\draw(4,-6)--(7,-8);
\draw(-4,-6)--(-4,-8);\draw(-4,-6)--(-7,-8);

\draw (0,-2) node[anchor=center]{\begin{centering}\Ylinecolour{white}\Ynodecolour{white}\Yfillcolour{white}\Yvcentermath1\Yboxdim{10pt}\gyoung(;)\end{centering}
       }		;
      \draw (0,-2)node[anchor=center]{$ (\Yvcentermath1\Ylinecolour{black}\Ynodecolour{white}\Yfillcolour{white}\Yboxdim{7pt}\gyoung(;) \; , \;\varnothing)$ 
       }		;

\draw (0,-6) node[anchor=center]{\begin{centering}\Ylinecolour{white}\Ynodecolour{white}\Yfillcolour{white}\Yvcentermath1\Yboxdim{10pt}\gyoung(;)\end{centering}
       }		;
      \draw (0,-6)node[anchor=center]{$ (\Yvcentermath1\Ylinecolour{black}\Ynodecolour{white}\Yfillcolour{white}\Yboxdim{7pt}\gyoung(;) \; , \;\varnothing)$ 
       }		;


\draw (2,-4) node[anchor=center]{\begin{centering}$\Yvcentermath1\Ylinecolour{white}\Ynodecolour{white}\Yfillcolour{white}\Yboxdim{11.5pt}\gyoung(;;,;;) $\end{centering}
       }		;
      \draw (2,-4) node[anchor=center]{\begin{centering}$\left(\Yvcentermath1\Yvcentermath1\Ylinecolour{black}\Ynodecolour{white}\Yfillcolour{white}\Yboxdim{7pt}\gyoung(;,;) \; , \;\varnothing\right)$\end{centering}
       }		;

\draw (-2,-4) node[anchor=center]{\begin{centering}$\Yvcentermath1\Ylinecolour{white}\Ynodecolour{white}\Yfillcolour{white}\Yboxdim{11.5pt}\gyoung(;;,;;) $\end{centering}
       }		;
      \draw (-2,-4) node[anchor=center]{\begin{centering}$\left(\Yvcentermath1\Yvcentermath1\Ylinecolour{black}\Ynodecolour{white}\Yfillcolour{white}\Yboxdim{7pt}\gyoung(;;) \; , \;\varnothing\right)$\end{centering}
       }		;

\draw (4,-6) node[anchor=center]{\begin{centering}$\Yvcentermath1\Ylinecolour{white}\Ynodecolour{white}\Yfillcolour{white}\Yboxdim{11.5pt}\gyoung(;;,;;) $\end{centering}
       }		;
      \draw (4,-6) node[anchor=center]{\begin{centering}$\left(\Yvcentermath1\Yvcentermath1\Ylinecolour{black}\Ynodecolour{white}\Yfillcolour{white}\Yboxdim{7pt}\gyoung(;,;) \; , \;\gyoung(;)\right)$\end{centering}
       }		;

\draw (-4,-6) node[anchor=center]{\begin{centering}$\Yvcentermath1\Ylinecolour{white}\Ynodecolour{white}\Yfillcolour{white}\Yboxdim{11.5pt}\gyoung(;;,;;) $\end{centering}
       }		;
      \draw (-4,-6) node[anchor=center]{\begin{centering}$\left(\Yvcentermath1\Yvcentermath1\Ylinecolour{black}\Ynodecolour{white}\Yfillcolour{white}\Yboxdim{7pt}\gyoung(;;) \; , \;\gyoung(;)\right)$\end{centering}
       }		;


\draw (1,-8) node[anchor=center]{\begin{centering}$\Yvcentermath1\Ylinecolour{white}\Ynodecolour{white}\Yfillcolour{white}\Yboxdim{11.5pt}\gyoung(;;,;;) $\end{centering}
       }		;
      \draw (1,-8) node[anchor=center]{\begin{centering}$\left(\Yvcentermath1\Yvcentermath1\Ylinecolour{black}\Ynodecolour{white}\Yfillcolour{white}\Yboxdim{7pt}\gyoung(;) \; , \; \gyoung(;)\right)$\end{centering}
       }		;

\draw (-1,-8) node[anchor=center]{\begin{centering}$\Yvcentermath1\Ylinecolour{white}\Ynodecolour{white}\Yfillcolour{white}\Yboxdim{9pt}\gyoung(;;,;;) $\end{centering}
       }		;
      \draw (-1,-8) node[anchor=center]{\begin{centering}$\left(\Yvcentermath1\Yvcentermath1\Ylinecolour{black}\Ynodecolour{white}\Yfillcolour{white}\Yboxdim{7pt}\varnothing,\varnothing\right)$\end{centering}
       }		;

\draw (4,-8) node[anchor=center]{\begin{centering}$\Yvcentermath1\Ylinecolour{white}\Ynodecolour{white}\Yfillcolour{white}\Yboxdim{11.5pt}\gyoung(;;,;;) $\end{centering}
       }		;
      \draw (4,-8) node[anchor=center]{\begin{centering}$\left(\Yvcentermath1\Yvcentermath1\Ylinecolour{black}\Ynodecolour{white}\Yfillcolour{white}\Yboxdim{7pt}\gyoung(;,;) \; , \;\gyoung(;;)\right)$\end{centering}
       }		;

\draw (7,-8) node[anchor=center]{\begin{centering}$\Yvcentermath1\Ylinecolour{white}\Ynodecolour{white}\Yfillcolour{white}\Yboxdim{11.5pt}\gyoung(;;,;;) $\end{centering}
       }		;
      \draw (7,-8) node[anchor=center]{\begin{centering}$\left(\Yvcentermath1\Yvcentermath1\Ylinecolour{black}\Ynodecolour{white}\Yfillcolour{white}\Yboxdim{7pt}\gyoung(;,;) \; , \;\gyoung(;,;)\right)$\end{centering}
       }		;

\draw (-4,-8) node[anchor=center]{\begin{centering}$\Yvcentermath1\Ylinecolour{white}\Ynodecolour{white}\Yfillcolour{white}\Yboxdim{11.5pt}\gyoung(;;,;;) $\end{centering}
       }		;
      \draw (-4,-8) node[anchor=center]{\begin{centering}$\left(\Yvcentermath1\Yvcentermath1\Ylinecolour{black}\Ynodecolour{white}\Yfillcolour{white}\Yboxdim{7pt}\gyoung(;;) \; , \;\gyoung(;;)\right)$\end{centering}
       }		;

\draw (-7,-8) node[anchor=center]{\begin{centering}$\Yvcentermath1\Ylinecolour{white}\Ynodecolour{white}\Yfillcolour{white}\Yboxdim{11.5pt}\gyoung(;;,;;) $\end{centering}
       }		;
      \draw (-7,-8) node[anchor=center]{\begin{centering}$\left(\Yvcentermath1\Yvcentermath1\Ylinecolour{black}\Ynodecolour{white}\Yfillcolour{white}\Yboxdim{7pt}\gyoung(;;) \; , \;\gyoung(;,;)\right)$\end{centering}
       }		;

\draw (0,0) node[anchor=center] {\begin{centering}\Ylinecolour{white}\Ynodecolour{white}\Yfillcolour{white}\Yboxdim{8pt}\gyoung(;)\end{centering}};
\draw (0,0) node[anchor=center] {$(\varnothing| \varnothing)$};
  \end{tikzpicture}
$$
 \caption{The branching graph, $\widehat{B}_\varepsilon$ when $r = s = 2$.}
 \label{figureBRAUER22222}
 \end{figure}
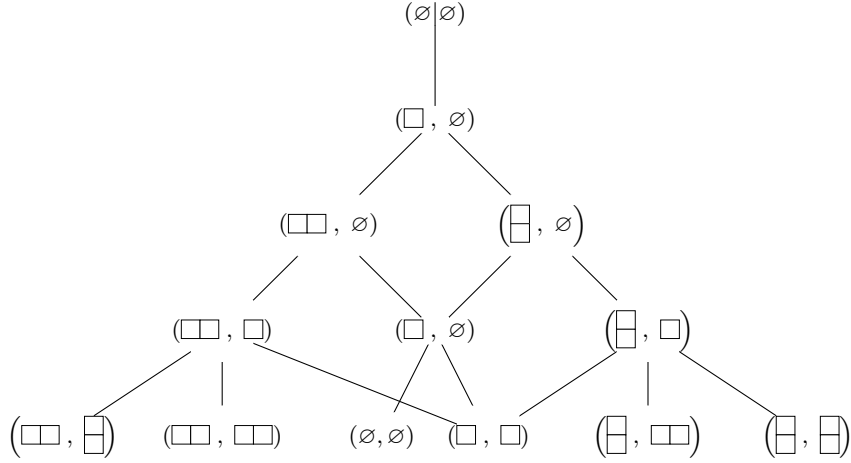

We  now show  
 that the cellular basis  $\mathscr B_{r, s}$
 in 
\myref{Corollary}{walled cellular basis 1}
  is actually a cellular basis over the generic ground ring $R = \Z[\deltabold]$.   The argument is a variant of those used in ~\cite[Section 6]{EG:2012}, and requires that we examine the transition matrix between the basis  $\mathscr B_{r, s}$ and the diagram basis of the walled Brauer algebra.
  
 \begin{defn}  Let $0 \le k \le s$ and $m \le  \min(r, k)$.     An $(r, s)$--walled  Brauer diagram $D$  is {\em of type} $(k, m)$ if $D$ has strands connecting the (nested) pairs of bottom vertices
\begin{equation}  
(\pbar{-r + m -1}, \pbar{k-m +1}),  \dots,  (\pbar{-r}, \pbar k)  \qquad  \text{($m$ strands)}. 
\end{equation}
and $D$ has no horizontal strands with both vertices in the interval $\{\pbar {-r + m}, \dots, \pbar {k-m}\}$.  
(A walled Brauer diagram of type $(0,0)$ is just an arbitrary  walled Brauer diagram.)
\end{defn}

\begin{lem} \label{corank and multiplication by branching factor walled case}
Let $D$ be an $(r,s)$--walled Brauer diagram  of type $(k, m)$ and corank $l$.
\begin{enumerate}[leftmargin=*,label=(\arabic{*}), font=\normalfont, align=left, leftmargin=*]
\item  Let $1 \le a \le k- m$  and let $D' =  D s_{a, k-m} e_{r, k-1}\power m$.  Then $D'$ is a walled  Brauer diagram of type $(k-1, m)$ and corankj $l$.
\item Let $1\le a \le r - m+1$  and let $D' =  D\,  \jj(s_{r - m + 1, a}) \, e_{r, k-1}\power{m-1}$.  Then $D'$ is a walled Brauer diagram of type $(k-1, m-1)$ and corank $l$.
\end{enumerate}
\end{lem}

\begin{proof} Evident from examining pictures.
\end{proof}

\begin{lem}    \label{corank sequence of  y lambda b t}
 Let $(\lambda, l) \in \widehat B_{r, s}$ and  $\mft \in \Std_\varepsilon(\lambda, l)$.  
 Write $\mft(j) = (\lambda\power j, l_j)$. 
\begin{enumerate}[leftmargin=*,label=(\arabic{*}), font=\normalfont, align=left, leftmargin=*]
\item
For $0 \le  k \le s$,  
$y_{(\lambda, l)}    b_{\mft_{[r+ k, r+s]}} $
is a  signed sum  of walled Brauer diagrams  of type $(k, l_{r+k})$ and corank $l$.
\item  
$y_{(\lambda, l)} b_\mft 
$
is a signed sum of walled Brauer diagrams of corank $l$. 
\end{enumerate}
\end{lem}

\begin{proof}  Statement (1) follows   from 
\myref{Lemma}{corank and multiplication by  branching factor walled case},
 \myref{Proposition}{B r s cellularity} 
part \eqref {brs branching factors}
 and induction on $s - k$.  In particular, we have that $y_{(\lambda, l)} b_{\mft_{[r, r+s] }}$  is  a signed sum of walled Brauer diagrams. 
  Now   
 $$
 y_{(\lambda, l)} b_\mft  = y_{(\lambda, l)} b_{\mft_{[r, r+s] }}  b_{\mft_{[0, r]}},
$$
and $ b_{\mft_{[0, r]}}$  is $\pm \sigma$ for some   $\sigma \in \mathfrak S\{-1, \dots, -r\}$,   so part (2) follows from part (1).
\end{proof}

 \begin{cor}  \label{transition matrix walled case}
 For $(\lambda, l) \in \widehat B_{r, s}$ and $\mfs, \mft \in \Std_\varepsilon({\lambda, l})$,
$y^{(\lambda, l)}_{\mfs \mft} = b_\mfs^* y_{(\lambda, l)}  b_\mft$ 
is a  signed sum of walled Brauer diagrams of corank $l$.
\end{cor}

\begin{thm}    \label{walled cellular basis 2}
Let $R = \Z[\deltabold]$ be the generic ground ring for the walled Brauer algebras. 
Fix  $r, s \ge 0$ and let  
  $\varepsilon$  be the standard lattice path from 
 from $(0, 0)$ to $(r, s)$ in  $\Z_{\ge 0} \times \Z_{\ge 0}$ given in 
 \eqref{standard lattice path}.   
 The tower  
$(B_{\varepsilon_i}(R; \deltabold))_{0 \le i \le r+s}$ satisfies conditions  \eqref{diagram 1} to \eqref{diagram 6}.  
In particular, 
  the set 
 $$
\mathscr B_{r, s} = 
\leftbrace  y_{\mfs \mft}\power{\lambda, l} = 
 b_\mfs^*  y_{(\lambda, l)}  b_\mft \suchthat    
(\lambda, l) \in \widehat B_{r, s} \text{ and }  
\mfs, \mft \in \Std_\varepsilon({\lambda, l})
   \rightbrace
$$
is a cellular basis of  $B_{r, s}(R; \deltabold)$.
 \end{thm}
 
 \begin{proof}    We know that $\mathscr B_{r, s}$ is a cellular basis of the walled Brauer algebra over $R' = R_0[\deltabold\inv]$, from  \myref{Corollary}{walled cellular basis 1}.  
 From 
  \iftoggle{arxiv}
{ \myref{Corollary}{transition matrix walled case}}
 {Corollary B.5 in the arXiv version of this paper},
  the elements of $\mathscr B_{r, s}$
  lie in $B_{r,s}(R; \deltabold)$,   and the transition matrix between $\mathscr B_{r, s}$ and the diagram basis is integer valued.   It follows that $\mathscr B_{r, s}$ is a cellular basis of $B_{r, s}(R; \deltabold)$.
 \end{proof}

  \begin{lem}  \label{walled Brauer block diagonal transition}
 If $D$ is an $(r, s)$--walled Brauer diagram of corank $\ge m+1$,  then for all $\mu \in  \widehat {\mathfrak S}_{r-m}  \times \widehat {\mathfrak S}_{s-m}$, 
$D$ is an element of the ideal $B_{r, s} (R; \deltabold)^{\coldom (\mu, m)}$.
 \end{lem}
 
\begin{proof} The proof is essentially the same as that of 
\myref{Lemma}{Brauer block diagonal transition}  
using the block diagonality of the transition matrix from the Murphy basis to the diagram basis, which is part of the statement of   
 \iftoggle{arxiv}
{ \myref{Corollary}{transition matrix walled case}}
 {Corollary B.5 in the arXiv version of this paper}.
\end{proof}
 
\subsection{Jucys--Murphy elements for the walled Brauer algebras}  \label{subsection JM elements WBA}
 
Let $R = \ZZ[\bolddelta]$ be the generic ground ring for the walled Brauer algebras,  and $\FF = \Q(\bolddelta)$ its field of fractions.
Write 
$B_{a, b}$ for $B_{a, b}(R; \bolddelta)$.



We define Jucys--Murphy elements in the double sequence of walled Brauer algebras, following 
\cite{MR2955190}, equations (2.12) and (2.13);     The JM elements are associated to the inclusions $\iota : B_{a, b}  \hookrightarrow B_{a, b+1}$ and  $\iota' : B_{a, b}  \hookrightarrow B_{a+1, b}$.
Define
\begin{equation}
\begin{aligned}
L_{a,b}^{a+1, b} &=  \sum_{1 \le k \le a}  \jj((k, a+1))  -  \sum_{1 \le k \le b} e_{a+1, k} \in B_{a+1, b},  \text{ and }  \\
L_{a,b}^{a, b+1} &=  \sum_{1 \le k \le b}  (k, b+1)    -  \sum_{1 \le k \le a}  e_{k, b+1} \in B_{a, b+1}.
\end{aligned}
\end{equation}
     One can easily check that  the elements  $L_{a,b}^{a+1, b} $ and $L_{a,b}^{a, b+1}$  are $*$--invariant and commute pointwise with $B_{a, b}$.  
If $\eta$ is a lattice path from $(0,0)$  to $(r, s)$ in $\ZZ_{\ge 0} \times \ZZ_{\ge 0}$, define $L_{\eta, i} =
 L_{\eta_{i-1}}^{\eta_i}$ for $1 \le i \le r+s$.
 \begin{lem} \label{wba JM 1}
  The sum $L_{\eta, 1} + \cdots + L_{\eta, r+s}$ is central in $B_{r, s}$ and independent of the choice of the lattice path $\eta$.  In fact, 
\begin{equation} \label{eqn wba JM 1}
 L_{\eta, 1} + \cdots + L_{\eta, r+s} = \sum_{1 \le j < k \le r} \jj((j, k))  + \sum_{1 \le j < k \le s} (j, k) -
 \sum_{\begin{subarray}c 
 1 \le j \le r \\
 1 \le k \le s
 \end{subarray}
 }  e_{j, k}.
\end{equation}
 \end{lem}
 
 \begin{proof}  As observed in \cite{MR2955190}, Section 2,  the equality \eqref{eqn wba JM 1} is straightforward to verify, as is the statement that the element on the right hand side is central.
 \end{proof}
 
For the remainder of this section, we fix  the lattice path $\varepsilon$ as in \eqref{standard lattice path}, and the corresponding sequence of walled Brauer algebras
$(B_{\varepsilon_i})_{1 \le i \le r+s}$, with branching graph $\widehat B_\varepsilon$.   Correspondingly, we define $L_i = L_{\varepsilon, i}$ for $1 \le i \le r+s$.     It follows immediately from the discussion above that
 $L_1, \dots, L_{r+s}$ is an additive family of Jucys--Murphy elements in the sense of \myref{Section}{subsection JM elements quotients}.   Consequently, as discussed in \myref{Section}{subsection JM elements quotients},   the JM elements  $L_1, \dots, L_{r+s}$ act diagonally on the seminormal basis of cell modules of $B_{r, s}^\FF$  and triangularly on the Murphy basis.  We proceed to determine the {\sf contents} of edges $(\lambda, l) \to (\mu, m)$ in $\widehat B_\varepsilon$, which are the eigenvalues of the JM elements $L_i$.

%
%
%


Recall that if $a$ is a node in a Young diagram, the content $c(a)$  of $a$ is the column index of $a$ minus the row index of $a$.   For a Young diagram $\lambda$, define $\alpha(\lambda)$ to be the sum of the contents of all of its nodes.

\begin{lem}  For $1 \le k \le r+s$,  and for $(\lambda, l)  = ((\lambda(1), \lambda(2)), l)\in \widehat B_{\varepsilon_k}$  the central element
$L_1 + \cdots +  L_k$ acts on the cell module $\Delta((\lambda, l))$  of $B_{\varepsilon_k}$  by 
$$
\beta((\lambda, l)) := \alpha(\lambda(1)) +  \alpha(\lambda(2)) - l \bolddelta.
$$
\end{lem}

\begin{proof}    \cite[Lemma 2.3]{MR2955190} or  \cite[Lemma 4.1]{MR2417984}.
\end{proof}

    
    For an edge
$(\lambda, l) \to (\mu, m)$  in $\widehat B_\varepsilon$, 
define  the {\sf content} of the edge to be
\begin{align}\label{equation: contents for Brauer algebra}
\kappa((\lambda, l) \to (\mu, m)) =  \beta((\mu, m)) - \beta((\lambda, l)) = 
\begin{cases}
 c(a)  &\text{if   $\mu = \lambda \cup \{a\}$,}\\
 - c(a) -  \bolddelta &\text{if  $\mu = \lambda \setminus \{a\}$.}   
\end{cases}  
\end{align}
For $(\lambda,l)\in\widehat{B}_{r, s}$ and $\stt  \in \Std_\varepsilon((\lambda, l))$,  and for $1 \le i \le r+s$,  define 
$
\kappa_\stt(i) = \kappa(\stt(i-1) \to \stt( i)).
$

 For $(\lambda, l) \in \widehat B_{r, s}$,  let $\{y_\mft  \suchthat  \mft \in \Std_\varepsilon((\lambda, l))\}$ denote the Murphy basis of the cell module $\Delta((\lambda, l))$ corresponding to the cellular basis of \myref{Theorem}{walled cellular basis 2}.  Let  $\{f_\mft  \suchthat  \mft \in \Std_\varepsilon((\lambda, l))\}$  denote the seminormal basis obtained from the Murphy basis as in  \myref{Section}{section some consequences from BEG}.


\begin{prop}  For $(\lambda, l) \in \widehat B_{r, s}$,  for  $\stt \in \Std_\varepsilon((\lambda, l))$, and for
 $i=1,\ldots,k,$, 
$$
f_{\stt} L_i = \kappa_\stt(i)f_\stt ,
$$
and 
there exist scalars $r_{\sts}\in R,$ for ${\sts}\in\Std_\varepsilon({(\lambda,l)}) $ such that  
\begin{align*} 
y_{\stt}L_i=\kappa_{\stt}(i)y_{\stt}+\sum_{{\sts}\rhd{\stt}}r_{\sts}y_{\sts}. 
\end{align*}
\end{prop}

\begin{proof}   \cite{BEG}  Proposition 4.5 and Theorem 4.7.
\end{proof}

\begin{rmk} \label{WBA JM separation condition}
It is straighforward to check that the sequence of JM elements $L_1, \dots, L_{r+s}$ satisfy the separation condition of Mathas.from ~\cite[Section 3]{MR2414949}; that is, the content sequence 
$(\kappa_\mft(i))_{1 \le i \le r+s}$ of a path determines the path.
\end{rmk}

\begin{rmk}  Using  \myref{Lemma}{wba JM 1}, this discussion shows more generally that for any lattice path $\eta$  from  $(0,0)$  to $(r, s)$ in $\ZZ_{\ge 0} \times \ZZ_{\ge 0}$,  the family of JM elements
$L_{\eta, 1}, \dots, L_{\eta, r+s}$   is an additive family of JM elements for the sequence of algebras
$(B_{\eta_i})_{0 \le i \le r+s}$.  
\end{rmk}

\section{The walled Brauer algebra on mixed tensor space}\label{section Brauer walled tensor}
\label{appendix wba tensor space}
\subsection{Action on mixed tensor space}  \label{subsection wba on mixed tensor space}
Let $V$ be an $N$--dimensional complex vector space.  For $r, s \ge 0$,  the general linear group $G = \GL(V)$  acts on the space of mixed tensors  $V^{\otimes (r, s)} := V^{\otimes r} \otimes {V^*}^{\otimes s} $.
One can define an action of the walled Brauer algebra $B_{r, s}(\C; N)$ on $V^{\otimes (r, s)}$, 
$$\Theta_{r, s} : B_{r, s}(\C; N) \to \End(V^{\otimes (r, s)}),$$
whose image is the centralizer algebra  for the action of $\GL(V)$  on $V^{\otimes (r, s)}$.  
The action of $B_{r, s}(\C; N)$ on $V^{\otimes (r, s)}$ is described as follows:
We write a simple tensor in $V^{\otimes (r, s)}$ as 
$$
u_{-r} \otimes \cdots \otimes u_{-1} \otimes u^*_1 \otimes \cdots \otimes u^*_s,
$$
where the $u_i$  are arbitrary elements of $V$ and the $u^*_j$ are arbitrary elements of $V^*$.  
  Let $\{v_i\}$ be any basis of $V$ and let $\{v_i^*\}$  be the dual basis of $V^*$.
Define $E \in \End(V \otimes V^*)$  by $E(u\otimes u^*) = u^*(u) \sum_i v_i \otimes v_i^*$.   Then $E$ is independent of the choice of the basis  $\{v_i\}$.   Define $E_{a, b} \in \End(V^{\otimes (r, s)})$  by $E$ acting in the
$-a$ and $b$ tensor positions.  Let $S'_i$ be the operator that interchanges the tensorands in the 
$-i$--th and $-(i+1)$--st positions, and let $S_j$ be the operator that  interchanges the tensorands in the 
$j$--th and $(j+1)$--st positions.  The homomorphism $\Theta_{r, s}$ is determined by 
$s_i'  \mapsto S'_i$,  $s_j \mapsto S_j$  and $e_{a, b} \mapsto E_{a, b}$.  

We have embeddings $\iota :  \End(V^{\otimes(r, s)}) \to  \End(V^{\otimes(r, s+1)}) $ by $T \mapsto T \otimes \id_{V^*}$  and $\iota' :  \End(V^{\otimes(r, s)}) \to  \End(V^{\otimes(r+1, s)}) $  by 
$T \mapsto   \id_{V}  \otimes   T $.   The homomorphisms $\Theta_{r, s}$ respect these embeddings,
$\Theta_{r, s+1} \circ \iota =  \iota \circ \Theta_{r, s}$  and similarly for $\iota'$.   Therefore, we will  just write $\Theta$ instead of $\Theta_{r,s}$.  

 Fix a basis $\{v_i\}$ of $V$ and the dual basis $\{v^*_i\}$ of $V^*$ as above.  For any multi-index
$[i] =( i_{-r}, \dots, i_{-1}, i_1, \dots, i_s)$, write $v_{[i]} = v_{i_{-r}} \otimes \cdots \otimes v_{i_{-1}} \otimes v^*_{i_1} \otimes \cdots \otimes v^*_{i_s}$.  The set of $v_{[i]}$ as $[i]$ varies over all multi-indices is a basis of $V\tensor{(r, s)}$.     Now we can verify that the matrix of $E_{a, b}$ with respect to this basis
is $E_{a, b}([i], [j]) = \delta(i_{-a}, i_b) \delta(j_{-a}, j_b) \prod_{k \ne -a, b} \delta(i_k, j_k)$.  Thus, the matrix of $E_{a, b}$ is symmetric.  Similarly, the matrices for the generators $S_l$ and $S_{-l}$ are symmetric. 
Define $*$ on  $\End(V^{\otimes (r, s)})$  to be matrix transposition with respect to this basis;  then $*$ is an algebra involution on  $\End(V^{\otimes (r, s)})$, and $\Theta$ is involution preserving.

\subsection{Verification of the quotient axioms} \label{subsection wba quotient axioms}
Let $R = \Z[\deltabold]$ be the generic ground ring for the walled Brauer algebras.  For fixed $(r, s)$, endow $B_{r, s}(R; \deltabold)$ with the cellular structure described in 
\myref{Theorem}{walled cellular basis 2}
 and cellular basis
$$
\mathscr B_{r, s} = 
\leftbrace  y_{\mfs \mft}\power{\lambda, l} = 
 b_\mfs^*  y_{(\lambda, l)}  b_\mft \suchthat    
(\lambda, l) \in \widehat B_{r, s} \text{ and }  
\mfs, \mft \in \Std_\varepsilon({\lambda, l})
   \rightbrace
$$
   This involves choosing the lattice path $\varepsilon$ as in \eqref{standard lattice path}, and placing $B_{r, s}(R; \deltabold)$ in the single sequence of algebras $ B_{\varepsilon_i}(R; \deltabold)$, with $B_{r, s}(R; \deltabold) = B_{\varepsilon_{r+s}}(R; \deltabold)$.   Evidently, this sequence can be extended to an infinite tower of algebras, and this infinite tower satisfies 
the assumptions
\eqref{diagram 1}--\eqref{diagram 6} of 
\myref{Section}{subsection diagram algebras}, as has been shown in 
\myref{Appendix}{subsection: cellularity walled Brauer}.

\begin{defn}Write $A_{r, s}\mixed(**) = \Theta(B_{r,s}( **; N))$, where $**$ stands for $\Z$, $\Q$ or $\C$.
 Thus $A_{r, s}\mixed(\Z)$ is the $\Z$--algebra generated by  the operators $E_{a, b} $   and $S'_i $, and $S_j$.  
 \end{defn}
 We want to show that the maps $\Theta: B_{\varepsilon_i}(\Z; N) \to A_{\varepsilon_i}\mixed(\Z)$  satisfy assumptions  \eqref{quotient axiom 1}--\eqref{quotient axiom 3} of 
\myref{Section}{section quotient framework}.  It will follow that  the concrete centralizer algebras $A_{r, s}\mixed(\Z)$  are cellular over the integers. 
 The appropriate collection of permissible paths is the following.
 
\begin{defn}   \label{defn permissibility for WBA}
 An $N$--permissible pair of partitions  $\lambda = (\lambda(1), \lambda(2))$ is a pair  such that $ {\lambda(1)}'_1  +  {\lambda(2)}'_1\le N$.    That is, the sum of the lengths of the first columns of $\lambda(1)$ and $\lambda(2)$  is no more than $N$.  
    Write $\widehat B\mixed_{r, s,  \perm}$   for the set of $N$--permissible points in $\widehat B_{r, s}$.
 
An element
$(\lambda, l) \in \widehat B_{r,s}$ is $N$--permissible if $\lambda$ is $N$--permissible.  
A path $\mft \in \Std_{\varepsilon}(\lambda, l)$ is $N$--permissible if $\mft\powerr k$ is $N$--permissible for all $0\leq k \leq r+s$.
 Write  $\Std\mixed_{\varepsilon,  \perm}({\lambda, l})$   for the set of   $N$--permissible paths  on $\widehat B_\varepsilon$ of shape $(\lambda, l)$.
\end{defn}

Recall the elements  of the walled Brauer algebras $\mathfrak d_{a, b}$  and $\mathfrak d'_{a, b}$ defined 
in 
\myref{Definition}{diagrammatic Pfaffians 1}.
Fix $(r, s)$ and 
let  $\lambda = (\lambda(1), \lambda(2)) \in \widehat{\mathfrak S}_{r-l} \times \widehat{\mathfrak S}_{s-l}$ for some $l$.  
\ignore{
Write $\widetilde {\lambda(1)}$, and  
 $\widetilde {\lambda(2)}$  for the transposed partitions.}
   Write $\lambda(1)_{>1}$ and $\lambda(2)_{>1}$ for the partitions with the first columns removed.  Write $a$ for the length of the first column of $\lambda(1)$ and $b$ for the length of the first column of $\lambda(2)$.
 Define
\begin{equation} \label{d lambda walled case}
\mathfrak d_\lambda =  \jj( y_{\lambda(1)_{>1}} ) \otimes \mathfrak d_{a, b }  \otimes  y_{\lambda(2)_{>1}}  
\quad\text{ and }\quad\mathfrak d'_\lambda =  \jj( y_{\lambda(1)_{>1}} ) \otimes \mathfrak d'_{a, b}  \otimes  y_{\lambda(2)_{>1}}.
\end{equation}
These are elements of $B_{r-l, s-l}$.    If $(\lambda, l) \in \widehat B_{r, s}$, define
\begin{equation} \label{d lambda l walled case}
\mathfrak d_{(\lambda, l)} =   \mathfrak d_{\lambda} e_{r, s}\power l, \quad \text{ and }\quad
\mathfrak d'_{(\lambda, l)} =   \mathfrak d'_{\lambda} e_{r, s}\power l.
\end{equation}
Thus for all $(r, s)$ and for all $(\lambda, l) \in  \widehat B_{r, s}$, 
\begin{equation} \label{walled 1}
y_{(\lambda, l)}   = \mathfrak d_{(\lambda, l)} - \mathfrak d'_{(\lambda, l)}.
\end{equation}
Because of 
\myref{Lemma}{factorization of d'}, 
we also have
\begin{equation}
 \mathfrak d'_{(\lambda, l)} =   y_{(\lambda, l)} \beta',
\end{equation}
where $\beta' = \beta'_{ {\lambda(1)}'_1,  {\lambda(2)}'_1}$.

For $m > 0$,  the multilinear functionals on $V^{m} \times (V^*)^m$ of the form
$$(w_1, \dots, w_m) \times (w^*_1, \dots, w^*_m) \mapsto  \prod_{i = 1}^m  w^*_i(w_{i\cdot \pi})$$  are evidently $\rm{GL}(V)$--invariant.  There are some obvious relations among these functionals, namely if $m = N+1$,  then $(w_1, \dots, w_m) \times (w^*_1, \dots, w^*_m) \mapsto \det((w^*_i(w_j)))$ is zero since the matrix $(w^*_i(w_j))$ is singular.  The following proposition reduces ultimately to this observation.

\begin{prop}  \label{mixed fund theorem 1}
Let $a, b$ satisfy $a + b = N+1$ and consider $\mathfrak d_{a, b} \in B_{a, b}$.  Then $\mathfrak d_{a, b} \in \ker(\Theta)$.   Hence if $(\lambda, l) \in \widehat B_{r, s}$  and  the sum of the lengths of the first column of $\lambda(1)$ and the first column of $\lambda(2)$ is $N+1$, then 
$\mathfrak d_{(\lambda, l)} \in \ker(\Theta)$.   
\end{prop}

\begin{proof} The proof is similar to that of 
  \myref{Proposition}{orthogonal fund theorem 1}.   
Details are provided in 
 {\myref{Appendix}{walled Brauer fund thm appendix}}.
\end{proof}

We can now verify axiom \eqref{quotient axiom 2}.   Let $\mft \in \Std_{\varepsilon}(\lambda, l)$ be a  path which is 
not $N$--permissible.    Since $\mft$ is not $N$--permissible, there exists a $0<k \le r+s$ such that
$\mft(k) = (\mu, m) = ((\mu(1), \mu(2)), m)$ and the sum of the lengths of the first column of $\mu(1)$ and the first column of $\mu(2)$ exactly equals $N+1$.   Hence, it follows from 
\myref{Proposition}{mixed fund theorem 1}
that 
\begin{equation} \label{walled 2}
\mathfrak d_{(\mu, m)} \in \ker(\Theta). 
\end{equation}
   In case $0<k \le r$,    $(\mu, m) = ((\mu(1), \emptyset), 0)$ and
$\mathfrak d'_{(\mu, m)}  = 0$.    On the other hand, if $  r < k \leq r+s$, we have (writing $a$ for the length of the first column of $\mu(1)$ and $b$ for the length of the first column of $\mu(2)$)
$$
\mathfrak d'_{(\mu, m)}  =  y_{(\mu, m)}   \beta'_{a, b} =  y_\mu  \beta'_{a, b}  e_{r + k}\power m.
$$
This is an element of $B_{r, s}(\Z; N)$ of corank greater than $m$, and hence an element of 
\break  $B_{r, s}(\Z; N)^{\coldom(\mu, m)}$, because of 
\myref{Lemma}{walled Brauer block diagonal transition}.
  On the other hand, we see that   \break $\mathfrak d'_{(\mu, m)} \in 
y_{(\mu, m)}  B_{r, s}(\Z; N)$.  Thus
\begin{equation} \label{walled 3}
\mathfrak d'_{(\mu, m)} \in  y_{(\mu, m)}  B_{r, s}(\Z; N) \cap B_{r, s}(\Z; N)^{\coldom(\mu, m)}.
\end{equation}
Equations  \eqref{walled 1}, \eqref{walled 2} and \eqref{walled 3}    show that  axiom  \eqref{quotient axiom 2} holds.  Axiom  \eqref{quotient axiom 3}  comes from:

\begin{lem}  \label{dimension of centralizer algebra on mixed tensors}
 The dimension of $A_{r, s}(\C) = \Theta(B_{r, s}( \C; N))$ is 
 $$
 \sum_{(\lambda, l) \in \widehat B_{r,s,  \perm}\mixed}   (\sharp \Std_{\varepsilon, \perm}\mixed(\lambda, l))^2.
$$
\end{lem}
\begin{proof}    This follows by \cite[Corollary 4.7]{MR899903} where the author  calculates the
 direct sum decomposition of mixed tensor space  as a  ${{\rm GL}}_n(\C)$-module  in terms of the combinatorics of ``up-down tableaux'' \cite[Definition 4.4]{MR899903}.
\end{proof}

\begin{thm}  
The algebra  $A_{r, s}\mixed(\Z)$ is a cellular algebra over $\Z$ with basis 
$$\mathbb A_{r,s}\mixed = \leftbrace \Theta(\widetilde y_{\mfs \mft}\power{\lambda, l})\mid   
(\lambda, l) \in \widehat B\mixed_{r, s,  \perm} \text{ and }  \mfs, \mft \in \Std\mixed_{\varepsilon,  \perm} (\lambda, l)
 \rightbrace 
$$
with respect to 
the involution $*$  determined by $E_{a,b}^* = E_{a,b}$, $(S_i')^* = S_i'$, and $S_j^* = S_j$ 
and 
the partially ordered set $(\widehat B\mixed_{\varepsilon,   \perm},\coldomeq)$ 
 The ideal   $\ker(\Theta) \subseteq B_{r,s}(\ZZ; N)$ has  $\ZZ$--basis 
$$\kappa_r = \leftbrace  \widetilde y_{\mfs \mft}\power{\lambda, l}  \mid   
(\lambda, l) \in \widehat B_{r, s} \text{ and   $\mfs$  or  $\mft$  is not $N$--permissible}
 \rightbrace.
$$
Moreover,    for $r +s > N$,  $\ker(\Theta)$ is the ideal generated by the set 
    $$ \{\mathfrak{d}_{a, b} \mid a+b = N+1  \},
 $$
 and for  $r + s \le N$, $\ker(\Theta) = 0$.

 \end{thm}

\begin{proof}The construction of the cellular basis  of  $A\mixed_{r,s}(\Z)$ and of the basis of 
$\ker(\Theta)$  follows immediately from   
\myref{Theorem}{good paths basis theorem},  
 since \eqref{quotient axiom 1}--\eqref{quotient axiom 3} have been verified.

Now, if $r + s \leq N$ then  $\ker(\Theta)=0$ since all paths on $\widehat B_\varepsilon$ of length $\le N$ are $N$--permissible. 
For $r + s > N$,  the kernel is is the ideal generated by all the $\mathfrak{d}_{(\mu, m)}$ such that $(\mu,m)$ is a   marginal point in $\widehat B_{\varepsilon_k}$ for some $0 < k \le r$,  using  
\myref{Theorem}{good paths basis theorem}.
But the marginal points are are all of the form $(\mu,m)$ for some $\mu$ with $\mu(1)'_1 + \mu(2)'_1 = N+1$.   Write $a = \mu(1)'_1 $ and $b = \mu(2)'_1$.   
By \eqref{d lambda walled case}  and \eqref{d lambda l walled case}, 
$$
\begin{aligned}
\mathfrak{d}_{(\mu, m)} &=  \mathfrak{d}_{\mu} e_{k-1} \power m  
 = \jj( y_{\lambda(1)_{>1}} ) \otimes \mathfrak d_{a, b }  \otimes  y_{\lambda(2)_{>1}} e_{k-1} \power m, 
\end{aligned}
$$
and so the result follows.
\end{proof}

\begin{rmk}
As in \myref{Remark}{7point8},
 our construction provides an integral form of the simple $B_{r,s}(\C;N)$-modules labelled by permissible points of $\widehat B_{r, s}$. 
\end{rmk}

\subsection{Fields of positive characteristic} Now, we let $\Bbbk$ denote a field of  arbitrary characteristic.    
 We let $B_{r, s}(\Bbbk; N)=  B_{r, s}(\ZZ; N)  \otimes_\ZZ\Bbbk$. This algebra is a cellular algebra 
  with  basis 
  $$\mathbb B_{r, s}(\Bbbk) = \leftbrace \widetilde y_{\mfs \mft}\power{\lambda, l} 
  \otimes_\ZZ 1_\Bbbk \mid   
(\lambda, l) \in \widehat B_{r, s} \text{ and }  \mfs, \mft \in \Std_{\varepsilon } (\lambda, l)
 \rightbrace 
 $$  
 We fix a basis $\{v_i\}$ of $V$ and the dual basis $\{v^*_i\}$ of $V^*$. 
 We let $V_\ZZ$ denote the $\ZZ$-submodule of $V$ with 
 with  basis $\{v_i\}$ , and
 we let    $V_\Bbbk= V_\ZZ \otimes_\ZZ 1_\Bbbk$.   
 We let $V^*_\ZZ$ denote the $\ZZ$-submodule of $V$ with 
 with  basis $\{v^*_i\}$ , and
 we let    $V^*_\Bbbk= V^*_\ZZ \otimes_\ZZ 1_\Bbbk$. 

 For all $r,s\geq 0$,  there  is a homomorphism $\Theta_\Bbbk : B_{r, s}(\Bbbk;  N) \to \End(V_{\Bbbk}^{\otimes r}\otimes (V_{\Bbbk}^\ast)^{\otimes s})$  
 such that $\Theta_{\Bbbk}(a \otimes 1_\Bbbk)(w \otimes 1_\Bbbk) = \Theta(a)(w) \otimes 1_\Bbbk$,  for 
 $a \in  B_{r,s}(\ZZ; N)$ and $w \in V_{\ZZ}^{\otimes r}\otimes (V_{\ZZ}^\ast)^{\otimes s}$.   Write $A\mixed_{r,s}(\Bbbk)$ for the image of 
 $\Theta_\Bbbk$.

  \begin{thm}  \label{field cellular basis 22}
The algebra  $A_{r, s}\mixed(\Bbbk)$ is a cellular algebra over $\Bbbk$ with basis 
$$\mathbb A_{r,s}\mixed = \leftbrace \Theta_\Bbbk(\widetilde y_{\mfs \mft}\power{\lambda, l})\mid   
(\lambda, l) \in \widehat B\mixed_{\varepsilon,  \perm} \text{ and }  \mfs, \mft \in \Std\mixed_{\varepsilon,  \perm} (\lambda, l)
 \rightbrace 
$$
with respect to 
the involution $*$  determined by $E_{a,b}^* = E_{a,b}$, $(S_i')^* = S_i'$, and $S_j^* = S_j$ 
and 
the partially ordered set $(\widehat B\mixed_{r, s,   \perm},\coldomeq)$ 
 The ideal   $\ker(\Theta_\Bbbk) \subseteq B_{r,s}(\Bbbk; N)$ has  $\Bbbk$--basis 
$$\kappa_r = \leftbrace  \widetilde y_{\mfs \mft}\power{\lambda, l}  \mid   
(\lambda, l) \in \widehat B_{r, s} \text{ and   $\mfs$  or  $\mft$  is not $N$--permissible}
 \rightbrace.
$$
Moreover,      for $r +s > N$,  $\ker(\Theta_\Bbbk)$ is the ideal generated by the set 
    $$ \{\mathfrak{d}_{a, b} \mid a+b = N+1  \},
 $$
and for $r \le N$, $\ker(\Theta_\Bbbk) = 0$. 

\end{thm}
\begin{proof}
The proof is essentially the same as that of 
\myref{Theorem}{field cellular basis 1}.     One requires the statement that
for  an arbitrary  field  $\Bbbk$,  the dimension of $\ker(\Theta_\Bbbk)$ and the dimension of 
$A_{r,s}\mixed(\Bbbk)$ are independent of the field (and of the characteristic);  this comes from \cite{MR3181742}.
\end{proof}

\begin{rmk}  As in the orthogonal and symplectic cases, the centralizer algebra  $\Theta_{\Bbbk}(B_{r, s}(\Bbbk; N))$  acting on  $V_{\Bbbk}^{\otimes r}\otimes (V_{\Bbbk}^\ast)^{\otimes s}$ is a specialization of the integral centralizer algebra 
$\Theta(B_{r, s}(\ZZ; N))$  acting on  $V^{\otimes r}\otimes (V^\ast)^{\otimes s}$. 
\end{rmk}

\subsection{Seminormal bases and seminormal representations}
\label{subsection wba seminormal}

Next we wish to verify that the setting of \myref{Section}{subsection seminormal quotient} applies to the walled Brauer algebras, their specializations   $B_{r, s}(\ZZ; N) \subseteq  B_{r, s}(\Q; N)$, and the quotients of these specializations acting on mixed tensor space.  First we make more explicit part of the setting of \myref{Section}{subsection setting quotient towers}.  We have $R = \ZZ[\bolddelta]$, the generic ground ring, and the quotient map $\pi : \ZZ[\bolddelta] \to \Z$ determined by $\bolddelta \mapsto N$.  The kernel of this map is the prime ideal $\mathfrak p = (\bolddelta - N)$.  The subring  of evaluable elements in $\FF = \Q(\deltabold)$ is $R_{\mathfrak p}$.    The subring of evaluable elements in $B_{r, s}(\FF; \bolddelta)$  is  $B_{r, s}(R_{\mathfrak p}; \bolddelta)$;  c.f.  \myref{Remark}{remark on evaluable elements}.   
 We fix  the lattice path $\varepsilon$ as in \eqref{standard lattice path}, and the corresponding sequence of walled Brauer algebras
$(B_{\varepsilon_i})_{1 \le i \le r+s}$, with branching graph $\widehat B_\varepsilon$.  We have verified in \myref{Section}{subsection: cellularity walled Brauer}  that this is a sequence of diagram algebras satisfying  properties
\eqref{diagram 1}--\eqref{diagram 6}   of 
\myref{Section}{subsection diagram algebras}.     We have verified in \myref{Section}{subsection wba quotient axioms}  that the quotient axioms \eqref{quotient axiom 1}--\eqref{quotient axiom 3} are satisfied by the maps $\Theta$ from $B_{\varepsilon_i}(\ZZ; N)$  to endomorphisms of mixed tensor space.  We have defined JM elements  $L_1, \dots, L_{r+s}$  for the tower $(B_{\varepsilon_i})_{1 \le i \le r+s}$
in \myref{Section}{subsection JM elements WBA}, 
 and have verified that they are an additive  family  JM elements in the sense discussed in \myref{Section}{subsection JM elements quotients}. We  have checked that the JM elements satisfy the  the separation condition; see \myref{Remark}{WBA JM separation condition}.  We know that quotient algebras $\Theta(B_{\varepsilon_i}(\Q; N))$, which are the centralizer algebras of  the general linear group acting on mixed tensor space, are split semisimple.  So it remains only to check that condition 
\eqref{assumption SN} of \myref{Section}{subsection seminormal quotient} holds, with  
\myref{Definition}{defn permissibility for WBA} providing the appropriate permissibilty condition.  

\begin{lem}  Let $\stt$ be an $N$--permissible path in $\Std_{\varepsilon, k}$.  
Then $F_t$ is evaluable.
\end{lem}

\begin{proof}  
 We apply  \myref{Lemma}{suff cond for SN} and \myref{Remark}{suff cond for SN 1}.
 Let $\mft, \mfs \in \Std_{k+1}$ for some $k$ with $\mfs' = \mft'$ and  with
  at least one of $\mfs,  
 \mft$   $N$--permissible.      We have to show that
  $\kappa_\stt(k+1)-\kappa_\sts(k+1) \not\equiv 0  \mod (\bolddelta - N)$.
  In order to reach a contradiction, assume  $\kappa_\stt(k+1)-\kappa_\sts(k+1) \equiv 0  \mod (\bolddelta - N)$.
  This can only happen if one of the two edges $\stt(k) \to \stt(k+1)$  and $\stt(k) \to \sts(k+1)$ involves adding a node to $\stt(k)$ and the other involves removing a node.  It follows in particular that $k \ge r$, since paths up to level $r$ involve only the addition of nodes. 
Write $\stt(k) = ((\lambda(1), \lambda(2)), l)$ and write $a$ for the length of the first column of $\lambda(1)$ and $b$ for the length of the first column of $\lambda(2)$.    Since $\stt' = \sts'$ is $N$--permissible, $a + b \le N$.   Assume without loss of generality that 
$\sts(k+1) =  ((\lambda(1) \setminus \{\alpha\}, \lambda(2)), l+1)$, for some removable node $\alpha$  of $\lambda(1)$, and
$\stt(k+1) =  ((\lambda(1), \lambda(2) \cup \{\beta\}), l)$  for some addable node $\beta$  of $\lambda(2)$.
Our assumption  is then $c(\beta) + c(\alpha) = -N$.   But $c(\beta) \ge   -b $ and $c(\alpha) \ge 1-a$.  Hence $c(\alpha) + c(\beta) \ge 1 - a - b \ge 1 - N$, a contradiction.
\end{proof}

Now that we have verified condition \eqref{assumption SN}, it follows that all the conclusions regarding seminormal bases and seminormal representations from \myref{Section}{subsection seminormal quotient} are valid for the quotients of the walled Brauer algebras acting on mixed tensor space.

\section{Diagrammatic minors and Pfaffians} \label{appendix: diagrammatic minors}
This appendix provides proofs of \myref{Propositions}{symplectic fund theorem 1} and \myrefnospace{}{orthogonal fund theorem 1}.
Let $V$ be a finite dimensional vector space over a field $\Bbbk$ with a non-degenerate symplectic or orthogonal form $[\ ,\ ]$.   The form induces a non-degenerate bilinear form on $V\tensor r$ for each $r$, and thus 
isomorphisms $\eta : V^{\otimes r} \to (V^*)^{\otimes r}$  and $A:  V^{\otimes 2r} \to \End(V^{\otimes r})$, as described in \myref{Section}{section tensor space}.   It follows from the definitions that
$\eta\circ A\inv(T)(x \otimes y) = [y, T(x)]$ for $T \in \End(V\tensor r)$ and $x, y \in V \tensor r$.

\subsection{The symplectic case}  \label{appendix symplectic} Take $V$ to be $2N$ dimensional over $\Bbbk$ with a symplectic form $\langle \ , \ \rangle$.    
Let $\Phi : B_r(\Bbbk; -2N) \to \End(V^{\otimes r})$ be the Brauer homomorphism 
determined by $e_i \mapsto E_i$ and $s_i \mapsto -S_i$, as in 
\myref{Section}{section Brauer symplectic}.   
We will  determine  $\Phi$ and   $\eta\circ A\inv \circ \Phi$
explicitly on the basis of Brauer diagrams.    

In this context we adopt an alternative labeling of the vertices of an $r$--strand Brauer diagram:  the top vertices are labeled by $1, 2, \dots r$ from left to right and the bottom vertices by $r+1, \break r+2, \dots, 2r$ from left to right.  A Brauer diagram $D$ is determined by its set of edges, each given as an unordered pair of vertices:
$$
D = \{ \{i_1, j_1\}, \{i_2, j_2\}, \dots, \{i_r, j_r\} \}.
$$
 The symmetric group $\mathfrak S_{2r}$  acts transitively (on the right)  on Brauer diagrams  by acting on the set of vertices.  The stabilizer  $H$ of the identity diagram  $D_0$ is isomorphic to $(\ZZ/2\ZZ)^r \rtimes \mathfrak S_r$.  For each Brauer diagram $D$, 
 $D$ can be written uniquely as
 $$
D = \{ \{h_1, k_1\}, \{h_2, k_2\}, \dots, \{h_r, k_r\} \},
$$
where $h_i < k_i$ for all $i$ and $h_1 <h_2 < \cdots < h_r$.   
Define $\sigma_D$  by $i \, \sigma_D = h_i$ and $(i+r)\,  \sigma_D = k_i$ for $1 \le i \le r$, so   $D_0 \ \sigma_D = D$.

Any Brauer diagram $D$  can be drawn so that each pair of strands crosses at most once and no strand crosses itself; the {\sf  length} $\ell(D)$ is the number of crossings in such a representative of $D$.    Recall that the  corank of a Brauer diagram  is  $1/2$ the number of horizontal strands.

\begin{lem} \label{lemma sign of Brauer diagram}
 If $D$ is a Brauer diagram with corank $s$,  then $(-1)^{s  + \ell(D)} = \sgn(\sigma_D)$. 
\end{lem}

We now compute $\Phi(D)$ for any Brauer diagram $D$.
If $\pi$ is a permutation diagram, then
\begin{equation} \label{symplectic action 1}
(x_1 \otimes \cdots \otimes x_r)  \Phi(\pi)   = \sgn(\pi) (x_{1\cdot \pi \inv} \otimes  \cdots \otimes  x_{{r\cdot \pi \inv}}).
\end{equation}    Moreover, 
\begin{equation} \label{symplectic action 2}
\begin{aligned}
(x_1 \otimes & \cdots \otimes x_r)  \Phi(e_1 e_3 \cdots e_{2s -1})  =   \\
&
\langle x_1, x_2 \rangle \langle x_3, x_4 \rangle \cdots \langle x_{2s -1}, x_{2s} \rangle    \ 
(\omega \otimes \cdots \otimes \omega) \otimes (x_{2s +1} \otimes \cdots \otimes x_{r}).
\end{aligned}
\end{equation}
Let $D$ be a Brauer diagram of corank $s$.  Then $D$ can be 
factored as 
\begin{equation} \label{symplectic action 3}
D = \pi_1\inv (e_1 e_3 \cdots e_{2s -1}) \pi_2,
\end{equation}
$\pi_1, \pi_2 \in \mathfrak S_r$ satisfy
$(2j -1) \cdot \pi_i < 2j \cdot \pi_i$ for $1 \le j \le s$,  and  $\pi_1$ also satisfies
$(2s +1) \cdot \pi_1 < (2s +2) \cdot \pi_1< \cdots < r \cdot \pi_1$.   This factorization is not unique, but the conditions on $\pi_1, \pi_2$ imply
 \begin{equation} \label{symplectic action 4}
 (-1)^{\ell(D)} = \sgn(\pi_1) \sgn(\pi_2).
 \end{equation} 
The homomorphism property of $\Phi$ together with    \eqref{symplectic action 1} --   \eqref{symplectic action 4}  determines $\Phi(D)$.

To write the answer, we use the following notation:  let $t = r - 2s$.  Let 
$\{1, 2, \cdots, r \} = \{i_1, \dots, i_s\} \cup \{j_1, \dots, j_s\} \cup \{k_1, \dots, k_t\}$,  with
$i_l < j_l$,  for $1 \le l \le s$.  Let $x_1, \dots, x_t$ be any elements of $V$.   Then we 
write 
$$\bigotimes_l \omega(i_l, j_l)  \otimes  \bigotimes_m x_m(k_m)$$
for the tensor that has a copy of $\omega$ in each of the pairs of tensor positions $(i_1, j_1)$, \dots $(i_s, j_s)$,  and the vectors $x_1, \dots x_t$ in the remaining $t$ tensor positions.  For example, with 
$\{v_i\}$  and $\{v_i^*\}$   dual bases of $V$,   
$$
\omega(1, 4) \otimes \omega(2, 5) \otimes x_1(3) \otimes x_2(6) = 
\sum_{i, j}  v_i^* \otimes v_j^* \otimes x_1 \otimes v_i \otimes v_j \otimes x_2.
$$

\def\top{{\rm top}}
\def\bot{{\rm bot}}
\def\vert{{\rm vert}}
\def\rank{{\rm rank}}
\def\corank{{\rm corank}}
 Write $(i, j) \in D$ if $i <j$ and $\{i, j\}$ is a strand of $D$.
Write $(i, j) \in \top(D)$, $(i, j) \in \bot(D)$, or $(i, j) \in \vert(D)$  if $i < j$ and $\{i, j\}$ is respectively a horizontal strand at the top of $D$,  a horizontal strand at the bottom of $D$, or a vertical strand of $D$.  
Then we have
\begin{equation} \label{symplectic action 5}
\begin{aligned}
(x_1 \otimes  & \cdots \otimes x_r)  \Phi(D)  = \\ 
&  (-1)^{\ell(D)} \ 
\smashoperator[lr]{\prod_{(i, j) \in \top(D)}} \langle x_i, x_j \rangle
\smashr{\bigotimes_{(r+ i, r+ j) \in \bot(D)}} \  \omega(i, j) \ \  \otimes 
\smashr{\bigotimes_{(i,r+  j) \in \vert(D)} }\ x_{i}(j).
\end{aligned}
\end{equation}

Next we compute $\eta \circ A\inv \circ\Phi(D)$, for $D$ a Brauer diagram of corank $s$.
\begin{equation}
\begin{aligned}
&(\eta  \circ A\inv \circ \Phi(D))((x_1\otimes \cdots \otimes x_r)\otimes(x_{r+1} \otimes \cdots \otimes x_{2r})) =  \\
&\langle ( x_{r+1} \otimes \cdots \otimes x_{2r}), (x_1\otimes \cdots \otimes x_r)\Phi(D) \rangle = \\
& (-1)^{\ell(D)} 
\smashr{\prod_{(i, j) \in \top(D)}} \  \langle x_i, x_j \rangle
\smashr{\prod_{( i,  j) \in \bot(D)}}\ \langle x_j, x_i \rangle
\smashr{\prod_{(i,  j) \in \vert(D)}}\ \langle x_j, x_i \rangle = \\
& (-1)^{\ell(D)} (-1)^{s} (-1)^{{\rank}(D)}
\smashr{ \prod_{(i, j) \in D}} \  \langle x_i, x_j \rangle = \\
& (-1)^r  \sgn(\sigma_D) 
\smashr{ \prod_{(i, j) \in D}} \  \langle x_i, x_j \rangle,
\end{aligned}
\end{equation}
where in the last line, we have used that $2s + {\rank}(D) = r$ and \myref{Lemma}{lemma sign of Brauer diagram}.

\def\Pf{{\rm Pf}}
For the following result see    \cite[Section 3.6]{MR2265844}. 
\begin{lem}\label{Pfaffian}
If $(a_{i, j})$ is a skew-symmetric $2r$--by--$2r$ matrix, then the Pfaffian of $(a_{i, j})$ is
$$ \Pf((a_{i, j})) = \sum_D  \sgn(\sigma_D) \prod_{(i, j) \in D}  a_{i, j},$$
where the sum is over $r$-strand Brauer diagrams. 
\end{lem}

  We defined $\mathfrak b_r$ to be the sum of all Brauer diagrams with $r$ strands 
 (\myref{Definition}{defn of diagrammatic Pfaffian}),
 so we have
 $$
 (\eta\circ  A\inv \circ \Phi(\mathfrak b_r))(x_1 \otimes x_2 \otimes \cdots \otimes x_{2r})  = \pm 
 \Pf((\langle x_i, x_j \rangle )).
 $$
 Because of this, Gavarini called $\mathfrak b_r$ a ``diagrammatic Pfaffian",  \cite[Definition 3.4]{MR1670662}.  
 If we take $r = N+1$, then for any choice of $x_1, \dots, x_{2N + 2}$, the  matrix $(\langle x_i, x_j \rangle)$ is singular, so the Pfaffian of this matrix is zero.  Thus $\eta \circ A\inv \circ \Phi(\mathfrak b_{N+1}) = 0$, or equivalently $\mathfrak b_{N+1} \in  \ker(\Phi)$.   This proves \myref{Proposition}{symplectic fund theorem 1}.

\subsection{The orthogonal case}  \label{appendix orthogonal}
Let $V$  be an $N$ dimensional vector space  over $\Bbbk$ with a non-degenerate symmetric bilinear  form $(\ , \ )$.   
Let $\Psi : B_r(\Bbbk; N) \to \End(V^{\otimes r})$ be the Brauer homomorphism 
determined by $e_i \mapsto E_i$ and $s_i \mapsto S_i$, as in 
\myref{Section}{section Brauer orthogonal}.

By a similar, but easier, computation as in the symplectic case, we can explicitly determine $\eta\circ A\inv \circ \Psi$ on the basis of Brauer diagrams.  The answer is 
$$(\eta\circ A\inv\circ \Psi(D))(w_1 \otimes \cdots \otimes w_{2r}) =  \prod_{(i, j) \in D}  ( w_i, w_j).$$

We now fix $a \ge b \ge 0$ with $a + b = r$, and restrict attention to $(a, b)$-walled Brauer diagrams $D$. 
There is a bijection between $(a, b)$-walled Brauer diagrams $D$ and permutations $\pi \in \mathfrak S_r$ by exchanging top and bottom vertices to the right of the wall.  
In \myref{Definition}{sign of walled Brauer algebra}, we defined the sign of an $(a, b)$-walled Brauer diagram $D$  to be the sign of the corresponding permutation.    We record the following lemma in order to make the connection with the signs of Brauer diagrams appearing in 
\myref{Appendix}{appendix symplectic}:

\begin{lem}  Let $a + b = r$, and let $D$ be an $(a, b)$-walled Brauer diagram of corank $s$.  Then
$$
\sgn(D) = (-1)^{s + \el(D)} =  \sgn(\sigma_D).
$$
\end{lem}

Having fixed $(a, b)$,  for any sequence of $2r$ vectors $w_1, \dots , w_{2r}$ in $V$,  let
$(x_1, \dots, x_r) = (w_1, \dots, w_a,  w_{r + a + 1}, \dots, w_{2r})$ and
$(y_1, \dots, y_r) = (w_{r+1}, \dots, w_{r+a}, w_{a + 1}, \dots, w_r)$.   Then for an $(a, b)$-walled Brauer diagram $D$ and the corresponding permutation $\pi$, we have
$$
\prod_{(i, j) \in D}  ( w_i, w_j) =  \prod_{i = 1}^r  (x_i,  y_{i \cdot \pi}). 
$$
We have defined $\mathfrak d_{a,b} = \sum_D \sgn(D)\  D$ in 
\myref{Definition}{diagrammatic Pfaffians 1},
where the sum is over all    $(a, b)$-walled Brauer diagrams.   Hence we have
$$
\begin{aligned}
(\eta\circ A\inv \circ \Psi(\mathfrak d_{a, b} ))(w_1 \otimes \cdots \otimes w_{2r}) &=  \sum_D  \sgn(D)  \prod_{(i, j) \in D}  ( w_i, w_j)  \\
&=    \sum_\pi  \sgn(\pi)  \prod_{i = 1}^r  (x_i,  y_{i \cdot \pi}) = \det((x_i, y_j)).  
\end{aligned}
$$
Because of this, Gavarini called $\mathfrak d_{a, b}$ a ``diagrammatic minor",  \cite[Definition 3.4]{MR1670662}.  

If we now take $r = N+1$, then $(x_1, \dots, x_r)$ is linearly dependent and the matrix $(x_i, y_j)$ is singular.  Hence for any choice of $(a, b)$ such that $a + b = N+1$,  we have $\eta\circ A\inv \circ \Psi(\mathfrak d_{a, b} ) = 0$, or equivalently,
$\mathfrak d_{a, b} \in  \ker(\Psi)$.    This proves \myref{Proposition}{orthogonal fund theorem 1}.

\subsection{The walled Brauer algebra} \label{walled Brauer fund thm appendix}
Let $V$ be an $N$--dimensional vectors space over a field $\Bbbk$, and let 
$A : V^* \otimes V \to \End(V)$ be the linear isomorphism determined by $A(v^* \otimes v)(w) = v^*(w) v$.
For any basis $\{v_i\}$ of $V$ and the dual basis $\{v^*_i\}$ ov $V^*$,  $A\inv(\id_V) = \sum_i v^*_i \otimes v_i$.  Write $\omega^* =  \sum_i v^*_i \otimes v$ and $\omega = \sum_i v_i \otimes v^*_i$. 
Let $E \in \End(V \otimes V^*)$  be determined by $E(v \otimes v^*) = v^*(v) \omega$.  

Fix $r, s \ge 0$ and set $V^{\otimes (r, s)} = V^{\otimes r} \otimes (V^*)^{\otimes s}$.   As in 
\myref{Appendix}{section Brauer walled tensor},
 there is a homomorphism $\Theta$ from the walled Brauer algebra $B_{r, s}(\Bbbk; N)$ to $\End(V^{\otimes (r, s)})$, which sends permutations in $\mathfrak S_r \times \mathfrak S_s$ to place permutations of $V^{\otimes (r, s)}$  and sends the generators $e_{a, b}$ to $E_{a, b}$, which is $E$ acting in the $-a$ and $b$ tensor positions.

We write $A$ also for the linear isomorphism $A :  (V^{\otimes (r, s)})^* \otimes V^{\otimes (r, s)} \to \End(V^{\otimes (r, s)})$, and we set 
$$
\tau = A\inv \circ \Theta : B_{r, s}(\Bbbk; N) \to  (V^{\otimes (r, s)})^* \otimes V^{\otimes (r, s)}  = 
(V^*)\tensor r \otimes V\tensor s \otimes V\tensor r \otimes (V^*)\tensor s.
$$
We will determine $\tau$ explicitly.  For purposes of this discussion, label the vertices of an $(r, s)$--walled Brauer diagram as follows:  the top vertices are labelled from left to right by 
$$1', 2', \dots, r'; (r+1), (r+2), \dots, (r+s),$$ and the bottom vertices by
$$1, 2, \dots, r; (r+1)', (r+2)', \dots, (r+s)'.$$  Each strand of a walled Brauer diagram connects a vertex $j'$ to a vertex $i$.   Now the result is $\tau(D) = \bigotimes_{(i, j') \in D}  \omega(i, j')$.  We regard 
$\tau(D)$ as a linear functional on $V\tensor{(r,s)} \otimes (V\tensor{(r,s)})^*$.   Write
$$
\bm x \otimes \bm y^* = 
(x_{1 } \otimes \cdots \otimes x_{r }) \otimes 
(y^*_{r+1} \otimes \cdots \otimes y^*_{r+s}) \otimes
(y^*_{1} \otimes \cdots \otimes y^*_{r}) \otimes
(x_{r+1} \otimes \cdots \otimes x_{r+s }).
$$
Then, $\tau(D)(\bm x \otimes \bm y) = \prod_{(i, j') \in D}  y^*_i(x_{j})$.  Now we use the correspondence between $(r, s)$--walled Brauer diagrams and permutations.  If $\pi \in \mathfrak S_{r+s}$ is the permutation corresponding to $D$,  then  $\tau(D)(\bm x \otimes \bm y) = \prod_{j=1}^{r+s}  y^*_{j\cdot \pi}(x_{j})$.

Now we recall that $\mathfrak d_{r, s} = \sum_D \sgn(D) \ D$,  with the sum over all $(r, s)$ walled Brauer diagrams, and that the sign of $D$ is the same as the sign of the corresponding permutation.  Thus
$$
\tau(\mathfrak d_{r, s})(\bm x \otimes \bm y^*)  = \sum_D \sgn(D)  \prod_{(i, j') \in D}   y^*_i(x_{j}) = 
\sum_\pi  \sgn(\pi) \prod_{j=1}^{r+s}  y^*_{j\cdot \pi}(x_{j}) = \det((y^*_i(x_{j}))).
$$
Finally, we observe that if $r+s = N+1$, then the matrix $(y^*_i(x_{j'}))$ is singular, for any choice of $\bm x \otimes \bm y^*$, and therefore $\mathfrak d_{r, s} \in \ker(\tau) = \ker(\Theta)$.  So we have proved:

\begin{lem} If $r + s = N+1$, then $\mathfrak d_{r, s}  \in \ker(\Theta)$.  
\end{lem}

 
 }  

\bibliographystyle{amsplain}

\bibliography{seminormalformbrauer}

\end{document}